\documentclass[a4paper,11pt]{article}
\usepackage{geometry}                
\usepackage[active]{srcltx}
\usepackage{hyperref}

\usepackage{graphicx}
\usepackage{amssymb}
\usepackage{epstopdf}
\usepackage{amsmath}
\usepackage{amsthm}
\usepackage{amssymb}
\usepackage{latexsym}
\usepackage{mathtools}
\usepackage{mathrsfs}
\usepackage{graphics}
\usepackage{latexsym}
\usepackage{psfrag}
\usepackage{import}
\usepackage{verbatim}
\usepackage{enumerate}
\usepackage{enumitem}
\usepackage{graphicx}
\usepackage[usenames]{color}

\newcommand{\dom}{{\textrm{Dom\,}}}
\newcommand{\R}{\mathbb{R}}
\newcommand{\N}{\mathbb{N}}

\newcommand{\cH}{\mathcal{H}}
\newcommand{\B}{\mathcal{B}}
\newcommand{\Q}{\mathbb{Q}}

\newcommand{\sgn}{\text{\rm sgn}}

\newcommand{\dist}{\text{\rm dist}}
\newcommand{\supp}{\text{\rm supp}}

\newcommand{\Lip}{\mathrm{Lip}}
\newcommand{\gr}{\textrm{graph}}

\newcommand{\ve}{\varepsilon}

\newcommand{\erre}{\mathbb{R}}

\newcommand{\f}{\varphi}

\newcommand{\T}{\mathcal{T}}

\renewcommand{\L}{\mathcal{L}}

\newcommand{\RCD}{\mathsf{RCD}}

\newcommand{\CD}{\mathsf{CD}}

\newcommand{\Geo}{{\rm Geo}}
\newcommand{\MCP}{\mathsf{MCP}}

\newcommand{\inte}{{\rm int}}

\newcommand{\Ric}{{\rm Ric}}

\newcommand{\abs}[1]{\left\vert#1\right\vert}
\newcommand{\set}[1]{\left\{#1\right\}}

\newcommand{\Real}{\mathbb{R}}

\newcommand{\eps}{\varepsilon}

\renewcommand{\L}{\mathcal{L}}

\renewcommand{\P}{\mathbb P}

\renewcommand{\P}{\mathcal{P}}

\renewcommand{\H}{\mathcal{H}}

\newcommand{\mm}{\mathfrak m}
\newcommand{\qq}{\mathfrak q}

\newcommand{\ee}{{\rm e}}
\newcommand{\QQ}{\mathfrak Q}

\newcommand{\sfd}{\mathsf d}
\newcommand{\sfb}{\mathsf b}
\newcommand{\Opt}{\mathrm{OptGeo}}

\newcommand{\LIP}{\mathrm{LIP}}

\theoremstyle{plain}
\newtheorem{lemma}{Lemma}[section]
\newtheorem{theorem}[lemma]{Theorem}

\newtheorem{proposition}[lemma]{Proposition}
\newtheorem{corollary}[lemma]{Corollary}
\newtheorem*{theorem*}{Theorem}
\newtheorem*{maintheorem*}{Main Theorem}

\theoremstyle{definition}

\newtheorem{definition}[lemma]{Definition}
\newtheorem*{definition*}{Definition}
\newtheorem{remark}[lemma]{Remark}

\numberwithin{equation}{section}

\title{New formulas for the  Laplacian of distance functions \\ and applications}
\author{Fabio Cavalletti \thanks{F. Cavalletti: Mathematics Area, SISSA, Trieste (Italy), email:cavallet@sissa.it.}  and Andrea Mondino\thanks{A. Mondino: Mathematics Institute,  University of Warwick  (UK), email: A.Mondino@warwick.ac.uk.}}
\date{}                                           

\begin{document}
\maketitle

\abstract{
The goal of the paper is to prove an exact representation formula for the Laplacian of the  distance (and more generally for an arbitrary $1$-Lipschitz function)  in the framework of metric measure spaces satisfying Ricci curvature lower bounds in a synthetic sense (more precisely in essentially non-branching $\MCP(K,N)$-spaces). Such a representation formula makes apparent the classical upper bounds together with lower bounds and  a precise description of the singular part.   The exact representation formula for the Laplacian of a general 1-Lipschitz function holds also (and seems new) in a general complete Riemannian manifold.

We apply these results to prove the  equivalence of $\CD(K,N)$ and a dimensional Bochner inequality on signed distance functions.
Moreover we obtain a measure-theoretic Splitting Theorem 
for infinitesimally Hilbertian, essentially non-branching spaces verifying $\MCP(0,N)$.
}

\bibliographystyle{plain}

\tableofcontents


\section{Introduction}

The Laplacian comparison Theorem  for the distance function from a point in a manifold with Ricci curvature bounded from below is one of the most fundamental results in Riemannian geometry. The local version  states that if $(M,g)$ is a smooth Riemannian manifold of dimension $N\geq 2$ satisfying $\Ric_{g}\geq (N-1)g$ then, calling $\sfd_{p}(\cdot):=\sfd(p,\cdot)$ the distance from a point $p\in M$, until the distance function is smooth the next upper bound holds:
\begin{equation}\label{eq:UBLSmoothLoc}
\Delta \sfd_{p} \leq (N-1) \cot \sfd_{p}.
\end{equation}  
Of course here $\Delta$ denotes  the Laplacian (also called Laplace-Beltrami operator) of the Riemannian manifold $(M,g)$ and $\cot$ is the cotangent (for a general lower bound  $\Ric_{g}\geq Kg$ an analogous upper bound holds by replacing the right hand side of \eqref{eq:UBLSmoothLoc} with the suitable (hyperbolic-)trigonometric function).
The result is very classical and can be proved either via Bochner inequality (see for instance \cite[Section 2]{CheegerBook})  or by Jacobi fields computations (see for instance \cite[Chapter 7]{Petersen}). 
\\

It was Calabi \cite{Calabi} who, in 1958, first extended the upper bound \eqref{eq:UBLSmoothLoc} to the whole manifold in the weak sense of barriers. Cheeger-Gromoll \cite{ChGr}, in their celebrated proof of the Splitting Theorem in 1971, then  proved that the upper bound \eqref{eq:UBLSmoothLoc} also holds globally on $M$ in distributional sense (see also \cite[Section 4]{CheegerBook}). Since those classical works, the  Laplacian comparison Theorem has become a fundamental technical tool in the investigation of Riemannian manifolds satisfying Ricci curvature lower bounds (see for instance \cite{CheegerBook, CC96,CC97,CC00a,CC00b,ChGr,Col96a,Col96b,Col97,CN, LY86, Petersen}). 
\\We finally mention that  recently   Mantegazza-Mascellani-Uraltsev \cite{MMU} obtained an exact representation formula for the distributional Hessian (and Laplacian) of the distance function from a point and that Gigli \cite{Gigli12} extended to the non-smooth setting  the upper bound \eqref{eq:UBLSmoothLoc}.   \\
 
The goal of this paper is to sharpen the Laplacian comparison Theorem in several ways. First of all we will give an \emph{exact} representation formula for the Laplacian of a general distance function (and for a general 1-Lipschitz function on its transport set, see later for the details) which describes exactly also the singular part concentrated on the cut locus.
 Such a representation formula will hold on \emph{ every complete } Riemannian manifold, without any curvature assumption. When specialised to Riemannian manifolds with  Ricci curvature bounded below, such an exact representation formula will make apparent not only the celebrated \emph {global upper bound} \eqref{eq:UBLSmoothLoc} but also \emph{a lower bound on the regular part of the Laplacian}. The results will be proved in the much higher generality of (non-necessarily smooth) metric measure spaces   satisfying Ricci curvature lower bounds in a synthetic sense (more precisely, essentially non branching $\MCP(K,N)$-spaces), see the final part of the introduction.  
\\In order to fix the ideas, we start the introduction discussing the smooth setting of Riemannian manifolds.
\\

 Let us introduce some notation in order to state the results.
Given a point $p\in M$, denote by ${\mathcal C}_{p}$ the cut locus of $p$.  The negative gradient flow $g_{t}:M\to M$ of the distance function $\sfd_{p}$ induces a partition $\{X_{\alpha}\}_{\alpha\in Q}$ of $M\setminus(\{p\}\cup {\mathcal C}_{p})$ into minimising geodesics; each $X_{\alpha}$ is called  (transport) ray and  $Q$ is a suitable set of indices. We will denote the initial (resp. final) point of the ray $X_{\alpha}$ as $a(X_{\alpha})$ (resp. $b(X_{\alpha})$); it is not hard to see that   $a(X_{\alpha})\in {\mathcal C}_{p}$ and  $b(X_{\alpha})=p$, for every $\alpha \in Q$. 
Let us stress that in this case the endpoints $a(X_{\alpha}), b(X_{\alpha})$ are not elements of the ray $X_{\alpha}$ (in general, end points may or may not be elements of the ray, depending on the specific case, see also Remark \ref{R:initialpoints}).
Such a partition induces a disintegration (the non-expert reader can think of a kind of ``non-straight Fubini Theorem'') of the Riemannian volume measure $\mm$ into measures $\mm_{\alpha}=h_{\alpha} \cH^{1}\llcorner_{X_{\alpha}}$ concentrated on $X_{\alpha}$:
\begin{equation}\label{eq:disIntro}
\mm= \int_{Q} h_{\alpha} \cH^{1}\llcorner_{X_{\alpha}} \, \qq(d\alpha),  
\end{equation}
where $\qq$ is a suitable probability measure on the set of indices $Q$.  
We refer to Section \ref{Ss:disint} for all the details 
on disintegration formula. Here we only mention that once 
the probability $\qq$ is fixed within a suitable family of probability measures, then the functions $h_{\alpha}$ are uniquely determined.

The fact that $(M,g)$ satisfies $\Ric_{g}\geq (N-1) g$ is inherited by the disintegration as concavity properties of the densities $h_{\alpha}$,   for the details see Section \ref{S:transportset}. For simplicity of notation, we will denote $(\log h_{\alpha})'(x):=\frac{d}{dt}|_{t=0} \log h_{\alpha}(g_{t}(x))$; thanks to the disintegration \eqref{eq:disIntro}  and the (semi-)concavity of $h_{\alpha}$ along $X_{\alpha}$,  the quantity $(\log h_{\alpha})'$ is well defined $\mm$-a.e..

The first main result of the paper is an exact representation formula for the Laplacian of the distance function in non-smooth spaces satisfying synthetic lower bounds on the Ricci curvature (see later in the introduction). In order to fix the ideas, we state it here for smooth Riemannian manifolds.
We denote with $C_{c}(M)$ the space of real valued continuous functions with compact support in $M$ endowed with the final topology and with $(C_{c}(M))'$ its dual space made  of real valued continuous linear functionals on  $C_{c}(M)$.

\begin{theorem} \label{thm:DeltadSmoothGen}
Let $(M,g)$ be a smooth complete  $2\leq N$-dimensional Riemannian manifold.
Fix $p \in M$, consider $\sfd_{p} : = \sfd(p,\cdot)$ and an associated disintegration $\mm=\int_{Q} h_{\alpha} \cH^{1}\llcorner_{X_{\alpha}} \, \qq(d\alpha)$.
\\ Then  $\Delta\sfd_{p}$ is an element of $(C_{c}(M))'$ with the following representation formula: 
\begin{equation}\label{eq:repdeltadpSmooth}
\Delta \sfd_{p} = - (\log h_{\alpha})' \, \mm -   \int_{Q} h_{\alpha} \delta_{a(X_{\alpha})} \,\qq(d\alpha).
\end{equation}
It can be written as sum of the following three Radon measures: 
$$
\Delta \sfd_{p} = \left[\Delta \sfd_{p}\right]_{reg}^{+} - 
\left[\Delta \sfd_{p}\right]_{reg}^{-} + 
\left[\Delta \sfd_{p}\right]_{sing},  
$$
with
$$
 \left[\Delta \sfd_{p}\right]_{reg}^{\pm}
= -[(\log h_{\alpha})' ]^{\pm}\,\mm,
\qquad
\left[\Delta \sfd_{p}\right]_{sing}  =
- \int_{Q} h_{\alpha} \,\delta_{a(X_{\alpha})} \,\qq(d\alpha)\leq 0, 
$$
where $\pm$ stands for the positive and negative part. Here, $\left[\Delta \sfd_{p}\right]_{reg}:= \left[\Delta \sfd_{p}\right]_{reg}^{+} - 
\left[\Delta \sfd_{p}\right]_{reg}^{-}$ is the regular part of $\Delta \sfd_{p}$ (i.e. absolutely continuous with respect to $\mm$), and $\left[\Delta \sfd_{p}\right]_{sing}$ is the singular part.
\\In particular, if $(M,g)$ is compact $\Delta \sfd_{p}$ is a finite signed Borel (and in particular Radon) measure.

Moreover,  if $\Ric_{g}\geq K g$ for some $K\in \R$, the  next comparison results hold true  (for simplicity here we assume $K=N-1$, for the bounds corresponding to a  general $K\in \R$ see \eqref{eq:logh'smoothproof}):
\begin{align}
\Delta \sfd_{p}&\leq (N-1) \,\cot \sfd_{p}  \, \mm ,  \label{eq:DeltadpleqSmooth} \\
\left[\Delta \sfd_{p}\right]_{reg}& =   - (\log h_{\alpha})'  \mm \geq  -  (N-1)\, \cot \sfd_{a(X_{\alpha})} \,   \mm. \label{eq:DeltadpgeqSmooth}
\end{align}
\end{theorem}

\begin{remark}[On the lower bound \eqref{eq:DeltadpgeqSmooth}]
Denote with ${\mathcal C}_{p}:=\{a(X_{\alpha})\}_{\alpha\in Q}$ the cut locus of $p$ and with $g_{t}$ the negative gradient flow of $\sfd_{p}$ at time $t$. More precisely, $g_{t}$ is defined ray by ray as the translation by $t$ in the direction of the negative gradient of $\sfd_{p}$, for $t\in (0, |X_{\alpha}|)$, where $|X_{\alpha}|$ denotes the length of the transport ray $X_{\alpha}$, i.e. $|X_{\alpha}|=\sfd(a(X_{\alpha}),b(X_{\alpha})) = \sfd(a(X_{\alpha}),p)$. Then for every $\ve>0$ there exists $C_{K,N,\ve}>0$ so that:
\begin{align*}
\left[\Delta \sfd_{p}\right]_{reg} &\geq -C_{K,N,\ve} \mm  \quad \text{on } \{x=g_{t}(a_{\alpha})\,:\, t\geq \ve\}\supset \{x \in X\,:\, \sfd(x, {\mathcal C}_{p} )\geq \ve\}.
\end{align*}
Let us stress that such a lower bound depends just on the dimension $N$, on the lower bound $K\in \R$ over the Ricci tensor, and on the distance $\ve>0$ from the cut locus ${\mathcal C}_{p}$, but is independent of the specific manifold $(M,g)$.
\end{remark}

We will prove the next more general statement  for any signed distance function. Let us first give some definition: given a continuous function $v : M \to \R$ so that $\set{v = 0} \neq \emptyset$, the function
\begin{equation}\label{E:levelsetsIntro}
d_{v} : M \to \R, \qquad d_{v}(x) : = \sfd(x, \{ v = 0 \} )\; \sgn(v),
\end{equation}
is called the \emph{signed distance function} (from the zero-level set of $v$). 
With a slight abuse  of notation, we denote with $\sfd$ both the distance between points and the induced distance between sets; more precisely 
$$
 \sfd(x, \{ v = 0 \} ):=\inf \left\{  \sfd(x,y)\,:\, y\in  \{ v = 0 \} \right\}.
$$
Analogously to $\sfd_{p}$, a signed distance function $d_{v}$ induces a partition of $M$ (up to a set of measure zero) into rays $\{X_{\alpha}\}_{\alpha\in Q}$ and a corresponding  disintegration of the Riemannian volume measure $\mm$. The orientation of the rays is analogous. More precisely, if $X_{\alpha}$ is a transport ray associated with $d_{v}$ and $a(X_{\alpha}), b(X_{\alpha})$ are its starting and the final point, then 
$
d_{v}(b(X_{\alpha})) \leq 0, 
d_{v}(a(X_{\alpha})) \geq 0,
$
so that transport rays are oriented from $\{v \geq 0 \}$ towards $\{ v \leq 0\}$.

\begin{theorem}\label{T:Deltadv2Smooth}
 Let $(M,g)$ be a smooth complete $2\leq N$-dimensional Riemannian manifold. 
\\Consider the signed distance function $d_{v}$ 
for some continuous function $v : X\to\R$ and an associated disintegration $\mm=\int_{Q} h_{\alpha} \cH^{1}\llcorner_{X_{\alpha}} \, \qq(d\alpha)$.

Then $\Delta d_{v}^{2}$ is an element of $(C_{c}(M))'$ with the following representation formula: 
\begin{equation}\label{eq:repdeltadpSmooth}
\Delta d_{v}^{2} = 2 (1  - d_{v} (\log h_{\alpha})' )\mm - 2\int_{Q}( h_{\alpha}d_{v})[\delta_{a(X_{\alpha})} - \delta_{b(X_{\alpha})}] \,\qq(d\alpha).
\end{equation}
It can be written as the sum of three Radon measures: 
$$
\Delta d_{v}^{2}  = \left[\Delta d_{v}^{2}\right]_{reg}^{+} - 
\left[\Delta d_{v}^{2}\right]_{reg}^{-} + 
\left[\Delta d_{v}^{2}\right]_{sing},  
$$
with
$$
 \left[\Delta d_{v}^{2}\right]_{reg}^{\pm}
:= 2 (1  - d_{v} (\log h_{\alpha})' )^{\pm}\,\mm,
\quad
\left[\Delta d_{v}^{2}\right]_{sing}  :=
- 2\int_{Q} ( h_{\alpha}d_{v})[\delta_{a(X_{\alpha})} - \delta_{b(X_{\alpha})}] \,\qq(d\alpha) \leq 0\ ,
$$
where $\pm$ stands for the positive and negative part; in particular if, $(M,g)$ is compact, $\Delta d_{v}^{2}$ is a finite signed Borel (and in particular Radon)  measure.

\noindent
Moreover, if $\Ric_{g}\geq K g$ for some $K\in \R$, the next comparison results hold true (for simplicity here we assume $K=N-1$, for the bounds corresponding to a  general $K\in \R$ see \eqref{eq:Deltadv2leq}, \eqref{eq:Deltadv2geq}):
\begin{align}
\left[\Delta d_{v}^{2}\right]^{+}_{reg} \leq & ~ 
2\mm+ 2(N-1)\sfd(\{ v= 0 \},x) \left( \cot\sfd_{b(X_{\alpha})}  \, \mm\llcorner_{\{v\geq 0\}}  +   \cot \sfd_{a(X_{\alpha})} \,  \mm\llcorner_{\{v< 0\}} \right),
\label{eq:Deltadv2leqSmooth}  \\
\left[\Delta d_{v}^{2}\right]^{-}_{reg} \leq & ~ 
2\mm- 2 (N-1) \sfd(\{ v =0 \},\cdot)  \left(\cot \sfd_{a(X_{\alpha})} \, \mm\llcorner_{\{v\geq 0\}}  + \cot \sfd_{b(X_{\alpha})}\,  \mm\llcorner_{\{v< 0\}} \right) \ 
\label{eq:Deltadv2geqSmooth}.
\end{align}
\end{theorem}

We will also present a general statement (Corollary \ref{C:main1Smooth}) valid for any $1$-Lipschitz function $u:M\to \R$, provided the rays of the induced disintegration satisfy a suitable integrability condition (roughly, they should not be too short), obtaining the same representation formula  
together with the two sided estimate we mentioned before.

An interesting feature of Corollary \ref{C:main1Smooth} is that it will hold \emph{for every} 1-Lipschitz function $u:X\to \R$. Let us stress that the 1-Lipschitz assumption is clearly a \emph{first order} condition, with no information on second order derivatives. Nevertheless, Corollary \ref{C:main1Smooth} will imply that in a general complete Riemannian manifold it is possible to deduce some information on the \emph{second derivatives once restricted to a suitable subset}. More precisely, if one considers only the set of points ``saturating the 1-Lipschitz assumption'' then the Laplacian of $u$ is a continuous linear functional on $C_{c}$. We stress that we will obtain an \emph{exact representation formula} of $\Delta u$ (restricted to such a set) which, in case the Ricci curvature of the ambient $N$-manifold is bounded below by $K\in \R$,  will give a \emph{two-sided bound} on the regular part in terms of $K, N$.    We refer to Corollary \ref{C:main1Smooth} for the details. 
\\

Up to now we focused the introduction on the setting of complete Riemannian manifolds (satisfying Ricci curvature lower bounds). However, everything will be proved in the much higher generality of (possibly non-smooth)
essentially non-branching, metric measure spaces $(X,\sfd,\mm)$ satisfying the measure contraction property $\MCP(K,N)$, for some $K\in \R, N\in (1,\infty)$. 
We refer to Subsection \ref{SS:MCPDef} for the detailed definitions; here let us just recall that $\MCP(K,N)$,  introduced independently by Ohta \cite{Ohta1} and Sturm \cite{Sturm06II},  is the weakest among the synthetic conditions of Ricci curvature bounded below by $K$ and dimension bounded above by $N$ for metric measure spaces.  In particular it is strictly weaker than the celebrated curvature dimension condition $\CD(K,N)$ pioneered by Lott-Sturm-Villani \cite{Lott-Villani09,Sturm06I,Sturm06II} and than the (weaker) reduced curvature dimension condition $\CD^{*}(K,N)$ \cite{BS10}.  The essential non-branching condition, introduced by T. Rajala-Sturm \cite{RS2014}, roughly amounts to ask that $W_{2}$-geodesics are concentrated on non-branching geodesics.

\begin{remark}[Notable examples of spaces fitting in the framework of the paper]
The class of essentially non-branching $\MCP(K,N)$ spaces include many remarkable families of spaces, among them:
\begin{itemize}
\item  Smooth  Finsler manifolds where the norm on the tangent spaces is strongly convex, and which satisfy lower Ricci curvature bounds. More precisely we consider a $C^{\infty}$-manifold  $M$, endowed with a function $F:TM\to[0,\infty]$ such that $F|_{TM\setminus \{0\}}$ is $C^{\infty}$ and  for each $p \in M$ it holds that $F_p:=T_pM\to [0,\infty]$ is a  strongly-convex norm, i.e.
$$g^p_{ij}(v):=\frac{\partial^2 (F_p^2)}{\partial v^i \partial v^j}(v) \quad \text{is a positive definite matrix at every } v \in T_pM\setminus\{0\}. $$
Under these conditions, it is known that one can write the  geodesic equations and geodesics do not branch; in other words these spaces are non-branching. 
We also assume $(M,F)$ to be geodesically complete and endowed with a $C^{\infty}$  measure $\mm$ in a such a way that the associated m.m.s. $(X,F,\mm)$ satisfies the  $\MCP(K,N)$ condition,  see \cite{OhtaJLMS, OhSt}. 
\item Sub-Riemannian manifolds. The following are all examples of essentially non-branching $\MCP(K,N)$-spaces:  the $(2n+1)$-dimensional Heisenberg group \cite{Juillet},  any co-rank one Carnot group \cite{Rizzi},   any ideal Carnot group \cite{Rifford}, any generalized H-type Carnot group of rank $k$ and dimension $n$ \cite{BarilariRizzi}.
\item Strong $\CD^{*}(K,N)$ spaces, and in particular $\RCD^{*}(K,N)$ spaces (see below). The class of $\RCD^{*}(K,N)$ spaces includes the following remarkable subclasses:
\begin{itemize}
\item Measured Gromov Hausdorff limits of Riemannian $N$-dimensional manifolds  satisfying Ricci $\geq K$,  see \cite{AGS11b,GMS2013}. 
\item Finite dimensional Alexandrov spaces with curvature bounded from below, see  \cite{Pet}. 
\end{itemize}
\end{itemize}
\end{remark}

In the context of metric measure spaces verifying Ricci curvature lower bounds in a synthetic form, the Laplacian comparison Theorem in its classical form \eqref{eq:UBLSmoothLoc} 
was established by Gigli \cite{Gigli12}. More precisely, \cite{Gigli12} developed a notion of a possibly multivalued Laplacian holding on a general metric measure space $(X,\sfd,\mm)$; in \cite{Gigli12}, a property of the space called \emph{infinitesimal strict convexity} is also introduced, which grants, among other things, uniqueness of the Laplacian. 
Finally in \cite{Gigli12}, assuming infinitesimal strict convexity and  $\CD^{*}(K,N)$,  a sharp upper bound for the Laplacian of a general Kantorovich potential for the $W_{2}$ distance is obtained and, in particular, for $\sfd_{p}^{2}$. The comparison in \cite{Gigli12} is stated for $\CD^{*}(K,N)$ but the same proof, in the case of $\sfd_{p}^{2}$, works assuming the weaker $\MCP(K,N)$.

Our results therefore  extend the ones in \cite{Gigli12} removing the assumption of infinitesimal strict convexity  (hence including  the possibility of a multivalued Laplacian, see Definition \ref{D:Laplace}); moreover we precisely describe the Laplacian of a general signed distance function or a 1-Lipschitz function with sufficiently long transport rays,  obtaining also a lower bound on the regular part and a representation formula for the singular part. 
We stress the fundamental role of the exact representation formulas: it will be the key in our application to Bochner inequality (signed distance functions) and for the Splitting Theorem (general 1-Lipschitz function), see  the discussions  below.

We conclude this part on the related results in the literature 
mentioning that the Laplacian comparison results 
\cite[Theorem 5.14, Corollary 5.15]{Gigli12} 
seem to claim the stronger conclusion that $\Delta \sfd_{p}^{2}$ is a Radon measure in the classical sense (see Definition \ref{D:radonfunctional} and comments shortly afterwards). 
This however seems to not follow from the proof,  when $(X,\sfd)$ is not compact: $\Delta \sfd_{p}^{2}$ is proved to be an element of $(C_{c}(X))'$ so, by Riesz Theorem, it is a difference of positive Radon measures but it may fail to be a Borel measure  (see \cite[Proposition 4.13]{Gigli12} and the application of 
Riesz Theorem in the last part of its proof). 
We will therefore adapt  the definition 
of Laplacian (see Definition \ref{D:Laplace}), 
weakening \cite[Definition 4.4]{Gigli12}. 
With this new definition 
also \cite[Proposition 4.13]{Gigli12} together 
with its applications seem to work.

\bigskip

The second part of the paper is devoted to applications. 
\\
In Section \ref{Sec:CDBE} we will use the representation formula for the Laplacian to show that, under essential non branching, the $\CD(K,N)$ condition is equivalent to a dimensional Bochner inequality on signed distance functions. The Bochner inequality corresponds to an \emph{Eulerian} formulation of Ricci curvature lower bounds while the  $\CD(K,N)$ condition, based on convexity of entropies along $W_{2}$-geodesics of probability measures, correspond to a \emph{Lagrangian} approach. 

It has been a long standing open problem, see for instance the celebrated book of Villani \cite[Open Problem 17.38, Conclusions and Open Problems p. 923]{Villani09}, to show that the Eulerian and the Lagrangian formulations of Ricci curvature lower bounds are equivalent.  Such an equivalence has already been proved to hold true under the additional assumption that  the heat flow ${\rm H}_{t}:L^{2}(X,\mm)\to L^{2}(X,\mm)$ is linear for every $t\geq 0$ (or, equivalently, the Cheeger energy ${\rm Ch}(f):=\int_{X} |\nabla f|^{2}_{w} \mm$ satisfies the parallelogram identity). The class  of $\CD(K,N)$ spaces satisfying such a linearity condition is called $\RCD(K,N)$.  
After its birth in \cite{AGS11b} (see also \cite{AGMR2012}) for $N=\infty$  and further developments for $N<\infty$  (see \cite{AMS2013,EKS2013, Gigli12} and the subsequent \cite{CMi16}), the theory of metric measure spaces satisfying $\RCD(K,N)$ (called $\RCD(K,N)$-spaces for short) has been flourishing in the last years  (for a survey of results, see the Bourbaki seminar \cite{VilB} and the recent ICM-Proceeding \cite{AmbrosioICM}).
\\ The equivalence between $\RCD(K,N)$ and Bochner inequality (properly written in a weak form, called Bakry-\'Emery condition ${\rm BE}(K,N)$) was proved for $N=\infty$ by Ambrosio-Gigli-Savar\'e \cite{AGS11b, AGS12}, and in the finite dimensional case by Erbar-Kuwada-Sturm \cite{EKS2013} and Ambrosio-Mondino-Savar\'e \cite{AMS2013}. 
\\Let us stress that the linearity of the heat flow was a crucial assumption in all of the aforementioned works.

The equivalence between Bochner inequality and $\CD(K,N)$ was proved also in \emph{smooth} Finsler manifolds by Ohta-Sturm \cite{OhSt}. In  \cite{OhSt}  no linearity of the heat flow is assumed, on the other hand the smoothness of the Finsler structure is heavily used in the computations.
In the present paper, in contrast to the aforementioned works, we assume \emph{neither that the heat flow is linear nor that the space is smooth} thus showing that the equivalence between \emph{Lagrangian} and \emph{Eulerian} approach to Ricci curvature lower bounds  holds in the higher generality of  non-smooth   ``possibly Finslerian'' spaces.

The proof of the equivalence seems also to follow rather easily once the representation formula for the Laplacian of signed distance functions is at disposal. 
Here we also crucially use \cite{CMi16} where 
it is shown that a control on the behaviour of 
signed distance functions is sufficient to control the geometry of the space (see the statement: $\CD^{1}(K,N)$
implies $\CD(K,N)$). This also motivates our interest 
on the Laplacian of this family of functions (Theorem \ref{T:d3}).
\\

A second application is a measure-theoretic Splitting Theorem stating, roughly, that an  infinitesimally Hilbertian (i.e. the Cheeger energy satisfies the parallelogram identity), essentially non-branching $\MCP(0,N)$ space containing a line is isomorphic as a \emph{measure space} to a splitting (for the precise statement see Theorem  \ref{thm:split}).
\\For smooth Riemannian manifolds \cite{ChGr}, as well as for Ricci-limits \cite{CC96} and  $\RCD(0,N)$ spaces \cite{GigliSplitting}, the Splitting Theorem has a stronger statement giving an \emph{isometric splitting}. However under the assumptions of Theorem \ref{thm:split} it is not conceivable to expect also a splitting of the metric. Indeed the Heisenberg group ${\mathbb H}^{n}$ is an example of  non-branching  infinitesimally Hilbertian $\MCP(0,N)$ space \cite{Juillet} containing a line, which is homeomorphic and isomorphic as measure space to a splitting (indeed it is homeomorphic to $\R^{n}$ and the measure is exactly the $n$-dimensional Lebesgue measure) but it is not isometric to a splitting.

\subsubsection*{Acknowledgements}   
A. M. acknowledges the support of   the EPSRC First Grant  EP/R004730/1 and of the ERC Starting Grant 802689.   The authors wish to thank the anonymous reviewers for the careful reading and for their comments, which led to an improvement in the exposition.


\section{Prerequisites}

In this Section we review the basic material needed throughout the paper.
The standing assumptions are that $(X, \sfd)$ is a complete, proper and separable metric space endowed with a  positive Radon measure $\mm$ satisfying $\supp(\mm)=X$.  The triple   $(X, \sfd,\mm)$ is said to be a metric measure space, m.m.s. for short.

The properness assumption is motivated by the synthetic Ricci curvature lower bounds we will assume to hold.

\subsection{Essentially non branching,  $\MCP(K,N)$ and  $\CD(K,N)$ metric measure spaces}\label{SS:MCPDef}

We denote by 
$$
\Geo(X) : = \{ \gamma \in C([0,1], X):  \sfd(\gamma_{s},\gamma_{t}) = |s-t| \sfd(\gamma_{0},\gamma_{1}), \text{ for every } s,t \in [0,1] \}
$$
the space of constant speed geodesics. The metric space $(X,\sfd)$ is a \emph{geodesic space} if and only if for each $x,y \in X$ 
there exists $\gamma \in \Geo(X)$ so that $\gamma_{0} =x, \gamma_{1} = y$.
\\Recall that, for complete geodesic spaces, local compactness is equivalent to properness (a metric space is proper if every closed ball is compact).

\medskip

We denote with  $\mathcal{P}(X)$ the  space of all Borel probability measures over $X$ and with  $\mathcal{P}_{2}(X)$ the space of probability measures with finite second moment.
$\mathcal{P}_{2}(X)$ can be  endowed with the $L^{2}$-Kantorovich-Wasserstein distance  $W_{2}$ defined as follows:  for $\mu_0,\mu_1 \in \mathcal{P}_{2}(X)$,  set
\begin{equation}\label{eq:W2def}
  W_2^2(\mu_0,\mu_1) := \inf_{ \pi} \int_{X\times X} \sfd^2(x,y) \, \pi(dxdy),
\end{equation}
where the infimum is taken over all $\pi \in \mathcal{P}(X \times X)$ with $\mu_0$ and $\mu_1$ as the first and the second marginal.
The space $(X,\sfd)$ is geodesic  if and only if the space $(\mathcal{P}_2(X), W_2)$ is geodesic. 

\medskip

 For any $t\in [0,1]$,  let ${\rm e}_{t}$ denote the evaluation map: 
$$
  {\rm e}_{t} : \Geo(X) \to X, \qquad {\rm e}_{t}(\gamma) : = \gamma_{t}.
$$
Any geodesic $(\mu_t)_{t \in [0,1]}$ in $(\mathcal{P}_2(X), W_2)$  can be lifted to a measure $\nu \in {\mathcal {P}}(\Geo(X))$, 
so that $({\rm e}_t)_\sharp \, \nu = \mu_t$ for all $t \in [0,1]$. 
\\Given $\mu_{0},\mu_{1} \in \mathcal{P}_{2}(X)$, we denote by 
$\Opt(\mu_{0},\mu_{1})$ the space of all $\nu \in \mathcal{P}(\Geo(X))$ for which $({\rm e}_0,{\rm e}_1)_\sharp\, \nu$ 
realizes the minimum in \eqref{eq:W2def}. Such a $\nu$ will be called \emph{dynamical optimal plan}. If $(X,\sfd)$ is geodesic, then the set  $\Opt(\mu_{0},\mu_{1})$ is non-empty for any $\mu_0,\mu_1\in \mathcal{P}_2(X)$.
\\We will also consider the subspace $\mathcal{P}_{2}(X,\sfd,\mm)\subset \mathcal{P}_{2}(X)$
formed by all those measures absolutely continuous with respect with $\mm$.
\medskip

A set $G \subset \Geo(X)$ is a set of non-branching geodesics if and only if for any $\gamma^{1},\gamma^{2} \in G$, it holds:
$$
\exists \;  \bar t\in (0,1) \text{ such that } \ \forall t \in [0, \bar t\,] \quad  \gamma_{ t}^{1} = \gamma_{t}^{2}   
\quad 
\Longrightarrow 
\quad 
\gamma^{1}_{s} = \gamma^{2}_{s}, \quad \forall s \in [0,1].
$$
In the paper we will only consider essentially non-branching spaces, let us recall their definition (introduced in \cite{RS2014}). 
\begin{definition}\label{def:ENB}
A metric measure space $(X,\sfd, \mm)$ is \emph{essentially non-branching} (e.n.b. for short) if and only if for any $\mu_{0},\mu_{1} \in \mathcal{P}_{2}(X)$,
with $\mu_{0},\mu_{1}$ absolutely continuous with respect to $\mm$, any element of $\Opt(\mu_{0},\mu_{1})$ is concentrated on a set of non-branching geodesics.
\end{definition}
It is clear that if $(X,\sfd)$ is a  smooth Riemannian manifold  then any subset $G \subset \Geo(X)$  is a set of non branching geodesics, in particular any smooth Riemannian manifold is essentially non-branching.
\medskip

In order to formulate curvature properties for $(X,\sfd,\mm)$  we recall the definition of the  distortion coefficients:  for $K\in \R, N\in [1,\infty), \theta \in (0,\infty), t\in [0,1]$, set 
\begin{equation}\label{eq:deftau}
\tau_{K,N}^{(t)}(\theta): = t^{1/N} \sigma_{K,N-1}^{(t)}(\theta)^{(N-1)/N},
\end{equation}
where the $\sigma$-coefficients are defined 
as follows:
given two numbers $K,N\in \R$ with $N\geq0$, we set for $(t,\theta) \in[0,1] \times \R_{+}$, 
\begin{equation}\label{eq:Defsigma}
\sigma_{K,N}^{(t)}(\theta):= 
\begin{cases}
\infty, & \textrm{if}\ K\theta^{2} \geq N\pi^{2}, \crcr
\displaystyle  \frac{\sin(t\theta\sqrt{K/N})}{\sin(\theta\sqrt{K/N})} & \textrm{if}\ 0< K\theta^{2} <  N\pi^{2}, \crcr
t & \textrm{if}\ K \theta^{2}<0 \ \textrm{and}\ N=0, \ \textrm{or  if}\ K \theta^{2}=0,  \crcr
\displaystyle   \frac{\sinh(t\theta\sqrt{-K/N})}{\sinh(\theta\sqrt{-K/N})} & \textrm{if}\ K\theta^{2} \leq 0 \ \textrm{and}\ N>0.
\end{cases}
\end{equation}

\noindent
Let us also recall the definition of the R\'enyi Entropy functional ${\mathcal E}_{N} : \P(X) \to [0, \infty]$, 
\begin{equation}\label{eq:defEnt}
{\mathcal E}_{N}(\mu)  := \int_{X} \rho^{1-1/N}(x) \,\mm(dx),
\end{equation}
where $\mu = \rho \mm + \mu^{s}$ with $\mu^{s}\perp \mm$.
\\ Next we recall the definition of $\MCP(K,N)$ given independently by Ohta \cite{Ohta1} and Sturm \cite{Sturm06II}. 
On general metric measure spaces the two definitions slightly differ, but on essentially non-branching spaces they coincide. We report the one given in
\cite{Ohta1}.

\begin{definition}[$\MCP$ condition]\label{def:MCP}
Let $K \in \R$ and $N \in [1,\infty)$. A metric measure space  $(X,\sfd,\mm)$ verifies $\MCP(K,N)$ if for any $\mu_{0} \in \P_{2}(X)$ of the form 
$\mu_{0} = \frac{1}{\mm(a)}\mm\llcorner_{A}$ for some Borel set $A \subset X$ 
with $\mm(A) \in (0,\infty)$, and any $o \in X$ there exists $\nu \in \Opt(\mu_{0},\delta_{o})$ such that 
\begin{equation}\label{eq:defMCP}
\frac{1}{\mm(A)} \mm 
 \geq  (\ee_{t})_{\sharp} \left( \tau_{K,N}^{(1-t)}(\sfd(\gamma_{0},\gamma_{1})) \nu(d\gamma) \right), \qquad \forall \ t\in [0,1].
\end{equation}
\end{definition}
From \cite[Proposition 9.1]{CMi16}, in the setting of essentially non-branching spaces Definition \ref{def:MCP}
is equivalent to the following condition: for all $\mu_{0},\mu_{1} \in \mathcal{P}_{2}(X)$ with $\mu_{0} \ll \mm$ and $\supp(\mu_{1}) \subset \supp(\mm)$, there exists a unique $\nu \in \Opt(\mu_0,\mu_1)$, $\nu$ is induced by a map  (i.e. $\nu = S_{\sharp}(\mu_0)$ for some map $S : X \rightarrow \Geo(X)$), $\mu_t := (\ee_t)_{\#} \nu \ll \mm$ for all $t \in [0,1)$, and writing $\mu_t = \rho_t \mm$, we have for all $t \in [0,1)$:
\begin{equation} \label{E:MCP-density}
\rho_t^{-\frac{1}{N}}(\gamma_t) \geq  \tau_{K,N}^{(1-t)}(\sfd(\gamma_0,\gamma_1)) \rho_0^{-\frac{1}{N}}(\gamma_0) \;\;\; \text{for $\nu$-a.e. $\gamma \in \Geo(X)$} .
\end{equation}

The curvature-dimension condition was introduced independently by Lott-Villani \cite{Lott-Villani09} and Sturm \cite{Sturm06I,Sturm06II}, let us recall its definition.

\begin{definition}[$\CD$ condition]\label{def:CD}
Let $K \in \R$ and $N \in [1,\infty)$. A metric measure space  $(X,\sfd,\mm)$ verifies $\CD(K,N)$ if for any two $\mu_{0},\mu_{1} \in \P_{2}(X,\sfd,\mm)$ 
with bounded support there exist $\nu \in \Opt(\mu_{0},\mu_{1})$ and  $\pi\in \P(X\times X)$ $W_{2}$-optimal plan, such that $\mu_{t}:=(\ee_{t})_{\sharp}\nu \ll \mm$ and for any $N'\geq N, t\in [0,1]$:
\begin{equation}\label{eq:defCD}
{\mathcal E}_{N'}(\mu_{t}) \geq \int \tau_{K,N'}^{(1-t)} (\sfd(x,y)) \rho_{0}^{-1/N'} 
+ \tau_{K,N'}^{(t)} (\sfd(x,y)) \rho_{1}^{-1/N'} \,\pi(dxdy).
\end{equation}
\end{definition}

Throughout this paper, we will always assume the proper metric measure space $(X,\sfd,\mm)$ to satisfy $\MCP(K,N)$, for some $K,N \in \R$, and to be
essentially non-branching. This will imply in particular that $(X,\sfd)$ is geodesic.
\\It is not difficult to see that if  $(X,\sfd,\mm)$ verifies $\CD(K,N)$ then it also verifies  $\MCP(K,N)$, but the converse implication is false in general (for example the sub-Riemannian Heisenberg group satisfies $\MCP(K,N)$ for some suitable $K,N$, but does not satisfy $\CD(K',N')$ for any choice of $K',N'$).

It is worth recalling that if $(M,g)$ is a Riemannian manifold of dimension $n$ and 
$h \in C^{2}(M)$ with $h > 0$, then the m.m.s. $(M,\sfd_{g},h \, {\rm Vol}_{g})$ (where $\sfd_{g}$ and ${\rm Vol}_{g}$ denote the Riemannian distance and volume induced by $g$) verifies $\CD(K,N)$ with $N\geq n$ if and only if  (see  \cite[Theorem 1.7]{Sturm06II})
$$
\Ric_{g,h,N} \geq  K g, \qquad \Ric_{g,h,N} : =  \Ric_{g} - (N-n) \frac{\nabla_{g}^{2} h^{\frac{1}{N-n}}}{h^{\frac{1}{N-n}}}.  
$$
In particular if $N = n$ the generalized Ricci tensor $\Ric_{g,h,N}= \Ric_{g}$ makes sense only if $h$ is constant.

A variant of the $\CD$ condition, called  reduced curvature dimension condition and denoted by  $\CD^{*}(K,N)$ \cite{BS10},  asks for the same inequality \eqref{eq:defCD} of $\CD(K,N)$ but  the
coefficients $\tau_{K,N}^{(t)}(\sfd(\gamma_{0},\gamma_{1}))$ and $\tau_{K,N}^{(1-t)}(\sfd(\gamma_{0},\gamma_{1}))$ 
are replaced by $\sigma_{K,N}^{(t)}(\sfd(\gamma_{0},\gamma_{1}))$ and $\sigma_{K,N}^{(1-t)}(\sfd(\gamma_{0},\gamma_{1}))$, respectively.
For both definitions there is a local version and it was recently proved in \cite{CMi16} that  on  an essentially  non branching m.m.s. with $\mm(X)<\infty$, the  $\CD^{*}_{loc}(K,N)$, $\CD^{*}(K,N)$, $\CD_{loc}(K,N)$, $\CD(K,N)$ conditions are all equivalent, for all $K\in \R, N\in (1,\infty)$, 
via the $\CD^{1}(K,N)$ condition defined in terms of $L^{1}$-Optimal Transport problem. For more details we refer to \cite{CMi16}.

\subsection{Lipschitz functions and Laplacians in metric measure spaces}

We recall some  facts about  calculus in metric measure spaces following the approach of \cite{AGS11a, AGS11b, Gigli12} with the slight difference that here we confine the presentation to the (easier) setting of Lipschitz functions (instead of Sobolev), as in the paper we will work in such a framework.  For this subsection it is enough to assume the metric space $(X,\sfd)$ to be 
complete and separable and $\mm$ to be a non-negative locally finite measure.

A function $f:X\to \R$ is Lipschitz (or more precisely $L$-Lipschitz) if there exists a constant $L\geq 0$ such that 
$$
|f(x)-f(y)|\leq L \, \sfd(x,y), \quad \forall x,y\in X.
$$
The minimal constant $L\geq 0$ satisfying the last inequality is called \emph{global Lipschitz constant} of $f$ and is denoted with $\Lip(f)$. 
\\We denote by $\LIP(X)$ the space of real valued  Lipschitz functions on $(X,\sfd)$ and with $\LIP_{c}(\Omega)\subset \LIP(X)$ the sub-space of Lipschitz functions  of $X$ with compact support contained in the open subset $\Omega\subset X$.
\\ Given $f\in \LIP(X)$, the \emph{local Lipschitz constant} $|Df|(x_{0})$ of $f$ at $x_{0}\in X$ is defined as
\begin{equation*}
|Df|(x_{0}):=\limsup_{x\to x_{0}}  \frac{|f(x)-f(x_{0})|}{\sfd(x, x_{0})} \; \text{ if $x_{0}$ is not isolated}, \quad |Df|(x_{0})=0 \; \text{ otherwise}.
\end{equation*}
It is clear that $|Df|\leq \Lip(f)$ on all $X$.

\begin{definition}
Let $f,u \in \LIP(X)$. Define  the functions $D^{\pm} f (\nabla u) : X \to \R$  by 
$$
D^{+} f (\nabla u) : = \inf_{\ve > 0} \frac{ |D (u + \ve f)|^{2} - |Du|^{2} }{2\ve},
$$
while $D^{-} f (\nabla u)$ is obtained replacing $\inf_{\ve>0}$ with $\sup_{\ve <0}$.
\\In case $D^{+} f (\nabla u)= D^{-} f (\nabla u)$ $\mm$-a.e. for all  $f,u \in \LIP(X)$, then $(X,\sfd,\mm)$ is said (Lipschitz-)infinitesimally strictly convex and we set $Df(\nabla u):= D^{+} f (\nabla u)$; if moreover $Df (\nabla u)= Du (\nabla f)$ $\mm$-a.e. for all  $f,u \in \LIP(X)$, then $(X,\sfd,\mm)$ is said (Lipschitz)-infinitesimally Hilbertian.
\end{definition}

\begin{remark}
Given  $f,u \in \LIP(X)$, it is easily seen the map $\ve\mapsto  |D (u + \ve f)|^{2}$ is convex and real valued. Thus  
$$ 
\inf_{\ve > 0} \frac{ |D (u + \ve f)|^{2} - |Du|^{2} }{2\ve}=\liminf_{\ve \downarrow 0}  \frac{ |D (u + \ve f)|^{2} - |Du|^{2} }{2\ve},
$$ 
and  
$$ 
\sup_{\ve < 0} \frac{ |D (u + \ve f)|^{2} - |Du|^{2} }{2\ve}=\limsup_{\ve \uparrow 0}  \frac{ |D (u + \ve f)|^{2} - |Du|^{2} }{2\ve}.
$$
\end{remark}

\begin{remark}\label{rem:MCPDoubPoinc1}
The local doubling \& Poincar\'e condition will be satisfied throughout the paper as we will work in essentially non-branching $\MCP(K,N)$-spaces, with $K\in \R, N\in (1,\infty)$ thanks to 
\cite[Corollary p. 28]{VRN}. 
The standing assumptions in \cite{VRN} are 
$\MCP(K,N)$ and that the set 
$$
C_{x} : = \{ y \in X \colon \exists \gamma^{1}\neq \gamma^{2} \in \Geo(X), \, x = \gamma^{1}_{0}=\gamma^{2}_{0}, \, y = \gamma^{1}_{1}=\gamma^{2}_{1}  \},
$$
has $\mm$-measure zero for $\mm$-a.e. $x \in X$.

In an essentially non-branching $\MCP(K,N)$ space 
the previous property can be obtained as follows:
for any $r>0$ invoke \cite[Theorem 5.2]{CM16} 
with $\mu_{0} := \mm\llcorner_{B_{r}(x)}/\mm(B_{r}(x))$
and $\mu_{1}: = \delta_{x}$; existence of a map 
pushing $\mu_{0}$ to the unique element of $\Opt(\mu_{0},\mu_{1})$ yields that $\mm(C_{x}\cap B_{r}(x)) = 0$, actually for any $x \in X$.
 \end{remark}

\begin{remark} \label{rem:MCPDoubPoinc2}
The notions of infinitesimally strictly convex and infinitesimally Hilbertian have been introduced in \cite{AGS11b, Gigli12} in the setting of Sobolev spaces, with the local Lipschitz constant replaced by the minimal weak upper gradient.
The corresponding Lipschitz counterparts that we defined above have been already considered in \cite{MonAng} and coincide with the ones of  \cite{Gigli12} provided the space satisfies  doubling \& Poincar\'e locally, thanks to a deep result of Cheeger \cite{Cheeger97}.  
Thanks to Remark \ref{rem:MCPDoubPoinc1} we will avoid therefore the prefix ``Lipschitz'' in the corresponding notions, for simplicity of notation.  
\end{remark}

\begin{definition}[Test plans, \cite{AGS11a}]\label{D:testplan} Let $(X,\sfd,\mm)$ be a metric measure space as above and  $\pi \in \P(C([0,1],X))$. 
We say that $\pi$ is a test plan provided it has bounded compression, i.e. there exists $C > 0$ such that
$$
(\ee_{t})_{\sharp} \pi = \mu_{t} \leq C \mm, \qquad \forall \ t \in [0,1],
$$
and
$$
\int \int_{0}^{1} |\dot \gamma_{t}|^{2} dt \, \pi(d\gamma) < \infty.
$$
\end{definition}

\begin{definition}[Plans representing gradients]\label{D:plansgradient} 
Let $(X,\sfd,\mm)$ be a m.m.s., $g \in \LIP(X)$ and $\pi$ a test plan. 
We say that $\pi$ represents the gradient of $g$ provided it is a test plan and we have
$$
\liminf_{t \to 0} \int  \frac{g(\gamma_{t}) - g(\gamma_{0})  }{t} \,\pi(d\gamma) 
\geq  
\frac{1}{2} \int |D g|^{2}(\gamma_{0}) \,\pi(d\gamma) + \frac{1}{2} \limsup_{t\to 0} \frac{1}{t} \int \int_{0}^{t} |\dot \gamma_{s}|^{2} ds \,\pi(d\gamma)
$$
\end{definition}

\begin{theorem}  {\rm \cite[Lemma 4.5]{AGS11b}, \cite[Theorem 3.10]{Gigli12}.} \label{T:verticalhorizontal}
Let $f,u \in \LIP(X)$ and $\pi$ be any plan representing the gradient of $u$, then 
\begin{align*}
\int D^{+} f (\nabla u) \, (\ee_{0})_{\sharp} \pi 
&~ \geq \limsup_{t \to 0} \int \frac{f(\gamma_{t}) - f(\gamma_{0})  }{t} \, \pi(d\gamma) \\
&~ \geq \liminf_{t \to 0} \int \frac{f(\gamma_{t}) - f(\gamma_{0})  }{t} \, \pi(d\gamma) \\
&~ \geq \int D^{-} f (\nabla u) \, (\ee_{0})_{\sharp} \pi .
\end{align*}
In particular, if  $(X,\sfd,\mm)$ is infinitesimally strictly convex then 
$$
\int_{X} D f ( \nabla u) \, (\ee_{0})_{\sharp} \pi  = \lim_{t \to 0} \int  \frac{f(\gamma_{t}) - f(\gamma_{0})  }{t} \,\pi(d\gamma).
$$
\end{theorem}

In order to define the Laplacian, let us recall the definition of Radon functional. For simplicity, from now on, we will assume $(X,\sfd)$ to be locally compact (this will be satisfied throughout the paper as we will work in the setting of $\MCP(K,N)$ spaces which are, even more strongly, locally doubling).

\begin{definition}\label{D:radonfunctional}
\begin{itemize}
\item 
A \emph{Radon functional  over an open set 
$\Omega \subset X$} is a  linear functional $T:\LIP_{c}(\Omega)\to \R$ such that  for every compact subset $W\subset \Omega$ there exists a constant $C_{W}\geq 0$ so that
$$
|T(f)|\leq C_{W} \max_{W} |f|, \quad \text{for all } f\in \LIP_{c}(\Omega) \text{ with } \supp(f)\subset W.
$$ 
\item A \emph{non-negative Radon measure over an open set 
$\Omega \subset X$} is a  Borel, non-negative measure $\mu : \B(\Omega) \to [0,+\infty]$  that is locally finite, i.e. for any $x \in \Omega$ there exists a neighbourhood $U_{x}$ of finite $\mu$-measure: $\mu(U_{x})<+\infty$. A non-negative Radon measure is said to be \emph{finite} if $\mu(X)<\infty$. 
\item A \emph{signed Radon measure over an open set 
$\Omega \subset X$} is a  Borel measure $\mu : \B(\Omega) \to \R\cup\{\pm \infty\}$ that can be written as $\mu=\mu^{+}-\mu^{-}$ with $\mu^{+},\mu^{-}$ non-negative Radon measures where at least one of the two is finite.
\\A signed Radon measure is said to be  \emph{finite} if, denoting $\|\mu\|:=\mu^{+}+\mu^{-}$ the total variation measure, it holds   $\|\mu\|(X)<\infty$.
\end{itemize}
\end{definition}

Note that, by the classical Riesz-Markov-Kakutani Representation Theorem,  for every \emph{non-negative} Radon functional $T$ over $X$ there exists a non-negative Radon measure $\mu_{T}$ representing $T$ via integration, i.e. 
$$T(f)=\int_{X} f(x)\, \mu_{T}(dx), \quad \forall f\in \LIP_{c}(X). $$ In particular, every Radon functional can be written as the sum of two Radon measures (i.e. the positive and negative parts, respectively).
\\Let us stress that the non-negativity assumption is crucial. Indeed a general Radon functional may not be representable by a measure, for example consider  $X =  \R, \Omega=\R\setminus \set{0}$ and $T : \LIP_{c}(\Omega) \to \R$ defined by
$$
T : \LIP_{c}(\Omega) \to \R, \qquad
T(f) : = \int_{\Omega} \frac{f(x)}{x}dx.
$$
It is straightforward to see that $T$ is a real valued Radon functional over $\Omega$ but cannot be represented by a signed Radon measure over $\Omega$, the point being that $(-\infty, 0)$ would have ``measure'' $-\infty$ and   $(0,+\infty)$ would have ``measure'' $+\infty$ thus failing the additivity axiom.  An expert reader may recognise that $T(f)$ is (up to a multiplicative constant) the Hilbert transform of $f$ evaluated at $0$.

\begin{definition}\label{D:Laplace}
Let $\Omega\subset X$ be an open subset and let $u \in \LIP(X) $. We say that $u$ is in the domain of the Laplacian of $\Omega$, and write $u \in D(\bold{\Delta},\Omega)$, provided  there exists a Radon functional $T$ over $\Omega$ 
such that for any $f \in \LIP_{c}(\Omega)$ it holds 
\begin{equation}\label{eq:defTfLap}
 \int_{X} D^{-}f (\nabla u) \,\mm \leq - T(f)  \leq  \int_{X} D^{+}f (\nabla u) \,\mm.
\end{equation}
In this case we write $T \in \bold{\Delta} u \llcorner_{\Omega}$. In case $T$ can be represented by a signed measure $\mu$ over $\Omega$, with a slight abuse of notation we will identify $T$ with $\mu$ and write $\mu \in \bold{\Delta} u \llcorner_{\Omega}$.
\end{definition}

Let us stress that in general, there is not a unique operator $T$ satisfying \eqref{eq:defTfLap}; in other words the Laplacian can be multivalued.

\bigskip
\subsection{Synthetic Ricci lower bounds over the real line}\label{Ss:MCP1d}

Given $K \in \Real$ and $N \in (1,\infty)$, a non-negative Borel function $h$ defined on an interval $I \subset \Real$ is called a $\MCP(K,N)$ density on $I$ if for all $x_0,x_1 \in I$ and $t \in [0,1]$:
\begin{equation}\label{E:MCPdef}
 h(t x_1 + (1-t) x_0) \geq \sigma^{(1-t)}_{K,N-1}(\abs{x_1-x_0})^{N-1} h(x_0).
\end{equation}
\noindent

Even though it is a folklore result, 
we will include a proof of the following fact

\begin{lemma}
A one-dimensional metric measure space, that for simplicity we directly identify with $(I, |\cdot|, h \L^{1})$, verifies $\MCP(K,N)$ 
if and only there exists $\tilde h$, $\MCP(K,N)$ density, such that 
$h = \tilde h$, $\mathcal{L}^{1}$-a.e. on $I$.
\end{lemma}

\begin{proof}
Assume $h$ is an $\MCP(K,N)$ density on $I$.  
From \cite[Proposition 9.1]{CMi16}, point $iv)$, it will be enough to prove 
\eqref{E:MCP-density} under the additional assumption that 
$\mu_{0} = \frac{1}{\mm(A)}\chi_{A}\mm$ for some $A \subset I$ such that 
$0< \mm(A) <\infty$, with $\mm = h \mathcal{L}^{1}$.

Without any loss in generality we assume $o = 0 \in I$.
Given then any $A \subset I$ as above, the unique $W_{2}$ geodesic $(\mu_{t})_{t\in [0,1]}$ connecting $\mu_{0}$ to $\delta_{0}$ is 
$$
\mu_{t} = (f_{t})_{\sharp} \mu_{0}, \qquad f_{t}(x) = (1-t)x.
$$
Then using change of variable formula, 
$$
\mu_{t} = \rho_{t} \mm, \qquad \rho_{t}(x) = \frac{h(x/(1-t))}{h(x)} 
\frac{\chi_{A}(x/(1-t))}{(1-t)\mm(A)},
$$
implying that 
$$
\left(\frac{\rho_{t}(f_{t}(x))}{\rho_{0}(x)}\right)^{-1/N} = \left(\frac{(1-t)h((1-t)x)}{h(x)}\right)^{1/N} 
\geq (1-t)^{1/N}\sigma_{K,N-1}^{(1-t)}(|x|)^{(N-1)/N} = \tau_{K,N}^{(1-t)}(|x|),
$$
proving \eqref{E:MCP-density}. 
In order to prove the converse implication, we fix $x_{1} = 0 = o$ 
and take 
$$
\mu_{0} : = \frac{1}{\mathcal{L}^{1}(A)}\mathcal{L}^{1}\llcorner_{A}, 
\qquad A\subset I,  0 < \mathcal{L}^{1}(A)< \infty.
$$
Then $\mu_{t} : = \frac{1}{\mathcal{L}^{1}((1-t)A)}
\mathcal{L}^{1}\llcorner_{(1-t)A}$  is the unique $W_{2}$-geodesic connecting 
$\mu_{0}$ to $\delta_{o}$. Hence \eqref{E:MCPdef} can be applied to 
$$
\mu_{t} = \rho_{t} \mm, 
\qquad \rho_{t}(x) = \frac{1}{(1-t)\mathcal{L}^{1}(A)} \frac{\chi_{(1-t)A}(x)}{h(x)}
$$
Then \eqref{E:MCPdef} along $(\mu_{t})_{}$ implies the claim.
\end{proof}

\smallskip

The estimate \eqref{E:MCPdef} implies several known properties that we collect in what follows. 
To write them in a unified way we 
define for $\kappa \in \erre$ the function $s_{\kappa} : [0,+\infty) \to \erre$ (on $[0,\pi/\sqrt{\kappa})$ if $\kappa>0$) 
\begin{equation}\label{E:sk}
s_{\kappa}(\theta):= 
\begin{cases}
(1/\sqrt{\kappa})\sin(\sqrt{\kappa}\theta) & {\rm if}\ \kappa>0, \crcr
\theta & {\rm if}\ \kappa=0, \crcr
(1/\sqrt{-\kappa})\sinh(\sqrt{-\kappa}\theta) &{\rm if}\ \kappa<0.
\end{cases}
\end{equation}
For the moment we confine ourselves to the case $I = (a,b)$ with $a,b \in \R$; hence
\eqref{E:MCPdef} implies
\begin{equation}\label{E:MCPdef2}
\left( \frac{s_{K/(N-1)}(b - x_{1}  )}{s_{K/(N-1)}(b - x_{0}  )} \right)^{N-1} 
\leq \frac{h(x_{1} ) }{h (x_{0})} \leq 
\left( \frac{s_{K/(N-1)}( x_{1} -a  )}{s_{K/(N-1)}( x_{0} -a  )} \right)^{N-1}, 
\end{equation}
for $x_{0} \leq x_{1}$ (see 
the proof of Lemma \ref{L:rigiditMCP} for the easier estimate in the case $K = 0$).
Hence denoting with $D = b - a$ the length of $I$,  for any $\ve >0$ it follows that
\begin{equation}\label{E:bounds}
\sup \left\{ \frac{h (x_{1})}{h (x_{0})} \colon x_{0},x_{1} \in [a+ \ve, b - \ve ]\right\} \leq C_{\ve},
\end{equation}
where $C_{\ve}$ only depends on $K,N$, provided $ 2 \ve \leq D \leq \frac{1}{\ve}$.

\smallskip
Moreover \eqref{E:MCPdef2} implies that $h$ is locally Lipschitz in the interior of $I$ and an
easy manipulation of it (cf. \cite[Lemma A.9]{CMi16}) yields the following bound on the derivative of $h$: 
\begin{equation}\label{E:logder}
- (N-1)  \frac{s_{K/(N-1)}' (b-x)}{s_{K/(N-1)}(b-x)}  \leq (\log h)'(x) \leq (N-1) \frac{s_{K/(N-1)}' (x-a)}{s_{K/(N-1)}(x-a)},
\end{equation}
if $x \in (a,b)$ is a point of differentiability of $h$. 
Finally if $k > 0$, then $b-a \leq \pi \sqrt{(N-1)/K}$.

\begin{remark}\label{R:continuityboundary}
The estimate \eqref{E:MCPdef2} also implies that a $\MCP(K,N)$ density $h : (a,b) \to (0,\infty)$, $a,b\in \R$, can always  be extended 
to a continuous function on the closed interval $[a,b]$.
Notice indeed that the map
$$
(a,b) \ni x \mapsto \frac{h (x ) }{(s_{K/(N-1)}(b-x))^{N-1}},
$$
is non-decreasing and strictly positive. Hence the following limit exists and is a real number
$$
\lim_{x \to a}\frac{h (x ) }{(s_{K/(N-1)}(b-x))^{N-1}}.
$$
Since $b - a > 0$, we obtain that also the 
limit $\lim_{x \to a} h (x )$
exists,
for every $K\leq 0$ and for $K > 0$ provided $b-a \neq \pi \sqrt{(N-1)/K}$. 
The case $K> 0$ and 
$b-a = \pi \sqrt{(N-1)/K}$ follows by rigidity: \eqref{E:MCPdef2} implies that 
$$
\frac{\sin(\pi -x_{1}\sqrt{K/(N-1)} )  }
{\sin(\pi -x_{0}\sqrt{K/(N-1)} )  } 
\leq \frac{h(x_{1})}{h(x_{0})} 
\leq 
\frac{\sin(x_{1}\sqrt{K/(N-1)} )  }
{\sin(x_{0}\sqrt{K/(N-1)} )  }, 
$$
showing that $h(x)$, up to a renormalisation constant,
coincide with $\sin(x\sqrt{K/(N-1)})$.
To show that $h$ can also be extended to a continuous function at $b$, one can argue as above starting from the non-increasing property of the following function 
$$
(a,b) \ni x \mapsto \frac{h (x ) }{(s_{K/(N-1)}(x-a))^{N-1}},
$$
following again from \eqref{E:MCPdef2}.
\end{remark}

The next lemma was stated and proved in  \cite[Lemma A.8]{CMi16}  under the $\CD$ condition; as the proof only uses $\MCP(K,N)$ we report it in this more general version.

\begin{lemma} \label{lem:apriori0}
Let $h$ denote a $\MCP(K,N)$ density on a finite interval $(a,b)$, $N \in (1,\infty)$, which integrates to $1$. Then:
\begin{equation}\label{eq:suphdiam}
\sup_{x \in (a,b)} h(x) \leq \frac{1}{b-a} \begin{cases} N & K \geq 0  \\ (\int_0^1 (\sigma^{(t)}_{K,N-1}(b-a))^{N-1} dt)^{-1} & K < 0 \end{cases} .
\end{equation}
In particular, for fixed $K$ and $N$, $h$ is uniformly bounded from above as long as $b-a$ is uniformly bounded away from $0$ (and from above if $K < 0$).
\end{lemma}

\medskip

From the previous auxiliary results we obtain the following lemma that will be used throughout the paper.
\begin{lemma}\label{L:integrabilityh}
Let $h$ denote a $\MCP(K,N)$ density on a finite interval $(a,b)$, $N \in (1,\infty)$, which integrates to $1$. Then:
\begin{equation}\label{eq:EstInth'}
\int_{(a,b)}|h'(x)| \,dx \leq \frac{1}{b-a} \, C^{(K,N)}_{(b-a)},
\end{equation}
for some $C^{(K,N)}_{(b-a)}>0$ with the property that, for fixed  $K\in \R$ and $N\in (1,\infty)$, it holds 
\begin{equation}\label{eq:defLambda}
\sup_{r\in (0,R)}C^{(K,N)}_{r}<\infty, \quad  \text{for every }R>0, \quad \lim_{r\uparrow \infty}C^{(K,N)}_{r}=\infty.
\end{equation}
\end{lemma}

\begin{proof}

\textbf{Case $K\leq 0$}.
The two inequalities in \eqref{E:logder} give for each point $x \in (a,b)$ of differentiability of $h$ 
\begin{equation}\label{eq:h'cot}
w_{1}:=h'(x) + (N-1) \frac{s_{K/(N-1)}' (b-x)}{s_{K/(N-1)}(b-x)}   h(x) \geq 0, \quad w_{2}:=h'(x) - (N-1) \frac{s_{K/(N-1)}' (x-a)}{s_{K/(N-1)}(x-a)} h(x) \leq 0.
\end{equation}
Thus, we can write
\begin{align}
\int_{[a,b]}|h'| \,dx& \leq \int_{\left[a,a+\frac{b-a}{2} \right]}w_{1} \,dx+ \int_{\left[a,a+\frac{b-a}{2}\right]}|w_{1}-h'| \,dx \nonumber\\
& \quad -  \int_{\left[a+\frac{b-a}{2}, b\right]}w_{2} \,dx+ \int_{\left[a+\frac{b-a}{2}, b\right]}|w_{2}-h'| \,dx .\label{eq:inth'pf}
\end{align}
First of all, observing that for $K \leq 0$ one has 
$
\frac{s_{K/(N-1)}' (t)}{s_{K/(N-1)}(t)}\geq0 
$ for all  $t\geq 0$,
we get
\begin{align*}
\int_{\left[a,a+\frac{b-a}{2} \right]}w_{1} \,dx &\leq h\left(a+\frac{b-a}{2} \right)-h(a) +(N-1)\|h\|_{L^{\infty}(a,b)} \log\left( \frac{s_{K/(N-1)}(b-a)} {s_{K/(N-1)}((b-a)/2)} \right) \nonumber\\
&\leq  C_{(b-a)}^{(K,N)}\, \|h\|_{L^{\infty}(a,b)}, \\
 \int_{\left[a,a+\frac{b-a}{2}\right]}|w_{1}-h'| \,dx &=  \int_{\left[a,a+\frac{b-a}{2}\right]}  (N-1) \frac{s_{K/(N-1)}' (b-x)}{s_{K/(N-1)}(b-x)}   h(x)  \leq  C^{(K,N)}_{(b-a)}\, \|h\|_{L^{\infty}(a,b)},  
\end{align*}
where $r\mapsto C^{(K,N)}_{r}$ satisfies  \eqref{eq:defLambda}. The bounds for the second line of \eqref{eq:inth'pf} are analogous. Thus we conclude
$$
\int_{[a,b]}|h'| \, dx \leq   C^{(K,N)}_{(b-a)}\, \|h\|_{L^{\infty}(a,b)}
$$
which, recalling \eqref{eq:suphdiam}, gives the claim \eqref{eq:EstInth'}.

\medskip
\textbf{Case $K>0$}.\\
In order to simplify the notation, we assume  $K = N-1 > 0$ (so that $b-a \leq \pi$), $a = 0$ and $b - a = D\leq \pi$. The discussion for general $K>0$, $a<b\in [0,\pi]$ is analogous.

We first consider the case $D \leq \pi/2$. Using \eqref{E:MCPdef2}, notice that 
$$
\frac{h(x)}{\sin(D-x)} \leq \frac{h(D/2)}{\sin(D/2)}, \qquad  \forall  \ x \in [0,D/2],
$$
and 
$$
\frac{h(x)}{\sin(x)} \leq \frac{h(D/2)}{\sin(D/2)}, \qquad  \forall  \ x \in [D/2,D].
$$
For $x \in [0,D/2]$ these yield (recall that $\cos(x)\geq 0$)
$$
\omega_{0} '(x) : = h'(x) +  \frac{\cos( D -x)}{\sin(D/2)} h(D/2) \geq h'(x) + \frac{\cos( D -x)}{\sin(D-x)} h(x) \geq 0,
$$
and for $x \in [D/2,D]$ (recall that $\cos(x)\geq 0$)
$$
\omega_{1} '(x) : = h'(x) -  \frac{\cos( x)}{\sin(D/2)} h(D/2) \leq h'(x) -  \frac{\cos( x)}{\sin(x)} h(x) \leq 0.
$$
Then we can collect all the estimates together: 
\begin{align}\label{E:MCPderivat}
\int_{[0,D]} |h'(x)|\,\leq& \int_{[0,D/2]} \omega_{0}'(x) \,dx + \int_{[0,D/2]} |\omega_{0}'(x) - h'(x) |\,dx  \nonumber\\
&- \int_{[D/2,D]} \omega_{1}'(x) \,dx +\int_{[D/2,D]} |\omega_{1}'(x) - h'(x) |\,dx \nonumber \\
 \leq~& C \| h\|_{L^{\infty}(0,D)}.
\end{align}
The claim \eqref{eq:EstInth'}  then follows applying  Lemma \ref{lem:apriori0}.

If $D>\pi/2$, like in the case $K\leq0$, the two inequalities in \eqref{E:logder} give for each point $x \in (0,D)$ of differentiability of $h$ 
\begin{equation*}
h'(x) + (N-1) \frac{\cos( D -x)}{\sin(D-x)} h(x) \geq 0, \qquad h'(x) - (N-1)  \frac{\cos( x)}{\sin(x)} h(x) \leq 0.
\end{equation*}
Hence for $x \in (0, D -\pi/2)$ we have $h' (x) \geq 0$ and for $x \in [\pi/2, D]$ we have $h'(x) \leq 0$. 
Then Lemma \ref{lem:apriori0} and the bound $D \leq  \pi$
imply  that
\begin{align*}
\int_{[0, D -\pi/2]\cup [\pi/2, D]} |h'(x)| \,dx &= \int_{[0, D -\pi/2]}  h'(x) \,dx - \int_{[\pi/2, D]}  h'(x) \,dx \\
&\leq 4 \sup_{[0,D]} |h|\leq \frac{4N}{D}.
\end{align*}
In order to complete the proof it is then enough to bound 
$\int_{ [D-\pi/2,\pi/2]}|h'(x)|\, dx$. Since 
\eqref{E:MCPderivat} was obtained for any $h$ $\MCP$-density on $[0,D]$ 
with $D \leq \pi/2$ without using the assumption of $\int h = 1$, it implies
$$
\int_{[0,\pi/2]} |h'(x)|\, dx \leq C \| h \|_{L^{\infty}[0,\pi/2]},
$$
for any $\MCP$-density on $[0,D]$ with $D \geq \pi/2$. 
Lemma \ref{lem:apriori0} gives the claim.
\end{proof}

In the proof of the Splitting Theorem for $\MCP(0,N)$ spaces we will use the next lemma.

\begin{lemma}\label{L:rigiditMCP}
Let $h$ be a $\MCP(0,N)$ measure on the whole real line $\R$. Then $h$ is identically equal to  a real constant.
\end{lemma}

\begin{proof}
We show that $h(x_{0})=h(x_{1})$ for all $x_{0}, x_{1}\in \R$. 
\\The $\MCP(0,N)$ condition reads as
\begin{equation*}
 h(t x_1 + (1-t) x_0) \geq ( 1-t )^{N-1} \, h(x_0).
\end{equation*}
For $a<z < b$ apply the previous estimate for $ z = x_{0}$ and $x_{1} = b$, 
it implies 
$$
\frac{h(t b + (1-t)z)}{h(z) } \geq ( 1-t )^{N-1};
$$
if $w \in (z,b)$ and $w = t b + (1-t)z$ for some $t \in (0,1)$, implies  
$1-t = (b-w)/(b-z)$. Plugging in the previous inequality the explicit expression 
of $(1-t)$ and repeating the argument taking now $x_{0}=a$ and $x_{1}=z$,
we obtain the next two sided estimate
\begin{equation}\label{E:log-bound0}
\left( \frac{b - x_{1}}{b -x_{0}} \right)^{N-1} 
\leq \frac{h(x_{1} ) }{h (x_{0})} \leq 
\left(\frac{x_{1}-a}{x_{0}-a}\right)^{N-1}, 
\end{equation}
valid for all $a \leq x_{0} \leq x_{1} \leq b$.
Since 
$$
\lim_{b\to + \infty} \left( \frac{b - x_{1}}{b -x_{0}} \right)^{N-1} =1= \lim_{a\to - \infty} \left(\frac{x_{1}-a}{x_{0}-a}\right)^{N-1}, 
$$
and since \eqref{E:log-bound0} holds for all $a\in(-\infty, x_{0})$ and all $b\in(x_{1}, +\infty)$, the thesis follows.
\end{proof}

We now review few facts about $\CD(K,N)$ densities of the real line 
(see \cite[Appendix]{CMi16}).
Given $K \in \Real$ and $N \in (1,\infty)$, a non-negative Borel function $h$ defined on an interval $I \subset \Real$ is called a $\CD(K,N)$ density on $I$ if for all $x_0,x_1 \in I$ and $t \in [0,1]$
\begin{equation}\label{E:CD}
h^{\frac{1}{N-1}}( (1-t)x_{0} + t x_{1}) \geq h^{\frac{1}{N-1}}( x_{0}) \sigma_{K,N}^{(1-t)}(|x_{1}-x_{0}|)+ 
h^{\frac{1}{N-1}}( x_{1})\sigma_{K,N-1}^{(t)}(|x_{1}-x_{0}|).
\end{equation}
A  one-dimensional metric measure space, say $(I, |\cdot|, h \L^{1})$, satisfies $\CD(K,N)$ 
if and only  $h$ has a continuous representative $\tilde h$ 
that is a $\CD(K,N)$ density.

We will make use of the fact that a $\CD(K,N)$ density $h : I \to [0,\infty)$ is locally semi-concave in the interior, i.e. for all $x_0$ in the interior of $I$, there exists $C_{x_0} \in \Real$ so that $h(x) - C_{x_0} x^2$ is concave in a neighborhood of $x_0$.

Recall moreover that 
if $f : I \rightarrow \Real$ denotes a convex function on an open interval $I \subset \Real$, it is well-known that the left and right derivatives $f^{\prime,-}$ and $f^{\prime,+}$ exist at every point in $I$ and that $f$ is locally Lipschitz; in particular, $f$ is differentiable at a given point if and only if the left and right derivatives coincide. Denoting by $D \subset I$ the differentiability points of $f$ in $I$, it is also well-known that $I \setminus D$ is at most countable.
Clearly, all of these results extend to locally semi-convex and locally semi-concave functions as well.
We finally recall the next regularization property for $\CD(K,N)$ densities obtained in \cite[Proposition A.10]{CMi16}

\begin{proposition} \label{prop:log-convolve}
Let $h$ be a $\CD(K,N)$ density on an interval $(a,b)$. Let $\psi_\eps$ denote a non-negative $C^2$ function supported on $[-\eps,\eps]$ with $\int \psi_\eps = 1$. 
For any $\eps \in (0,\frac{b-a}{2})$, define the function $h^\eps$ on $(a+\eps,b-\eps)$ by:
\[
\log h^\eps := \log h \ast \psi_\eps:= \int \log h(y) \, \psi_{\eps}(x-y)  \, dy.
\]
Then $h^\eps$ is a $C^2$-smooth $\CD(K,N)$ density on $(a+\eps ,b -\eps)$. 
\end{proposition}

\bigskip



\part{A representation formula for the Laplacian}\label{part1}

\section{Transport set and Disintegration}\label{S:transportset}
Throughout this section we assume $(X,\sfd,\mm)$ to be a metric measure space with 
$\supp(\mm) = X$ and $(X,\sfd)$ geodesic and proper (and hence complete).

\subsection{Disintegration of $\sigma$-finite measures}\label{Ss:disint}

To any $1$-Lipschitz function $u : X \to \R$ there is a naturally associated $\sfd$-cyclically monotone set: 
\begin{equation}\label{E:Gamma}  
\Gamma_{u} : = \{ (x,y) \in X\times X : u(x) - u(y) = \sfd(x,y) \}.
\end{equation}
Its transpose is given by $\Gamma^{-1}_{u}= \{ (x,y) \in X \times X : (y,x) \in \Gamma_{u} \}$. We define the \emph{transport relation} $R_u$ and the \emph{transport set} $\mathcal{T}_{u}$, as:
\begin{equation}\label{E:R}
R_{u} := \Gamma_{u} \cup \Gamma^{-1}_{u} ~,~ \mathcal{T}_{u} := P_{1}(R_{u} \setminus \{ x = y \}) ,
\end{equation}
where $\{ x = y\}$ denotes the diagonal $\{ (x,y) \in X^{2} : x=y \}$ and $P_{i}$ is the projection onto the $i$-th component. Recall that $\Gamma_u(x) = \set{y \in X \; :\; (x,y) \in \Gamma_u}$ denotes the section of $\Gamma_u$ through $x$ in the first coordinate, and similarly for $R_u(x)$ (through either of the  coordinates by symmetry). 
Since $u$ is $1$-Lipschitz, $\Gamma_{u}, \Gamma^{-1}_{u}$ and $R_{u}$ are closed sets, and so are $\Gamma_u(x)$ and $R_u(x)$.

Also recall the following definitions, introduced in \cite{cava:MongeRCD}:
\begin{align*}
	A_{+}	: = 	&~\{ x \in \mathcal{T}_{u} : \exists z,w \in \Gamma_{u}(x), (z,w) \notin R_{u} \}, \nonumber \\ 
	A_{-}		: = 	&~\{ x \in \mathcal{T}_{u} : \exists z,w \in \Gamma^{-1}_{u}(x), (z,w) \notin R_{u} \}.
\end{align*}
$A_{\pm}$ are called the \emph{sets of forward and backward branching points}, respectively.
If $x \in A_{+}$ and $(y,x) \in \Gamma_{u}$ necessarily also $y \in A_{+}$ (as $\Gamma_{u}(y) \supset \Gamma_{u}(x)$ by the triangle inequality);
similarly, if $x \in A_{-}$ and $(x,y) \in \Gamma_{u}$ then necessarily $y \in A_{-}$.

Consider the \emph{non-branched transport set} 
\begin{equation}\label{eq:defTub}
\T_{u}^{nb} : = \T_{u} \setminus (A_{+} \cup A_{-}),
\end{equation}
and define the \emph{non-branched transport relation}:
\[
R_u^{nb} := R_u \cap (\T_u^{nb} \times \T_u^{nb}) .
\]
In was shown in \cite{cava:MongeRCD} (cf. \cite{biacava:streconv}) that $R_u^{nb}$ is an equivalence relation over $\T_{u}^{nb}$ and that for any $x \in \T_{u}^{nb}$, $R_{u}(x) \subset (X,\sfd)$ is isometric to a closed interval in $(\Real,\abs{\cdot})$. 

Therefore, from the non-branched transport relation $R_{u}^{nb}$, one obtains a partition of the non-branched transport set $\T_{u}^{nb}$ into a disjoint family (of equivalence classes) $\{X_{\alpha}\}_{\alpha \in Q}$ each of them isometric to a closed interval of $\R$. Here $Q$ is any set of indices.
Concerning the measurability, as the space $(X,\sfd)$ is proper, 
$\T_{u}$ and $A_{\pm}$ are $\sigma$-compact sets 
and, consequently, $\T_u^{nb}$ and $R_u^{nb}$ are Borel.

\begin{remark}[Initial and final points]\label{R:initialpoints}
It will be useful to isolate two families of distinguished points of 
the transport set: the \emph{set of initial and final points}, respectively: 
$$
a : = \{ x \in \T_{u} \colon \nexists y \in \T_{u}, y \neq x, \ (y,x) \in R_{u} \},
$$
$$
b : = \{ x \in \T_{u} \colon \nexists y \in \T_{u}, y \neq x, \ (x,y) \in R_{u} \}.
$$
Notice that no inclusion of the form $a \subset A_{+}$, $b \subset  A_{-}$ 
is valid. For instance consider $X  = \{ (x_{1},x_{2}) \in \R^{2} \colon x_{1}\geq 0 \}$ endowed with the Euclidean distance and $u (x) := \dist (x, \{ x_{1} = 0\})$; then 
 $a = \{x_{1} =0 \}$ and $A_{\pm} = \emptyset$. 
In particular, sets $a$ and $b$ may or may not be subset of $\T_{u}^{nb}$.
See also the discussion right above \eqref{eq:disIntro}. 
Curvature assumptions will anyway imply that $a$ and $b$ have measure zero.
We will also use the notations $a(X_{\alpha}), b(X_{\alpha})$ to denote 
the starting and the final points, respectively, of the transport 
set $X_{\alpha}$, whenever they exist.  
\end{remark}

\smallskip

Once a partition of the non-branched transport set $\T_{u}^{nb}$ is at disposal, a decomposition of the reference measure $\mm \llcorner _{\T_{u}^{nb}}$ can be obtained using the Disintegration Theorem. 
In the recent literature of Optimal Transportation, 
disintegration formulas have always been obtained under the additional assumption of finiteness of the measure $\mm(X) < \infty$. 
We will therefore spend few words on how 
to use Disintegration Theorem to obtain a disintegration associated to the family of transport rays without assuming $\mm(X) < \infty$.

We first introduce the quotient map $\QQ : \T_{u}^{nb} \to Q$ induced by the partition:
\begin{equation}\label{E:defineQ}
\alpha = \QQ(x) \iff x \in X_{\alpha}.
\end{equation}
The set of indices (or quotient set) $Q$ can be endowed with the quotient $\sigma$-algebra $\mathscr{Q}$ (of the $\sigma$-algebra $\mathscr{X}$ over $X$ of $\mm$-measurable  subsets):
$$
C \in \mathscr{Q} \quad \Longleftrightarrow \quad \QQ^{-1}(C) \in \mathscr{X},
$$
i.e. the finest $\sigma$-algebra on $Q$ such that $\QQ$ is measurable.

The set of indices $Q$ can be identified with 
any subset of $\bar Q \subset X$ verifying the following two properties
\begin{itemize}
\item[-] for all $x \in \T_{u}^{nb}$ there exists a unique $\bar x \in \bar Q$ such that $(x,\bar x) \in 
R_{u}^{nb}$;
\item[-] if $x, y \in \T_{u}^{nb}$ and $(x,y) \in R_{u}^{nb}$, then $\bar x = \bar y$.
\end{itemize}
In particular $\bar Q$ has to contain a single element for each equivalence class $X_{\alpha}$.  \\
Another way to obtain a quotient set is to look instead first for an 
explicit quotient map:
in particular, any map
$\bar \QQ : \T_{u}^{nb} \to \T_{u}^{nb}$ 
verifying the following two properties
\begin{itemize}
\item [-] $(x,\bar \QQ(x)) \in R_{u}^{nb}$;
\item [-] if $(x,y)\in R_{u}^{nb}$, then $\bar\QQ(x) = \bar \QQ(y)$,
\end{itemize}
will be a quotient map for the equivalence relation $R^{nb}_{u}$ 
over $\T_{u}^{nb}$;
then the quotient set associated to $\bar \QQ$ will be the set $\{x \in R_{u}^{nb} \colon x = \bar \QQ(x) \}$.

Existence of $\bar Q$ or of $\bar \QQ$ can be always deduced by axiom of choice. 
Anyway in order to apply disintegration theorem  measurability properties are needed.

A rather explicit construction of the quotient map has been already obtained 
under the additional assumption of $\mm(X) <\infty$ (cf. \cite{CM1}, \cite[Lemma 3.8]{CM17}); anyway $\mm(X) <\infty$ did not play any role in the proof and 
we therefore simply report the next statement.  \\ 
We will denote with $\mathcal{A}$  the $\sigma$-algebra generated by 
the analytic sets of $X$.

\begin{lemma}[$Q$ is locally contained in level sets of $u$]\label{lem:Qlevelset}
There exists an $\mathcal{A}$-measurable quotient map $\QQ: \T_{u}^{nb} \to Q$ such that the quotient set $Q \subset X$ is 
$\mathcal{A}$-measurable and can be written locally  as a level set of $u$ in the following sense: 
$$
Q = \bigcup_{n\in \N} Q_{n}, \qquad Q_{n} \subset u^{-1}(l_{n}), 
$$
where $l_{n} \in \Q$ and $Q_{i} \cap Q_{j} = \emptyset$, for $i\neq j$.
\end{lemma}

Lemma \ref{lem:Qlevelset} allows to apply Disintegration Theorem (cf. \cite[Section 6.3]{CMi16}), provided  the ambient measure $\mm$  is suitably modified into a finite measure.
To this aim, it will be useful the next elementary lemma.

\begin{lemma}\label{lem:constrf}
Let $\mm$ be a $\sigma$-finite measure over the proper metric space $(X,\sfd)$ with $\supp (\mm)=X$. Then there exists a Borel function  $f : X \to (0,\infty)$ satisfying 
\begin{equation}\label{eq:deff}
 \inf_{K} f>0, \; \text{for any compact subset $K\subset X$}, \quad \int_{\T_{u}^{nb}} f \mm = 1.
\end{equation}
\end{lemma}

\begin{proof}
Since by assumption  $(X,\sfd)$ is proper, then for every $x_{0}\in X$ and $R>0$ the closed metric ball $\bar{B}_{R}(x_{0})$ is compact. Thus, using that $\mm$ is $\sigma$-finite and $\supp (\mm)=X$, we get that
$$0< \mm(B_{n}(x_{0}) \setminus B_{n-1}(x_{0}))<\infty, \quad \text{for all } n \in \N_{\geq 1}.$$
It is then readily checked that  $f : X \to (0,\infty)$ defined by
$f:= \frac{1}{2^{n}\mm(B_{n}(x_{0}) \setminus B_{n-1}(x_{0}))}$ on $B_{n+1}(x_{0}) \setminus B_{n}(x_{0})$ for all  $n \in \N_{\geq 1}$ satisfies \eqref{eq:deff}.
\end{proof}

Under the assumption that $\mm$ is $\sigma$-finite, let $f : X \to (0,\infty)$ be satisfying \eqref{eq:deff},  set 
\begin{equation}\label{eq:defmu}
\mu : = f \mm\llcorner_{\T_{u}^{nb}},
\end{equation}
and define the normalized quotient measure 
\begin{equation}\label{eq:defqq}
\qq : = \QQ_{\sharp}\, \mu.
\end{equation}
Notice that $\qq$ is a Borel probability measure over $X$.
It is straightforward to check that 
$$
\QQ_{\sharp} (\mm\llcorner_{\T_{u}^{nb}}) \ll \qq.
$$
Take indeed $E \subset Q$ with $\qq(E) = 0$; then by definition
$\int_{\QQ^{-1}(E)} f(x)\, \mm(dx) = 0$,
implying  $\mm(\QQ^{-1}(E)) = 0$, since $f > 0$.

From the Disintegration Theorem \cite[Section 452]{Fremlin4}, we deduce the existence 
of a map 
$$
Q \ni \alpha \longmapsto \mu_{\alpha} \in \mathcal{P}(X)
$$
verifying the following properties:
\begin{itemize}
\item[(1)] for any $\mu$-measurable set $B\subset X$, the map $\alpha \mapsto \mu_{\alpha}(B)$ is $\qq$-measurable; \smallskip
\item[(2)] for $\qq$-a.e. $\alpha \in Q$, $\mu_{\alpha}$ is concentrated on $\QQ^{-1}(\alpha)$; \smallskip
\item[(3)] for any $\mu$-measurable set $B\subset X$ and $\qq$-measurable set $C\subset Q$, the following disintegration formula holds: 
$$
\mu(B \cap \QQ^{-1}(C)) = \int_{C} \mu_{\alpha}(B) \, \qq(d\alpha).
$$
\end{itemize}
Finally the disintegration is $\qq$-essentially unique, i.e. if any other 
map $Q \ni \alpha \longmapsto \bar \mu_{\alpha} \in \mathcal{P}(X)$
satisfies the previous three points, then 
$$
\bar \mu_{\alpha} = \mu_{\alpha},\qquad \qq\text{-a.e.} \ \alpha \in Q. 
$$
Hence once $\qq$ is given (recall that $\qq$ depends on $f$ 
from Lemma \ref{lem:constrf}), the disintegration is unique up to a set of 
$\qq$-measure zero. 
In the case $\mm(X) < \infty$, the natural choice, that we tacitly assume, 
is to take as $f$ the characteristic function of $\T_{u}^{nb}$ divided by $\mm(\T_{u}^{nb})$ so that $\qq : = \QQ_{\sharp} (\mm\llcorner_{\T_{u}^{nb}}/\mm(\T_{u}^{nb})$).

All the previous properties will be summarized saying that 
$Q \ni \alpha \mapsto \mu_{\alpha}$ is a  \emph{disintegration of 
$\mu$ strongly consistent with respect to $\QQ$}.
\\It follows from \cite[Proposition 452F]{Fremlin4} that 
$$
\int_{X} g(x) \mu(dx) = \int_{Q} \int g(x) \mu_{\alpha}(dx) \,\qq(d\alpha),
$$
for every $g : X \to \R\cup\{\pm \infty\}$ such that $\int g \mu$ is 
well-defined in $\R\cup\{\pm \infty\}$. 
Hence picking $g = 1/f$ (where $f$ is the one used to define $\mu$),  we get that 
\begin{equation}\label{E:disintmm}
\mm\llcorner_{\T_{u}^{nb}} = \int_{Q} \frac{\mu_{\alpha}}{f} \, \qq(d\alpha);
\end{equation}
previous identity has to be understood with test functions as the previous formula.

Defining $\mm_{\alpha} : = \mu_{\alpha}/f$, we 
obtain that $\mm_{\alpha}$ is a non-negative Radon measure over $X$ verifying 
all the measurability properties (with respect to $\alpha\in Q$) of $\mu_{\alpha}$
and giving a disintegration of $\mm\llcorner_{\T_{u}^{nb}}$ 
strongly consistent with respect to $\QQ$. Moreover, for every compact subset $K\subset X$, it holds
\begin{equation}\label{eq:boundmmalpha}
\frac{1}{\sup_{K} f } \mu_{\alpha}(K) \leq  \mm_{\alpha}(K) =   \frac{\mu_{\alpha}}{f} (K) \leq \frac{1}{\inf_{K} f }, \quad \text{for $\qq$-a.e. $\alpha\in Q$.}
\end{equation}

In the next statement, we  summarize what obtained so far concerning the disintegration 
of a $\sigma$-finite reference measure $\mm$ with respect to the non-branched
transport relation induced by any $1$-Lipschitz function $u : X \to \R$. \\
We denote by  $\mathcal{M}_{+}(X)$ the 
space of non-negative Radon measures over $X$. 

\begin{theorem}\label{T:sigma-disint}
 Let $(X,\sfd,\mm)$ be any geodesic and proper (hence complete) m.m.s. 
 with $\supp(\mm) = X$ and $\mm$ $\sigma$-finite. 
Then for any $1$-Lipschitz function $u : X \to \R$, the 
measure $\mm$ restricted to the non-branched transport set $\T_{u}^{nb}$ 
admits the following disintegration formula: 
$$
\mm\llcorner_{\T_{u}^{nb}} = \int_{Q} \mm_{\alpha} \, \qq(d\alpha),
$$
where $\qq$ is a Borel probability measure over $Q \subset X$ such that 
$\QQ_{\sharp}( \mm\llcorner_{\T_{u}^{nb}} ) \ll \qq$ and the map 
$Q \ni \alpha \mapsto \mm_{\alpha} \in \mathcal{M}_{+}(X)$ satisfies the following properties:
\begin{itemize}
\item[(1)] for any $\mm$-measurable set $B$, the map $\alpha \mapsto \mm_{\alpha}(B)$ is $\qq$-measurable; \smallskip
\item[(2)] for $\qq$-a.e. $\alpha \in Q$, $\mm_{\alpha}$ is concentrated on $\QQ^{-1}(\alpha) = R_{u}^{nb}(\alpha)$ (strong consistency); \smallskip
\item[(3)] for any $\mm$-measurable set $B$ and $\qq$-measurable set $C$, the following disintegration formula holds: 
$$
\mm(B \cap \QQ^{-1}(C)) = \int_{C} \mm_{\alpha}(B) \, \qq(d\alpha).
$$
\item[(4)]  For every compact subset $K\subset X$ there exists a constant $C_{K}\in (0,\infty)$ such that
$$
 \mm_{\alpha}(K) \leq C_{K}, \quad \text{for $\qq$-a.e. $\alpha\in Q$.}
$$
\end{itemize}
Moreover, fixed any $\qq$ as above such that $\QQ_{\sharp}( \mm\llcorner_{\T_{u}^{nb}} ) \ll \qq$, the disintegration is $\qq$-essentially unique (see above).
\end{theorem}

\subsection{Localization of Ricci bounds}\label{Ss:localsigma}

Under the additional assumption of a synthetic lower bound on the Ricci curvature, 
one can obtain regularity properties both on $\T_{u}^{nb}$ and on 
the conditional measures $\mm_{\alpha}$. 
As some of these results where obtained assuming $\mm(X) < \infty$, 
in what follows we review how to obtain the same regularity with no finiteness assumption on $\mm$. 
First of all recall that, for any $K\in \R$ and $N\in (1,\infty)$,  $\CD(K,N)$ implies $\MCP(K,N)$, which in turn implies that $\mm$ is $\sigma$-finite. Thus Theorem \ref{T:sigma-disint} can be applied.

\smallskip

\begin{lemma}\label{L:nullsetMCP}
Let $(X,\sfd,\mm)$ be an essentially non-branching m.m.s. with $\supp(\mm) = X$ and verifying $\MCP(K,N)$, for some $K\in \R, N\in (1,\infty)$. Then for any $1$-Lipschitz function $u : X \to \R$, it holds  $\mm(\T_u \setminus \T_u^{nb}) = 0$. 
\end{lemma}

Lemma \ref{L:nullsetMCP} has been proved in \cite{cava:MongeRCD} for 
metric measure spaces $(X,\sfd,\mm)$ verifying $\RCD(K,N)$ with $N<\infty$
and $\supp(\mm) = X$. The  $\RCD(K,N)$ assumption was used in that proof only to have at disposal the following property: 
given $\mu_{0},\mu_{1} \in\mathcal{P}(X)$ with $\mu_{0} \ll \mm$, 
there exists a unique optimal dynamical plan for the $W_{2}$-distance 
and it is induced by a map. 
In \cite[Theorem 1.1]{CM16}  this property is also verified for an e.n.b. metric measure space verifying $\MCP(K,N)$ with $\supp(\mm) = X$, without any finiteness assumption on $\mm$. 
Hence Lemma \ref{L:nullsetMCP} can be proved following verbatim \cite{cava:MongeRCD}.   
\\

Building on top of \cite{CM16}, in \cite[Theorem 7.10]{CMi16} an additional information 
on the transport rays was proved: for $\qq$-a.e. $\alpha \in Q$ it holds: 
\begin{equation}\label{E:maximalray}
R_u(\alpha) = \overline{R_u^{nb}(\alpha)} \supset R_u^{nb}(\alpha) \supset \mathring{R}_u(\alpha) ,
\end{equation}
with the latter to be  interpreted as the relative interior. 
The additional assumption of $\mm(X) < \infty$ was used in the proof only to 
obtain the existence of a disintegration of $\mm$ strongly consistent with the 
non-branched equivalence relation. 
Hence from Theorem \ref{T:sigma-disint} also \eqref{E:maximalray} is valid in the present framework.

To conclude, we assert that the localization results for the 
synthetic Ricci curvature lower bounds $\MCP(K,N)$ and $\CD(K,N)$, with $
K,N \in \R$ and $N >1$, are valid also in our framework. \begin{itemize}
\item \emph{Localization of $\MCP(K,N)$}.
In \cite[Theorem 9.5]{biacava:streconv},   assuming non-branching and the $\MCP(K,N)$ condition, it is proved (adopting a slightly different notation) that for 
$\qq$-a.e. $\alpha \in Q$ it holds
$$
\mm_{\alpha} = h_{\alpha} \, \H^{1}\llcorner_{X_{\alpha}}, 
$$
where $\H^{1}$ denotes the one-dimensional Hausdorff measure.
Moreover, the one-dimensional metric measure space $(\bar X_{\alpha}, \sfd,\mm_{\alpha})$,
isomorphic  to $([0,D_{\alpha}], |\cdot|, h_{\alpha} \L^{1})$, is proved to verify $\MCP(K,N)$; here $\bar X_{\alpha}$ stands for the closure of the transport ray $X_{\alpha}$ with respect to $\sfd$. Note that 
$\bar X_{\alpha}$ might not be a subset of $\mathcal{T}_{u}^{nb}$ because of its endpoints but this will not affect any argument as
$\mm_{\alpha}(\bar X_{\alpha} \setminus X_{\alpha}) = 0$.
No finiteness assumption was assumed in 
\cite[Theorem 9.5]{biacava:streconv} and, since here we  restrict to the non-branched transport set, the arguments can be carried over  to give the same statement.

\item  \emph{Localization of $\CD(K,N)$}.
 The localization of  $\CD(K,N)$ was  proved in \cite[Theorem 5.1]{CM16} under the assumption  $\mm(X)= 1$. Nevertheless, in \cite{CM16} the $\CD(K,N)$ condition was assumed to be valid only locally, i.e. 
the space was assumed to satisfy $\CD_{loc}(K,N)$. 
In particular the proof first shows that the one-dimensional metric measure space 
$(X_{\alpha},\sfd,\mm_{\alpha})$ verifies $\CD_{loc}(K,N)$ for $\qq$-a.e. $\alpha \in Q$ and then, thanks to the local-to-global property of one-dimensional $\CD(K,N)$ condition, concludes with the full claim. 
Hence, if $(X,\sfd,\mm)$ is e.n.b. and verifies $\CD(K,N)$, since by Theorem \ref{T:sigma-disint} a disintegration formula is at disposal and the reference measure $\mm$ is locally finite,  one can repeat the arguments in  \cite[Theorem 5.1]{CM16}  and obtain that  the one-dimensional metric measure space $(\bar X_{\alpha}, \sfd,\mm_{\alpha})$,
isomorphic  to $([0,D_{\alpha}], |\cdot|, h_{\alpha} \L^{1})$, verifies $\CD(K,N)$. 
\end{itemize}

We summarize the above discussion in the next statement.
\begin{theorem}\label{T:sigmadisint}
Let $(X,\sfd,\mm)$ be an essentially non-branching m.m.s. with $\supp(\mm) = X$ and satisfying $\MCP(K,N)$, for some $K\in \R, N\in (1,\infty)$. 
\\Then, for any $1$-Lipschitz function $u : X \to \R$, 
there exists a disintegration of $\mm$ strongly consistent with $R_{u}^{nb}$ verifying 
$$
\mm\llcorner_{\T_{u}^{nb}} = \int_{Q} \mm_{\alpha} \, \qq(d\alpha), \qquad \qq(Q) = 1.
$$
Moreover,  for $\qq$-a.e. $\alpha$, $\mm_{\alpha}$ is a Radon measure  with $\mm_{\alpha}=h_{\alpha} \cH^{1}\llcorner_{X_{\alpha}} \ll \cH^{1}\llcorner_{X_{\alpha}}$
and $(\bar X_{\alpha},\sfd,\mm_{\alpha})$ verifies $\MCP(K,N)$. 
\\If, additionally,  $(X,\sfd,\mm)$ satisfies $\CD_{loc}(K,N)$, then $h_{\alpha}$ is a $\CD(K,N)$ density on $X_{\alpha}$ for $\qq$-a.e. $\alpha$. 
\end{theorem}

It is worth recalling that, once we know that $(\bar  X_{\alpha},\sfd,\mm_{\alpha})$ verifies $\MCP(K,N)$, it is straightforward to get that   $\mm_{\alpha} = h_{\alpha} \H^{1}\llcorner_{X_{\alpha}}$
for some density $h_{\alpha}$. We refer to Section \ref{Ss:MCP1d}
for all the properties verified by one dimensional metric measure spaces verifying lower Ricci curvature bounds. 
\\We conclude the section by specialising the results to the smooth framework of Riemannian manifolds (cf. \cite{klartag}).
\begin{corollary}\label{C:sigmadisintSmooth}
Let $(M,g)$ be a complete $2\leq N$-dimensional Riemannian manifold, and let $\mm$ denote its Riemannian volume measure.   
\\Then, for any $1$-Lipschitz function $u : M \to \R$, 
there exists a disintegration of $\mm$ strongly consistent with $R_{u}^{nb}$ verifying 
$$
\mm\llcorner_{\T_{u}^{nb}} = \int_{Q} \mm_{\alpha} \, \qq(d\alpha), \qquad \qq(Q) = 1.
$$
Moreover, 
\begin{enumerate}
\item For $\qq$-a.e. $\alpha$, $\mm_{\alpha}$ is a Radon measure  with $\mm_{\alpha}=h_{\alpha} \cH^{1}\llcorner_{X_{\alpha}} \ll \cH^{1}\llcorner_{X_{\alpha}}$;
\item For every $x\in M$ there exist a (compact, geodesically convex) neighbourhood $U$ of $x$ and $\bar K\in \R$ such that $h_{\alpha}\llcorner_{U}$ is 
a $\CD(\bar K,N)$ density on $X_{\alpha}\cap U$, for $\qq$-a.e. $\alpha$;
\item If, additionally,  $\Ric_{g}\geq K g$, for some $K\in \R$, then $h_{\alpha}$ is a $\CD(K,N)$ density on $X_{\alpha}$ for $\qq$-a.e. $\alpha$. 
\end{enumerate}
\end{corollary}

\begin{proof}
The corollary follows directly from Theorem  \ref{T:sigma-disint} and Theorem \ref{T:sigmadisint} reasoning as follows. 
A complete Riemannian manifold is geodesic and proper hence 
Theorem  \ref{T:sigma-disint} implies the first part of the claim:
$$
\mm\llcorner_{\T_{u}^{nb}} = \int_{Q} \mm_{\alpha} \, \qq(d\alpha), \qquad \qq(Q) = 1,
$$
and for $\qq$-a.e. $\alpha$, $\mm_{\alpha}$ is a Radon measure  
with $\mm_{\alpha}=h_{\alpha} \cH^{1}\llcorner_{X_{\alpha}} \ll \cH^{1}\llcorner_{X_{\alpha}}$.

Moreover every point $x\in M$ admits a geodesically convex compact 
neighbourhood $U$ where, by compactness, the Ricci tensor is bounded below by some $\bar K\in \R$.  
In particular $(U,\sfd,\mm\llcorner_{U})$ is an essentially non-branching $\CD(\bar K,N)$ space and thus we can apply Theorem \ref{T:sigmadisint} to 
$(U,\sfd,\mm\llcorner_{U})$. 
Since the partition associated to $u : U \to \R$ is given by the restriction of 
transport rays, the quotient measure of $\mm$ restricted to $U\cap\T_{u}^{nb}$ will be absolutely continuous with respect to $\qq$; hence by $\qq$-essential
uniqueness of disintegration we deduce 
that $h_{\alpha}\llcorner_{U}$ is 
a $\CD(\bar K,N)$ density on $X_{\alpha}\cap U$, for $\qq$-a.e. $\alpha$.
Last point is already contained in Theorem \ref{T:sigmadisint}.
\end{proof}

\section{Representation formula for the Laplacian}

From now on we will  assume $(X,\sfd,\mm)$ to be an e.n.b. metric measure space verifying $\MCP(K,N)$, 
for some $K\in \R$ and $N\in (1,\infty)$.  In particular $(X,\sfd,\mm)$ is locally doubling \& Poincar\'e space (recall Remark \ref{rem:MCPDoubPoinc1}).

We will obtain an explicit representation formula for the Laplacian for a general $1$-Lipschitz function
$$
u : X \to \R, \qquad |u(x) - u(y)| \leq \sfd(x,y),
$$
assuming a mild regularity property on
$\T_{u}$, the associated transport set defined in Section \ref{S:transportset}.

A distinguished role will be played by a particular family of 1-Lipschitz functions, namely the so-called signed distance functions. Such a class played a key role in the recent proof  \cite{CMi16} of the local-to-global property of $\CD(K,N)$ under the e.n.b. assumption. 

\begin{definition}[Signed Distance Function]
Given a continuous function $v : (X,\sfd) \to \R$ so that $\set{v = 0} \neq \emptyset$, the function:
\begin{equation}\label{E:levelsets}
d_{v} : X \to \R, \qquad d_{v}(x) : = \sfd(x, \{ v = 0 \} )\; \sgn(v),
\end{equation}
is called the signed distance function (from the zero-level set of $v$).  
\end{definition}
With a slight abuse  of notation, we denote with $\sfd$ both the distance between points and the induced distance between sets; more precisely 
$$
 \sfd(x, \{ v = 0 \} ):=\inf \left\{  \sfd(x,y)\,:\, y\in  \{ v = 0 \} \right\}.
$$

\begin{lemma} \label{lem:df-Lip}
$d_v$ is $1$-Lipschitz on $\set{v \geq 0}$ and $\set{v\leq 0}$. If $(X,\sfd)$ is a length space, then $d_v$ is $1$-Lipschitz on the entire $X$. 
\end{lemma}

For the proof we refer to \cite[Lemma 8.4]{CMi16}.

\smallskip

We now fix once and for all a $1$-Lipschitz function $u :X \to \R$. In order not to  have empty statements, throughout the section we will assume that $\mm(\T_{u})>0$.

\subsection{Representing the gradient of $-u$}\label{Ss:representu}
The translation along $\T_{u}^{nb}$ is defined as:
$$
g : \R \times \T_{u}^{nb} \to  \T_{u}^{nb} \subset X, \qquad \gr (g) = \{ (t,x,y) \in  \R \times R^{nb}_{u} \colon u(x) - u(y) = t \}.
$$
Since $R^{nb}_{u}$ is Borel, the same applies to $g$, while $\dom (g)  = P_{12} (\gr (g))$ is analytic.
We will write $g_{t}$ for $g(t,\cdot)$. Notice that 
$$
\gr (g_{t}) = \{ (x,y) \in R^{nb}_{u} \colon u(x) - u(y) = t \}
$$
is Borel as well and thus for $t \in \R$,
$$
\dom (g_{t})= \T^{nb}_{u}(t):=\{x \in \T^{nb}_{u}\, :\, \exists y \in R^{nb}_{u}(x) \text{ with } u(x)-u(y)=t \}
$$
is an analytic set. The rough intuitive picture is of course that $g_{t}$ plays the role of negative gradient flow of $u$, restricted to the points of maximal slope 1.
In order to handle the case when $\mm(\T^{nb}_{u})=+\infty$, it is useful to introduce the following definition. 
\begin{definition}
A measurable subset $E \subset X$   is said \emph{$R^{nb}_{u}$-convex} 
if for any $x \in \T_{u}^{nb}$ the set $E \cap R_{u}^{nb}(x)$ is isometric to an interval.
\end{definition}

For every bounded $R^{nb}_{u}$-convex subset $E \subset \T^{nb}_{u}(2\ve)$ with $\mm(E)>0$,  consider the function $\Lambda : E  \to C([0,1];X)$ defined by
$$
[0,1] \ni \tau \mapsto \Lambda(x)_{\tau} : = 
\begin{cases}
g_{\tau}(x), & \tau \in [0,\ve], \\
g_{\ve}(x), & \tau \in [\ve, 1],
\end{cases}
$$
and set
\begin{equation}\label{E:planm} 
\pi_{E} : =\frac{1} {\mm(E)} \Lambda_{\sharp} \mm\llcorner_{E}.
\end{equation}
Note  that 
\begin{equation}\label{E:pi}
\mm(E) (\ee_{\tau})_{\sharp} \pi_{E}= (\ee_{\tau} \circ  \Lambda)_{\sharp} \mm\llcorner_{E}  = 
\begin{cases}
(g_{\tau})_{\sharp} \mm \llcorner_{E}=: \mm_{E}^{\tau}  & \tau \in [0,\ve], \\
(g_{\ve})_{\sharp} \mm \llcorner_{E}=: \mm_{E}^{\ve}  & \tau \in [\ve,1].
\end{cases}
\end{equation}
The rough intuitive idea is of course that $\mm_{E}^{\tau}$ is the push forward of $\mm \llcorner_{E} $ via the  negative gradient flow of  $u$ at time $\tau$.

\begin{proposition}\label{P:Krepresents}
Let $(X,\sfd,\mm)$ be an e.n.b. metric measure space verifying $\MCP(K,N)$ and $u$ be as before.
For every bounded $R^{nb}_{u}$-convex subset $E \subset \T^{nb}_{u}(2\ve)$ with $\mm(E)>0$,
the measure $\pi_{E}$ defined in \eqref{E:planm} is a test plan representing the gradient of $-u$ (see Definition \ref{D:plansgradient}).
\end{proposition}

\begin{proof}
Fix $t\in [0,\ve]$. First of all write
\begin{equation}\label{eq:etshPi}
\mm(E) \, (\ee_{t})_{\sharp}\pi_{E}=\mm_{E}^{t}= \int_{Q} (g_{t})_{\sharp}\mm_{\alpha} \llcorner_{E} \, \qq(d\alpha).
\end{equation}
Since $\mm_{\alpha}\llcorner_{E}=h_{\alpha} \, \cH^{1}\llcorner_{X_{\alpha}\cap E}$,  we have
\begin{equation}\label{eq:gtmalpha}
(g_{t})_{\sharp}\mm_{\alpha}\llcorner_{E} = \frac{h_{\alpha} \circ g_{-t}}{h_{\alpha}} \mm_{\alpha}\llcorner_{g_{t}(E)}.
\end{equation}
Identifying $X_{\alpha}\cap (\cup_{t\in [0,\ve]} g_{t}(E))$ with an  interval  $[a_{\alpha}, b_{\alpha}]\subset \R$ (for the sake of the argument we assume the interval to be closed, but all the other cases are completely analogous), from  \eqref{E:MCPdef2}, for $x \in [a_{\alpha}+t, b_{\alpha} - 2\ve +t]$ and $t \leq\ve$ it holds
\begin{equation}\label{eq:halphaMCPproof}
\frac{h_{\alpha} (x -t ) }{h_{\alpha} (x)} \leq \left[ \frac{s_{K/(N-1 )}(b_{\alpha} - x + t)  }{s_{K/(N-1 )}(b_{\alpha}-x)} \right]^{N-1} \leq C_{\ve}, \; \text{for all $x \in [a_{\alpha}+t, b_{\alpha} - 2\ve +t]$ and $t \leq\ve$}
\end{equation}
where the last inequality follows from the fact that $b_{\alpha}-x\geq 2 \ve-t\geq \ve>0$.  We stress that $C_{\ve}>0$ is independent of $\alpha\in Q$.
The combination of \eqref{eq:etshPi}, \eqref{eq:gtmalpha} and \eqref{eq:halphaMCPproof} gives that  $(\ee_{t})_{\sharp} \pi \leq \frac{C_{\ve}}{\mm(E)} \mm$ for all $t\in [0,1]$, i.e.  $\pi_{E}$ has bounded compression. 
Moreover since $\mm(E) < \infty$, and $|\dot \gamma|=1$  for $\pi$-a.e. $\gamma$, it follows that $\pi_{E}$ is a test plan (Definition \ref{D:testplan}). 

We now prove that $\pi_{E}$ represents the gradient of $-u$. Since by construction $u(x) - u(g_{\tau}(x))=\tau$ for $\mm\llcorner  E$-a.e. $x$, we have that
$$
\liminf_{\tau \to 0} \int  \frac{u(\gamma_{0}) - u(\gamma_{\tau})  }{\tau} \,\pi_{E}(d\gamma) = \frac{1}{\mm(E) }
\liminf_{\tau \to 0} \int_{E}  \frac{u(x) - u(g_{\tau}(x))  }{\tau} \,\mm(dx)  = 1.
$$
Hence the claim (recall Definition \ref{D:plansgradient}) follows by the fact that the  $1$-Lipschitz regularity of $u$ implies $|Du|\leq 1$ $\mm$-a.e. and thus
$$
1  \geq \frac{1}{2 \mm(E)}\int_{\T^{nb}_{u}(2\ve)} |D u|^{2}(x) \mm(dx)  + \frac{1}{2} \pi_{E}(C([0,1];X)).
$$

\end{proof}

In the next statement and in the rest of the paper, we will often consider the restriction of a  Lipschitz function $f$ to some transport ray $R_{u}^{nb}(\alpha)$ giving a real variable Lipschitz function: 
$[a_{\alpha},b_{\alpha}] \ni t\mapsto f( g (t,a_{\alpha}))$. 
It will make sense then to compute the $t$-derivative of the previous map: whenever it exists, we will use the following notation 
\begin{equation}\label{eq:deff'}
f'(x) : = \lim_{t \to 0} \frac{f(g(t,x)) -f(x) }{t}.
\end{equation}
Note that $f'$ is roughly the directional derivative of $f$ ``in the direction of $-\nabla u$''. 
Observe that, if $(X,\sfd, \mm)$ is $\MCP(K,N)$ e.n.b.,  for every $f\in \LIP(X)$ the quantity $f'$ is well defined $\mm$-a.e. on $\T_{u}$.

\begin{theorem}\label{T:Laplacerepresentation}
Let $(X,\sfd,\mm)$ be an e.n.b. metric measure space verifying $\MCP(K,N)$ and $u$ be as before. 
Then for any Lipschitz function $f : X \to \R$ it holds
\begin{equation}\label{E:LaplaceMCP}
D^{-}f (-\nabla u) \leq f' \leq D^{+}f (-\nabla u),  \quad \mm\text{-a.e. on } \T_{u}.
\end{equation}

\end{theorem}

\begin{proof}
Given $f \in \LIP(X)$, fix $\ve>0$ and let $E \subset \T^{nb}_{u}(2\ve)$  be any  bounded $R^{nb}_{u}$-convex subset with $\mm(E)>0$. 
Theorem \ref{T:verticalhorizontal} together with Proposition \ref{P:Krepresents} and equation \eqref{E:pi} implies that 
\begin{align*}
\int_{E}& D^{-} f (-\nabla u) \, \mm \\
&~ \leq \liminf_{\tau \to  0} \int_{E} \frac{f(g_{\tau}(x)) - f(x) }{\tau} \, \mm (dx)  \\
&~ \leq \limsup_{\tau \to  0} \int_{E} \frac{f(g_{\tau}(x)) - f(x) }{\tau} \, \mm (dx)  \\
&~\leq \int_{E} D^{+} f (-\nabla u) \, \mm.
\end{align*}
To conclude it is enough to observe that 
$$
 \int_{E} \frac{f(g_{\tau}(x)) - f(x) }{\tau} \, \mm (dx)  =  \int_{Q}\int_{E \cap X_{\alpha}} \frac{f(g_{\tau}(x)) - f(x) }{\tau} \, \mm_{\alpha} (dx) \,\qq(d\alpha),
$$
and notice that for each $\alpha \in Q$ the incremental ratio $(f(g_{\tau}(x))) - f(x))/\tau$ converges to $f'(x)$  for $\mm_{\alpha}$-a.e. $x \in X_{\alpha}$ and is dominated by the Lipschitz constant of $f$.
Therefore, by Dominated Convergence Theorem, for each $E$ as above it holds
$$
\int_{E} D^{-} f (-\nabla u) \, \mm \leq \int_{E} f'\,\mm \leq \int_{E} D^{+} f (-\nabla u) \, \mm. 
$$
The claim follows by the arbitrariness of $\ve>0$ and  $E \subset \T^{nb}_{u}(2\ve)$.
\end{proof}

Chain rule \cite[Proposition 3.15]{Gigli12}  combined with Theorem \ref{T:Laplacerepresentation} permits to obtain the next

\begin{corollary}\label{C:squared}
Let $(X,\sfd,\mm)$ be an e.n.b. metric measure space verifying $\MCP(K,N)$ and $u$ as before. 
Then for any Lipschitz function $f : X \to \R$  
$$
D^{-}f (-\nabla u^{2}) \leq 2 u f' \leq D^{+}f (-\nabla u^{2}),
$$
where the inequalities hold true $\mm$-a.e. over $\T_{u}$.
\end{corollary}

\begin{proof}
We show that $D^{-}f (-\nabla u^{2}) \leq 2 u f'$, the argument for proving  $2 u f' \leq D^{+}f (-\nabla u^{2})$ is completely analogous.
\\By chain rule  \cite[Proposition 3.15]{Gigli12}, we know that 
$$D^{-}f (-\nabla u^{2})= 2u D^{-\sgn u}f (-\nabla u). $$
Combining the last identity with Theorem \ref{T:Laplacerepresentation} yields
\begin{equation*}
D^{-}f (-\nabla u^{2})= \begin{cases}
2u \,D^{+}f (-\nabla u) \leq 2uf' \quad \text{$\mm$-a.e. on  } \{u\leq 0\} \\
2u \,D^{-}f (-\nabla u) \leq 2uf' \quad \text{$\mm$-a.e. on  } \{u\geq 0\},
\end{cases}
\end{equation*}
giving the claim.
\end{proof}

\subsection{A formula for the Laplacian of a general $1$-Lipschitz function}

The next proposition, which is key to show that $\Delta u$ is a Radon functional, follows from Lemma \ref{lem:apriori0}  and Lemma \ref{L:integrabilityh}. We use the notation that $a(X_{\alpha})$ (respectively $b(X_{\alpha})$) denotes the initial (respectively the final) point of the transport ray $X_{\alpha}$. Recall also that $h_{\alpha}$ is positive and differentiable a.e. on $X_{\alpha}$, in particular $\log h_{\alpha}$ is  well defined and differentiable a.e. along $X_{\alpha}$.

\begin{proposition}\label{P:Radon}
Let $(X,\sfd,\mm)$ be an e.n.b. metric measure space verifying $\MCP(K,N)$, for some $K\in \R, N\in (1,\infty)$.   Let $u : X \to \R$ be a $1$-Lipschitz function 
with associated disintegration $
\mm\llcorner_{\T_{u}^{nb}} = \int_{Q} \mm_{\alpha} \, \qq(d\alpha),$ with $\qq(Q) = 1$,  $\mm_{\alpha} = h_{\alpha} \H^{1}\llcorner_{X_{\alpha}}$,  $h_{\alpha}\in L^{1}(\H^{1}\llcorner_{X_{\alpha}})$ for $\qq$-a.e. $\alpha\in Q$.
Assume  that 
$$
\int_{Q} \frac{1}{\sfd(a(X_{\alpha}),b(X_{\alpha}))} \, \qq(d\alpha) < \infty.
$$
Then $T_{u}:\LIP_{c}(X)\to \R$
\begin{equation}\label{eq:defMeas}
T_{u}(f):=\int_{Q} \int_{X_{\alpha}}  (\log h_{\alpha})'f \, \mm_{\alpha} \qq(d\alpha) +\int_{Q} (h_{\alpha}f)(a(X_{\alpha})) - (h_{\alpha} f) (b(X_{\alpha}))\, \qq(d\alpha)
\end{equation}
is a Radon functional over $X$. 
\end{proposition}

\begin{proof}
Fix any bounded open subset $W\subset X$ and  observe that  we can find a bounded  $R^{nb}_{u}$-convex measurable subset $E \subset \T^{nb}_{u}$ such that $W\cap \T^{nb}_{u}\subset E$ 
(take for instance on each $X_{\alpha}$ the convex-hull of $W\cap X_{\alpha}$) and 
\begin{equation}\label{eq:lbXaE}
\sfd(a(X_{\alpha}\cap E),b(X_{\alpha}\cap E)) \geq \min\{1,  \sfd(a(X_{\alpha}),b(X_{\alpha}))\}, \quad \text{for all } \alpha \in Q. 
\end{equation}
Note that $E$ depends just on $W$ and the ray relation $R^{nb}_{u}$.
For any  $f\in \LIP_{c}(X)$ with  $\supp(f)\subset W$, it is clear that 
\begin{align*}
\int_{Q} \int_{X_{\alpha}}& (\log h_{\alpha})' \, f\, \mm_{\alpha} \qq(d\alpha) +\int_{Q}  (h_{\alpha} f)(a(X_{\alpha})) - (h_{\alpha} f)(b(X_{\alpha})) \, \qq(d\alpha) \nonumber \\
&=\int_{Q} \int_{X_{\alpha}\cap E} (\log h_{\alpha})' \, f\, \mm_{\alpha} \qq(d\alpha) +\int_{Q}  (h_{\alpha} f)(a(X_{\alpha}\cap E)) - (h_{\alpha} f)(b(X_{\alpha}\cap E)) \, \qq(d\alpha) .
\end{align*}
Since $E$ is bounded, we have $\sup_{\alpha\in Q} \sfd\left(a(X_{\alpha}\cap E), b(X_{\alpha}\cap E) \right)\leq C_{W}$ for some $C_{W}\in (0,\infty)$ depending only on $W\subset X$.
Moreover, Theorem \ref{T:sigma-disint}(4) implies $\sup_{\alpha\in Q} \int_{X_{\alpha}\cap E} h_{\alpha}\, d\cH^{1}\leq C_{W}$.
 Therefore, applying  Lemma \ref{lem:apriori0}  and Lemma \ref{L:integrabilityh} to the renormalized densities $\tilde{h}_{\alpha}:= \frac{1}{\int_{X_{\alpha}\cap E} h_{\alpha} \, d\cH^{1}} h_{\alpha}$ and rescaling back to get $h_{\alpha}$, recalling also \eqref{eq:lbXaE} we infer
$$
\sup_{X_{\alpha}\cap E} h_{\alpha}(x) + \int_{X_{\alpha}\cap E} |h'_{\alpha}| \, d\cH^{1} \leq C_{W} \frac{1}{\sfd(a(X_{\alpha}),b(X_{\alpha})}, \quad \text{for $\qq$-a.e. }\alpha \in \QQ(E)\subset Q. 
$$
We can thus estimate
\begin{align*}
\Big|\int_{Q} &  \int_{X_{\alpha}\cap E} (\log h_{\alpha})' \, f\, \mm_{\alpha} \qq(d\alpha) +\int_{Q} (h_{\alpha} f)(a(X_{\alpha}\cap E)) - (h_{\alpha} f)(b(X_{\alpha}\cap E)) \, \qq(d\alpha) \Big| \\
&\leq  \left(  C_{W} \int_{Q} \frac{1}{\sfd(a(X_{\alpha}),b(X_{\alpha}))} \qq(d\alpha) \right) \; \max |f|. 
\end{align*}
We can thus conclude that \eqref{eq:defMeas} defines a Radon functional.
\end{proof}

The first main result  follows by combining   Theorem \ref{T:Laplacerepresentation} and Proposition \ref{P:Radon}.

\begin{theorem}\label{T:main1}
Let $(X,\sfd,\mm)$ be an e.n.b. metric measure space with $\supp(\mm) = X$ and verifying $\MCP(K,N)$ for some $K\in \R, N\in (1,\infty)$.
 Let $u : X \to \R$ be a $1$-Lipschitz function 
with associated disintegration $
\mm\llcorner_{\T_{u}^{nb}} = \int_{Q} \mm_{\alpha} \, \qq(d\alpha),$ with $\qq(Q) = 1$,  $\mm_{\alpha} = h_{\alpha} \H^{1}\llcorner_{X_{\alpha}}$,  $h_{\alpha}\in L^{1}(\H^{1}\llcorner_{X_{\alpha}})$ for $\qq$-a.e. $\alpha\in Q$.
Assume  that 
$$
\int_{Q} \frac{1}{\sfd(a(X_{\alpha}),b(X_{\alpha}))} \, \qq(d\alpha) < \infty.
$$
Then, for any open subset  $U\subset X$ such that $\mm(U \setminus \T_{u}) = 0$,  it holds $u \in D(\bold{\Delta},U)$. 
More precisely, $T_{U}:\LIP_{c}(U)\to \R$ defined by
$$
T_{U}(f) := - \int_{Q} f\, h_{\alpha}' \H^{1}\llcorner_{X_{\alpha}\cap U}\, \qq(d\alpha) + \int_{Q} (f\, h_{\alpha}) (b(X_{\alpha})) -  (f\, h_{\alpha}) (a(X_{\alpha})) \qq(d\alpha),
$$
is a Radon functional with $T_{U}\in \bold{\Delta}u\llcorner_{U}$. 
Moreover, writing $T_{U}=T_{U}^{reg}+T_{U}^{sing}$ with 
$$
T_{U}^{reg}(f) := - \int_{Q} f\, h_{\alpha}' \H^{1}\llcorner_{X_{\alpha}\cap U}\, \qq(d\alpha), \quad T_{U}^{sing}(f):= \int_{Q} (f\, h_{\alpha}) (b(X_{\alpha})) -  (f\, h_{\alpha}) (a(X_{\alpha})) \qq(d\alpha),
$$
it holds that $T_{U}^{reg}$ can be represented by  $T_{U}^{reg}= -(\log h_{\alpha})' \mm\llcorner_{U}$ and satisfies the bounds: 
\begin{equation}\label{eq:logh'}
- (N-1) \frac{s_{K/(N-1)}' (\sfd(b(X_{\alpha}),x))}{s_{K/(N-1)}(\sfd(b(X_{\alpha}),x))}  \leq (\log h_{\alpha})'(x) \leq  (N-1) \frac{s_{K/(N-1)}' (\sfd(x,a(X_{\alpha})))}{s_{K/(N-1)}(\sfd(x,a(X_{\alpha})))}.
\end{equation}
\end{theorem}

\begin{remark}[Interpretation in case $X_{\alpha}$ is unbounded]
Let us explicitly note that, in case the ray $X_{\alpha}$ is isometric to $(-\infty, b)$ (respectively $(a, +\infty)$), then by definition $(fh_{\alpha}) (a(X_{\alpha}))=0$ (resp. $(fh_{\alpha}) (b(X_{\alpha}))=0$).
Let us discuss the case $K=-(N-1)$, the other cases being analogous. In case the ray $X_{\alpha}$ is isometric to $(-\infty, b)$ (respectively $(a, +\infty)$), then the upper bound (resp. the lower bound) in \eqref{eq:logh'}  should be interpreted as  $(\log h_{\alpha})' \leq N-1 $ (resp.  $(\log h_{\alpha})' \geq -(N-1)$).
In particular,   if for $\qq$-a.e. $\alpha\in Q$ the ray $X_{\alpha}$ is isometric to $(-\infty, +\infty)$, then for any open subset  $U\subset X$ with $\mm(U \setminus \T_{u}) = 0$ the singular part $T_{U}^{sing}$ vanishes  and it holds
$-(N-1)\mm\llcorner_{U} \leq T_{U}^{reg}\leq (N-1) \mm\llcorner_{U}$. 
\end{remark}

\begin{proof}[Proof of Theorem \ref{T:main1}]
Fix an arbitrary open subset  $U\subset X$ such that $\mm(U \setminus \T_{u}) = 0$. Let  $f : X \to \R$ be any  Lipschitz function compactly supported in $U$  and let $f'$ be defined $\mm$-a.e. by \eqref{eq:deff'}. 
Recall that  the closure of the transport ray $(\bar X_{\alpha}, \sfd, \mm_{\alpha})$ is isomorphic  to a (possibly unbounded, possibly not open) real interval $[a(X_{\alpha}),b(X_{\alpha})]$  endowed with the weighted measure $h_{\alpha} \mathcal{L}^{1}$, so we can integrate by parts Lipschitz functions on $X_{\alpha}$ analogously as on  a weighted real interval.
\\Via an integration by parts, we thus obtain 
\begin{align*}
\int_{X_{\alpha}} &h_{\alpha}(x) \, f'(x)  \,\cH^{1}(dx)  \\
 = &~ - \int_{X_{\alpha}}  h'_{\alpha} (x) \,f(x) \,\cH^{1}(dx)  + (h_{\alpha}f) (b(X_{\alpha}))  - (h_{\alpha}f)(a(X_{\alpha})), \; \qq\text{-a.e. }\alpha,
\end{align*}
which, together with Theorem \ref{T:sigmadisint}, gives
\begin{align}
\int_{U} & f'(x)  \,\mm(dx) \nonumber \\
 = &~ - \int_{Q}\int_{X_{\alpha}}  h'_{\alpha} (x) \,f(x) \,\cH^{1}(dx)  + (h_{\alpha}f) (b(X_{\alpha}))  - (h_{\alpha}f)(a(X_{\alpha}))\, \qq(d\alpha). \label{eq:f'intBP}
\end{align}
Proposition \ref{P:Radon} ensures that, under the assumptions of Theorem \ref{T:main1},  the expression
$$
T_{\Delta u} (f) : = - \int_{Q} f\,  h'_{\alpha}\cH^{1}\llcorner_{X_{\alpha}} \,\qq(d\alpha) + \int_{Q} (h_{\alpha}f)(b(X_{\alpha})) - (h_{\alpha}f)(a(X_{\alpha}))  \,\qq(d\alpha) 
$$
defines a  Radon functional on $U$.   
\\The combination of \eqref{eq:f'intBP} with Theorem \ref{T:Laplacerepresentation} gives that   
$$
\int_{U} D^{-}f (-\nabla u) \, \mm \leq  T_{\Delta u} (f) \leq \int_{U}D^{+}f (-\nabla u) \, \mm. 
$$
Noting that (see \cite[Proposition 3.15]{Gigli12}) %
$$
D^{-}f(-\nabla u) = - D^{+}f(\nabla u), \qquad  D^{+}f(-\nabla u) = - D^{-}f(\nabla u), \quad \text{$\mm$-a.e.,}
$$
the previous inequalities imply 
$$
 \int_{U} D^{-}f (\nabla u) \, \mm \leq -   T_{\Delta u} (f)  \leq - \int_{U}D^{+}f (\nabla u) \, \mm.
$$
Recalling \eqref{E:logder}, the proof of all the claims is complete.
\end{proof}

The next result, dealing with smooth Riemannian manifolds, can be proved using  Corollary \ref{C:sigmadisintSmooth}  in the proof of Theorem \ref{T:main1}  and following verbatim the arguments. Let us just mention that the Laplacian here is single valued, i.e.  $\{T_{U}\}= \bold{\Delta}u\llcorner_{U}$, since on a smooth Riemannian manifold $(M,g)$ it holds $D^{+}f(\nabla u)=D^{-}f(\nabla u)= g(\nabla f, \nabla u)$.
\begin{corollary}\label{C:main1Smooth}
Let $(M,g)$ be a $2\leq N$-dimensional  complete Riemannian manifold.  Let $u : M \to \R$ be a $1$-Lipschitz function 
with associated disintegration $
\mm\llcorner_{\T_{u}} = \int_{Q} \mm_{\alpha} \, \qq(d\alpha),$ with $\qq(Q) = 1$,  $\mm_{\alpha} = h_{\alpha} \H^{1}\llcorner_{X_{\alpha}}$,  $h_{\alpha}\in L^{1}(\H^{1}\llcorner_{X_{\alpha}})$ for $\qq$-a.e. $\alpha\in Q$.
Assume  that 
$$
\int_{Q} \frac{1}{\sfd(a(X_{\alpha}),b(X_{\alpha}))} \, \qq(d\alpha) < \infty.
$$
Then, for any open subset  $U\subset M$ such that $\mm(U \setminus \T_{u}) = 0$,  it holds $u \in D(\bold{\Delta},U)$. 
More precisely, $T_{U}:\LIP_{c}(U)\to \R$ defined by
$$
T_{U}(f) := - \int_{Q} f\, h_{\alpha}' \H^{1}\llcorner_{X_{\alpha}\cap U}\, \qq(d\alpha) + \int_{Q} (f\, h_{\alpha}) (b(X_{\alpha})) -  (f\, h_{\alpha}) (a(X_{\alpha})) \qq(d\alpha),
$$
is a Radon functional with $\{T_{U}\}= \bold{\Delta}u\llcorner_{U}$. 
Moreover, writing $T_{U}=T_{U}^{reg}+T_{U}^{sing}$ with 
$$
T_{U}^{reg}(f) := - \int_{Q} f\, h_{\alpha}' \H^{1}\llcorner_{X_{\alpha}\cap U}\, \qq(d\alpha), \quad T_{U}^{sing}(f):= \int_{Q} (f\, h_{\alpha}) (b(X_{\alpha})) -  (f\, h_{\alpha}) (a(X_{\alpha})) \qq(d\alpha),
$$
it holds that $T_{U}^{reg}$ can be represented by  $T_{U}^{reg}= -(\log h_{\alpha})' \mm\llcorner_{U}$.
\\In addition, if  $\Ric_{g}\geq Kg$ for some $K\in \R$, then the following bounds hold: 
\begin{equation}\label{eq:logh'smooth}
- (N-1) \frac{s_{K/(N-1)}' (\sfd(b(X_{\alpha}),x))}{s_{K/(N-1)}(\sfd(b(X_{\alpha}),x))}  \leq (\log h_{\alpha})'(x) \leq  (N-1) \frac{s_{K/(N-1)}' (\sfd(x,a(X_{\alpha})))}{s_{K/(N-1)}(\sfd(x,a(X_{\alpha})))}.
\end{equation}
\end{corollary}
Specialising  Corollary \ref{C:main1Smooth} to the distance function gives Theorem \ref{thm:DeltadSmoothGen}, we briefly discuss the details below.
\\

\begin{proof}[Proof of Theorem \ref{thm:DeltadSmoothGen}]
Fix $p\in M$.

\textbf{Step 1}.  $u:=\sfd_{p}:=\sfd(p,\cdot)$  satisfies the assumptions of Corollary  \ref{C:main1Smooth}.
\\ Since by hypothesis   $(M,g)$ is complete, any point $x\in M$ can be joined to $p$ with a length minimising geodesic. Thus $\T_{\sfd_{p}}=M$, $b(X_{\alpha})=p$ and $a(X_{\alpha})\in {\mathcal C}_{p}$ for every $\alpha\in Q$.
Moreover, there exists $\epsilon=\epsilon(p)>0$ such that all the minimising geodesics $X_{\alpha}$ emanating from $p$ have length $\sfd(a(X_{\alpha}), b(X_{\alpha}))=\sfd(a(X_{\alpha}), p)\geq \epsilon$. Since by construction $\qq(Q)=1$, we conclude that  the assumptions of Corollary  \ref{C:main1Smooth} are verified.
\\

\textbf{Step 2}. The representation formula \eqref{eq:repdeltadpSmooth} holds.
We are left to show that 
$$
\int_{Q}  h_{\alpha} \, \delta_{b(X_{\alpha})} \,\qq(d\alpha)=0.
$$
Clearly, it is enough to show that 
\begin{equation}\label{eq:hap=0}
h_{\alpha}(p)=0, \quad \text{for $\qq$-a.e. $\alpha\in Q$.}
\end{equation}
Suppose by contradiction that there exists $\bar Q \subset Q$ where $h_{\alpha}(p) \geq c > 0$, with $\qq(\bar Q )> 0$. 
For simplicity of notation, we identify the minimising geodesic   $X_{\alpha}$ with the real interval $[a_{\alpha}, b_{\alpha}]$, where $p$ corresponds to $b_{\alpha}$.
Then by Fatou's Lemma it holds:
\begin{align*}
\infty&>\omega_{N}=\liminf_{r\downarrow 0} \frac{\mm(B_{r}(p))}{r^{N}}  \geq \liminf_{r\downarrow 0} \int_{\bar Q} \frac{1}{r^{N}}\int_{[b_{\alpha}-r,b_{\alpha}]}h_{\alpha}(t) dt \, \qq(d\alpha) \\
&\geq \int_{\bar Q} \liminf_{r\downarrow 0} \frac{1}{r} \int_{[b_{\alpha}-r,b_{\alpha}]}\frac{h_{\alpha}(t)}{r^{N-1}} dt \, \qq(d\alpha)  = \infty,
\end{align*}
giving a contradiction and thus proving the claim \eqref{eq:hap=0}.
\\

\textbf{Step 3}. $\left[\Delta \sfd_{p}\right]_{reg}^{\pm}:= - [(\log h_{\alpha})' ]^{\pm}\, \mm$,  $\left[\Delta \sfd_{p}\right]_{sing}  :=
- \int_{Q} h_{\alpha} \,\delta_{a(X_{\alpha})} \,\qq(d\alpha)$ define three non-negative Radon measures, and $\Delta \sfd_{p}= \left[\Delta \sfd_{p}\right]_{reg}^{+}- \left[\Delta \sfd_{p}\right]_{reg}^{-}+\left[\Delta \sfd_{p}\right]_{sing}$. 
\\Combining ``2.'' of Corollary \ref{C:sigmadisintSmooth} with  \eqref{E:logder}, it follows  that $\left[\Delta \sfd_{p}\right]_{reg}:= - (\log h_{\alpha})' \, \mm$ defines a Radon functional; by Riesz Theorem, its positive and negative parts are thus Radon measures. Also $\left[\Delta \sfd_{p}\right]_{sing}  :=
- \int_{Q} h_{\alpha} \,\delta_{a(X_{\alpha})} \,\qq(d\alpha)= \Delta \sfd_{p}- \left[\Delta \sfd_{p}\right]_{reg}$  is a non-positive Radon functional (as difference of Radon functionals) and thus, by Riesz Theorem, it defines a Radon measure.
\\

\textbf{Step 4}. Upper and lower bounds in case $\Ric_{g}\geq Kg$, for some $K\in \R$.
\\If $\Ric_{g}\geq Kg$, by ``3.'' of Corollary \ref{C:sigmadisintSmooth}  we know that $h_{\alpha}$ is a $\CD(K,N)$ (and in particular $\MCP(K,N)$) density over $X_{\alpha}$ for $\qq$-a.e. $\alpha$. Thus \eqref{E:logder} gives the bounds:
\begin{equation}\label{eq:logh'smoothproof}
- (N-1) \frac{s_{K/(N-1)}' (\sfd_{a(X_{\alpha})})} {s_{K/(N-1)}(\sfd_{a(X_{\alpha})})} \mm \leq  \left[\Delta \sfd_{p}\right]_{reg}  \leq(N-1) \frac{s_{K/(N-1)}' (\sfd_{p})}{s_{K/(N-1)}(\sfd_{p})}  \mm.
\end{equation}

\end{proof}

\begin{remark}[On the bounds under the assumption $\Ric_{g}\geq K g$]
Few comments are in order. 
\begin{itemize}
\item \emph{The upper bound} in \eqref{eq:logh'smoothproof} is the celebrated Laplacian comparison Theorem. Note that, a similar upper bound is proved above to hold more generally for the (regular part of the) Laplacian of a  (rather) general $1$-Lipschitz function \eqref{eq:logh'smooth} in the high generality of e.n.b. $\MCP(K,N)$-spaces  \eqref{eq:logh'}.

\item \emph{The case of the round sphere}. Let $p,q\in {\mathbb S}^{N}$ be a couple of antipodal points; clearly the cut locus of $p$ coincides with $q$. In this case, choosing $u=\sfd_{p}$ in the construction above gives the partition of ${\mathbb S}^{N}\setminus\{p,q\}$ into meridians, and each ray is a meridian without its endpoints $p,q$, oriented from $q$ to $p$.  Theorem \ref{thm:DeltadSmoothGen} thus yields
$$
- (N-1) \cot \sfd_{q}
\leq \Delta \sfd_{p} \leq (N-1) \cot \sfd_{p}, \quad \text{ on } {\mathbb S}^{N}.
$$
Note that (for the round sphere) the same conclusion could be achieved by applying the Laplacian comparison Theorem to $\sfd_{p}$ and to $\sfd_{q}$, and using that $\sfd_{p}=\pi-\sfd_{q}$.

\item \emph{The lower bound for a smooth Riemannian manifold}. Arguing analogously to the spherical case, one can achieve the lower bound  along a (minimising) geodesic $\gamma:[0,1]\to M$ with $(M,g)$ satisfying $\Ric_{g}\geq (N-1)g$   (cf. \cite[Lemma 3.2]{CN}). In this case,  the function $x\mapsto \sfd_{\gamma_{0}}(x)+\sfd_{\gamma_{1}}(x)$ achieves its minimum $\sfd(\gamma_{0}, \gamma_{1})$ along $\gamma([0,1])$; thus $\Delta  (\sfd_{\gamma_{0}}+\sfd_{\gamma_{1}})\geq 0$ along $\gamma((0,1))$ and, applying the upper bound \eqref{eq:UBLSmoothLoc} to $\sfd_{\gamma_{0}}$, $\sfd_{\gamma_{1}}$ and exploiting the linearity of the Laplacian we get 
\begin{equation}\label{eq:ULBDELTAdpHeur}
- (N-1) \cot \sfd_{\gamma_{1}}
\leq \Delta \sfd_{\gamma_{0}} \leq (N-1) \cot \sfd_{\gamma_{0}}, \; \text{ along } \gamma((0,1)).
\end{equation}
By ``glueing''  all the inequalities \eqref{eq:ULBDELTAdpHeur} corresponding to all the (minimising) geodesics emanating from $p$, gives \eqref{eq:logh'smoothproof}. Clearly this argument holds for smooth Riemannian manifolds, but in situations where the space is a-priori not smooth and the Laplacian is a-priori not linear (as for e.n.b. $\MCP(K,N)$-spaces), one has to argue differently. As the reader could already appreciate (see e.g. the proof of Theorem  \ref{thm:DeltadSmoothGen}), we attacked the problem by using techniques from $L^{1}$-optimal transport.
\end{itemize}
\end{remark}

A crucial fact in order to apply  Theorem  \ref{T:main1} to the distance function in the smooth case was that the cut locus of a point $p$ is at strictly positive distance from $p$. This fact is clearly not at disposal in the general setting of an e.n.b. $\MCP(K,N)$ space (e.g. the boundary of a convex body  in $\R^{3}$ whose cut locus is dense). In the next Subsection \ref{SS:LapDp} we will thus argue differently, showing first the result for the distance squared, and then getting the claim for the distance via chain rule.

\subsection{A formula for the Laplacian of a signed distance function}\label{SS:LapDp}

The goal of the subsection is to prove the existence of the Laplacian of $d_{v}$ and  $d^{2}_{v}$ as Radon measures and   to show upper and lower bounds; let us stress that, contrary to the previous subsection, here there will be no integrability assumption  on the reciprocal of the  length of the transport rays. 
 
Recall that given a continuous function 
$v : (X,\sfd) \to \R$ so that $\{ v = 0 \} \neq \emptyset$, the signed distance function 
$$
d_{v} : X \to \R, \qquad  d_{v}(x) : = 
\sfd(x, \{ v =0 \})\; \sgn(v),
$$
is $1$-Lipschitz. 
\\Notice also that since $(X,\sfd)$ is proper, $\mathcal{T}_{d_{v}} \supset X \setminus \{v=0\}$. 
Indeed, given $x \in X \setminus \{v =0 \}$, consider the distance minimising $z \in \{ v= 0 \}$ (whose existence is guaranteed by the  compactness of closed bounded sets). Then $(x,z) \in R_{d_{v}}$ and thus  $x \in \mathcal{T}_{d_{v}}$,  as $x \neq z$.  The next remark follows.

\begin{remark}\label{R:distancesign}
Let $X_{\alpha}$ be any transport ray associated with $d_{v}$ and let $a(X_{\alpha}), b(X_{\alpha})$ 
be its starting and the final point, respectively. 
Then 
$$
d_{v}(b(X_{\alpha})) \leq 0, \qquad
d_{v}(a(X_{\alpha})) \geq 0,
$$
whenever $b(X_{\alpha})$ and $a(X_{\alpha})$ exist.
\end{remark}

The next Lemma will be key to show the existence of the Laplacian  of $d^{2}_{v}$ as a Radon measure.

\begin{lemma}\label{lem:nucases}
Let $(X,\sfd,\mm)$ be an e.n.b. metric measure space verifying $\MCP(K,N)$, for some $K\in \R, N\in (1,\infty)$.
\item The expression
\begin{align}\label{eq:defnu}
\nu :  =  & 2  \left(1+ \sfd(\{ v= 0 \},x) (N-1)  \frac{s_{K/(N-1)}' (\sfd_{b(X_{\alpha})}(x))}{s_{K/(N-1)}(\sfd_{b(X_{\alpha})}(x))} \right) \, \mm\llcorner_{\{v\geq 0\}}
\nonumber\\
& +
 2  \left(1+ \sfd(\{ v= 0 \},x)  (N-1) \frac{s_{K/(N-1)}' (\sfd_{a(X_{\alpha})}(x))}{s_{K/(N-1)}(\sfd_{a(X_{\alpha})}(x))} \right) \, \mm\llcorner_{\{v< 0\}}
\end{align}
defines a signed Radon measure over $X$. More precisely:
\begin{itemize}
\item Case $K>0$: $\nu$ is a signed finite measure on $X$ satisfying $\nu\leq C_{K,N}\, \mm$.  \\
Moreover:
\begin{itemize}
\item If 
$$
\sup_{x\in \{v\geq 0\}}\sfd(x, b(X_{\alpha}))<\pi\sqrt{(N-1)/K}, \qquad 
\sup_{x\in \{v<0\}}\sfd(x, a(X_{\alpha}))<\pi\sqrt{(N-1)/K}
$$ 
then $\nu$ has density bounded in $L^{\infty}(X,\mm)$;
\item If 
$$
\sup_{x\in \{v\geq 0\}} 
\sfd(x, b(X_{\alpha}))=\pi\sqrt{(N-1)/K} \quad 
{\textrm or} \quad 
\sup_{x\in \{v< 0\}} 
\sfd(x, a(X_{\alpha}))=\pi\sqrt{(N-1)/K}
$$ 
then there exist exactly two points $\bar{a},\bar{b}\in X$ with $\sfd(\bar{a},\bar{b})=\pi\sqrt{(N-1)/K}$ such that 
for $\qq$-a.e.  $\alpha$
$$
a(X_{\alpha}) = \bar{a}, \quad  b(X_{\alpha}) = \bar{b},
$$
$\nu$ has density bounded in 
$$
L_{loc}^{\infty}(\{v \geq 0\}\setminus\{\bar a\},\mm)\cap
L_{loc}^{\infty}(\{ v\leq 0 \}\setminus\{\bar b\},\mm)\cap L^{1}(X,\mm).
$$ 
Moreover, in this case $(X,\mm)$ 
is isomorphic to a spherical suspension as a measure space. If in addition $(X,\sfd,\mm)$ is an $\RCD(K,N)$ space, then $(X,\sfd,\mm)$ is isomorphic to a spherical suspension as a metric measure space.
\end{itemize}
\item Case $K=0$: $\nu=2\mm$ is a non-negative Radon measure; if $b(X_{\alpha})$ or $a(X_{\alpha})$ do not exists, the two ratios in \eqref{eq:defnu} are posed by definition equal to $0$, respectively
\item Case $K\leq 0$: $\nu$ is a non-negative Radon measure.
if $b(X_{\alpha})$ or $a(X_{\alpha})$ do not exists, the two ratios in \eqref{eq:defnu} are posed by definition equal to $1$,
respectively.
\end{itemize}
\end{lemma}

\begin{proof}
$\bullet$ For  $K = 0$ the bounds are straightforward consequence of the definition of the coefficients $s_{K/(N-1)}$ given in \eqref{E:sk}.

$\bullet$ For  $K<0$ observe that, since $(0,\infty)\ni t \mapsto \coth t\in (0,\infty)$ is decreasing and $\sfd(\{ v= 0 \},x) \leq \sfd_{b(X_{\alpha})}(x)$ for all $x\in \{ v \leq 0 \}$, it holds
\begin{align*}
0\leq \sfd(\{ v= 0 \},x) &  \frac{s_{K/(N-1)}' (\sfd_{b(X_{\alpha})}(x))}{s_{K/(N-1)}(\sfd_{b(X_{\alpha})}(x))} \leq  \sfd(\{ v= 0 \},x)  \frac{s_{K/(N-1)}' (\sfd(\{ v= 0 \},x)) }{s_{K/(N-1)}(\sfd(\{ v= 0 \},x))}\\
 &=  \sfd(\{ v= 0 \},x) \,\sqrt{\frac{-K}{N-1}} \,  \coth\left(\sqrt{\frac{-K}{N-1}}  \sfd(\{ v= 0 \},x) \right) , \text{ for all } x\in \{ v \leq 0 \}.
\end{align*}
Since the function $[0,\infty)\ni t\mapsto t  \coth\left(\sqrt{\frac{-K}{N-1}} t \right) $ is locally bounded and the discussion for the second line of \eqref{eq:defnu} is completely analogous, the claim follows.

$\bullet$ For $K> 0$,  recall that a $\MCP(K,N)$-space has diameter at most $\pi\sqrt{(N-1)/K}$. Since  $(0,\pi)\ni t \mapsto \cot t$ is decreasing and $\sfd(\{ v= 0 \},x) \leq \sfd_{b(X_{\alpha})}(x)$ for all $x\in \{ v \leq 0 \}$, it holds
\begin{align*}
 \sfd(\{ v= 0 \},x) &  \frac{s_{K/(N-1)}' (\sfd_{b(X_{\alpha})}(x))}{s_{K/(N-1)}(\sfd_{b(X_{\alpha})}(x))} \leq  \sfd(\{ v= 0 \},x)  \frac{s_{K/(N-1)}' (\sfd(\{ v= 0 \},x)) }{s_{K/(N-1)}(\sfd(\{ v= 0 \},x))}\\
 &=  \sfd(\{ v= 0 \},x) \, \sqrt{\frac{K}{N-1}}  \,\cot\left(\sqrt{\frac{K}{N-1}}  \sfd(\{ v= 0 \},x) \right), \text{ for all } x\in \{ v \leq 0 \}.
\end{align*}
It is easily checked that  $\sup_{t \in [0,\pi \sqrt{(N-1)/K}]} \left(1+t  \sqrt{K(N-1)}  \cot \left( t \sqrt{K/(N-1)} \right) \right)\leq C'_{K,N}$, thus the bound $\nu\leq C_{K,N} \mm$ follows.
\\Since $\inf_{t \in [0,\pi  \sqrt{(N-1)/K}-\ve]} \cot  \left( t \sqrt{K/(N-1)} \right)>-\infty$ for every $\ve\in (0,\pi \sqrt{(N-1)/K}]$, we have that $\nu$ is a measure with $L^{\infty}$-bounded density unless 
we are in the second case.

To discuss the latter case we assume $K=N-1$ in order to simplify the notation, the case for general $K>0$ being completely analogous.
\\Using the maximal diameter Theorem (proved by Ohta \cite{OhtaJLMS} in the non-branched $\MCP(N-1,N)$-setting and easily extendable to the  present e.n.b situation) 
one can show that all the rays $X_{\alpha}$ are of length $\pi$, for the reader's convenience let us give a self-contained argument.  Let $X_{\bar \alpha}$ be a ray of length $\pi $ and $X_{\alpha}$ be any other ray, then 
\begin{align}
\sfd(a(X_{\alpha}),b(X_{\alpha})) + \pi 
= &~
\sfd(a(X_{\alpha}),b(X_{\alpha})) + 
\sfd(a(X_{\bar \alpha}), b(X_{\alpha}) ) + 
\sfd(b(X_{\alpha}), b(X_{\bar \alpha}) ) \nonumber \\
\geq &~
\sfd(a(X_{\bar \alpha}), b(X_{\alpha}) ) + 
\sfd(a(X_{\alpha}), b(X_{\bar \alpha}) ), \label{eq:axabxa}
\end{align}
where the first equality follows from 
\cite[Lemma 5.2]{OhtaJLMS} (since $|X_{\bar \alpha}| = \pi $, then for each $x \in X$, $\sfd(x,a(X_{\bar \alpha})) + \sfd(x,b(X_{\bar \alpha})) = \pi $). 
By $\sfd$-cyclical monotonicity also the reverse inequality is valid giving
\begin{equation}\label{eq:daabaa1}
\sfd(a(X_{\alpha}),b(X_{\alpha})) + \pi 
= \sfd(a(X_{\bar \alpha}), b(X_{\alpha}) ) + 
\sfd(a(X_{\alpha}), b(X_{\bar \alpha}) ).
\end{equation}
In particular,  $a(X_{\bar \alpha})\neq b(X_{\alpha})$: indeed otherwise \eqref{eq:daabaa1} would give $\sfd(a(X_{\alpha}),b(X_{\alpha})) + \pi 
= 
\sfd(a(X_{\alpha}), b(X_{\bar \alpha}) )$ which, in virtue of Myers diameter's bound, would imply $a(X_{\alpha})=b(X_{\alpha})$. Contradicting the fact that the rays have strictly positive length.
\\Summing $\sfd(b(X_{\alpha}), b(X_{\bar \alpha}))-\pi $ (resp.  $\sfd(a(X_{\alpha}), a(X_{\bar \alpha}) )-\pi$) to both sides of  \eqref{eq:daabaa1} and using again \cite[Lemma 5.2]{OhtaJLMS}, we get
\begin{align*}
\sfd(a(X_{\alpha}),b(X_{\alpha})) +  
\sfd(b(X_{\alpha}), b(X_{\bar \alpha}) )
=  \sfd(a(X_{\alpha}),b(X_{\bar \alpha})), \\
\sfd(a(X_{\alpha}),b(X_{\alpha})) + 
\sfd(a(X_{\alpha}), a(X_{\bar \alpha}) )
= \sfd(a(X_{\bar \alpha}),b(X_{ \alpha})).
\end{align*}
Summing up the last two identities, together with \eqref{eq:daabaa1}, yields 
$$
\sfd(a(X_{\alpha}),b(X_{\alpha})) +
\sfd(a(X_{\alpha}), a(X_{\bar \alpha}) ) +
\sfd(b(X_{\bar \alpha}), b(X_{\alpha}) ) = \pi.
$$
Since  $\sfd(a(X_{\bar \alpha}), b(X_{\bar \alpha}))=\pi$, the last identity forces the  four points $a(X_{\bar \alpha}), a(X_{\alpha}),  b(X_{\alpha}), b(X_{\bar \alpha})$ to lie on the same geodesic $\gamma$. If $a(X_{\alpha}) \neq a(X_{\bar \alpha})$ (or resp. $b(X_{\alpha}) \neq b(X_{\bar \alpha})$ ) then  $a(X_{\alpha})$ (resp. $b(X_{\alpha})$) would be an internal point of $\gamma$, contradicting that $a(X_{\alpha})$ is the initial point (resp. $b(X_{\alpha})$ is the final point) of the non-extendible ray $X_{\alpha}$.

Moreover   $(X,\mm)$ is isomorphic as a measure space to a spherical suspension over any transport ray of length $\pi$, \cite[Page 235]{OhtaJLMS}.
\\We are left to show that the density of $\nu$ is in $L^{1}(X,\mm)$. 
 By symmetry it is enough to show that 
\begin{equation}\label{eq:claimnuL1}
\int_{\{v\geq 0\}} \left|1+ (N-1)\, \sfd(\{ v= 0 \},x) \cot(\sfd(\bar{b}, x))  \right| \, \mm(dx)<\infty.
 \end{equation}
 Notice that, for every fixed $\ve\in [0,\pi/2]$, the integrand is bounded for  $\sfd(\bar{b}, x)\in [\ve,\pi-\ve]$. 
 \\Since $\bar{b}\in \{v\leq 0\}$, if $\bar{b}$ is accumulation point for $\{v\geq 0\}$, then $v(\bar{b})=0$. As $v$ is strictly decreasing on the rays, which cover a dense subset,  it follows that $\set{ v=0}=\set{\bar{b}}$.  Thus, in this case, the integrand becomes $1+ \sfd(\bar{b}, x) \cot(\sfd(\bar{b}, x))$ which is bounded for  $\sfd(\bar{b}, x)\in [0,\ve]$.
 \\We now show that the integral is finite also on $\set{x\,:\, \sfd(\bar{b}, x)\in [\pi-\ve, \pi]}\cap \set{v\geq 0}$. Since $\sfd(\{ v= 0 \},x)\leq \sfd(\bar{b}, x)$, it is enough to show that 
\begin{equation}\label{eq:claimnuL1Pf}
\int_{\{v\geq 0\}\cap \sfd(\bar{b}, x)\in [\pi-\ve,\pi]} \left|1+ (N-1)\, \sfd(\bar{b}, x) \cot(\sfd(\bar{b}, x))  \right| \, \mm(dx) <\infty.
 \end{equation}
Recalling that  $(X,\mm)$ is isomorphic as a measure space to a spherical suspension over any transport ray of length $\pi$, the integral in \eqref{eq:claimnuL1Pf} is bounded by
\begin{align*}
\int_{[\pi-\ve,\pi]}  \left|1+ (N-1)\,t \cot t\right| \sin^{N-1}(t) \,dt &= \int_{[0,\ve]} \left[(N-1)(\pi-s)  \cot s -1\right] \sin^{N-1}(s) \, ds  \\
&= (N-1) \pi \int_{[0,\ve]} s^{N-2} \, ds +O(\ve) <\infty,
\end{align*} 
since $N>1$. This concludes the proof that the density of $\nu$ is in $L^{1}(X,\mm)$.
The stronger rigidity statement under the stronger $\RCD(K,N)$ assumption is a direct consequence of the Maximal Diameter Theorem proved by  Ketterer \cite{Ket} in the $\RCD(K,N)$-setting.
\end{proof}

Corollary \ref{C:squared} and Lemma \ref{lem:nucases} have far reaching consequences.

\begin{theorem}\label{T:d3}
Let $(X,\sfd,\mm)$ be an e.n.b. metric measure space verifying $\MCP(K,N)$, for some $K\in \R, N\in (1,\infty)$. 

Consider the signed distance function $d_{v}$ 
for some continuous function $v : X\to\R$ and the associated disintegration $\mm\llcorner_{X\setminus \set{v=0}} =\int_{Q} h_{\alpha} \cH^{1}\llcorner_{X_{\alpha}} \, \qq(d\alpha)$.

Then $d_{v}^{2} \in D(\bold{\Delta})$ and  one element of $\bold{\Delta}(d_{v}^{2})$, that we denote with $\Delta d_{v}^{2}$,  has the  following representation formula: 
\begin{equation}\label{eq:repdeltadv2}
\Delta d_{v}^{2} = 2 (1  - d_{v} (\log h_{\alpha})' )\mm - 2\int_{Q}( h_{\alpha}d_{v})[\delta_{a(X_{\alpha})} - \delta_{b(X_{\alpha})}] \,\qq(d\alpha).
\end{equation}
Moreover $\Delta d_{v}^{2}$ is a sum of two signed Radon measures and the  next comparison results hold true:
\begin{align}
\Delta d_{v}^{2} \leq \nu&:= ~ 
2\mm+ 2 (N-1)\, \sfd(\{ v= 0 \},x)  \frac{s_{K/(N-1)}' (\sfd_{b(X_{\alpha})}(x))}{s_{K/(N-1)}(\sfd_{b(X_{\alpha})}(x))}  \, \mm\llcorner_{\{v\geq 0\}}  \nonumber\\
 &~ \quad \qquad + 2 (N-1)\,   \sfd(\{ v= 0 \},x)  \frac{s_{K/(N-1)}' (\sfd_{a(X_{\alpha})}(x))}{s_{K/(N-1)}(\sfd_{a(X_{\alpha})}(x))} \,  \mm\llcorner_{\{v< 0\}},
\label{eq:Deltadv2leq}
\\
\left[\Delta d_{v}^{2}\right]^{reg}& := 2\left(1  - d_{v} (\log h_{\alpha})' \right)\mm  \nonumber\\
&~  \geq  
2\mm- 2 (N-1)\,  \sfd(\{ v =0 \},x) \,
\frac{s_{K/(N-1)}' (\sfd_{a(X_{\alpha})}(x))}{s_{K/(N-1)}(\sfd_{a(X_{\alpha})}(x))} \, \mm\llcorner_{\{v\geq 0\}}  \nonumber\\
 &~ \quad \qquad - 2 (N-1)\, 
 \sfd(\{ v= 0 \},x) \, \frac{s_{K/(N-1)}' (\sfd_{b(X_{\alpha})}(x))}{s_{K/(N-1)}(\sfd_{b(X_{\alpha})}(x))}\,  \mm\llcorner_{\{v< 0\}},
\label{eq:Deltadv2geq}
\end{align}
where $\left[\Delta d_{v}^{2}\right]^{reg}$ is the regular part of $\Delta d_{v}^{2}$  (i.e. absolutely continuous with respect to $\mm$).
\end{theorem}

\begin{proof} 
Fix any compactly supported Lipschitz function 
$f : X \to \R$  and integrate by parts on each ray $X_{\alpha}$ to obtain
\begin{align}
\int_{X_{\alpha}} & d_{v}(x)f'(x) h_{\alpha}(x) \,\cH^{1}(dx)  \nonumber \\
 = &~ - \int_{X_{\alpha}} f(x) d_{v}'(x) h_{\alpha} (x) \,\cH^{1}(dx)   - \int_{X_{\alpha}} f(x) d_{v}(x) h'_{\alpha} (x) \,\cH^{1}(dx) \nonumber  \\
&\qquad + (f d_{v} h_{\alpha}) (b(X_{\alpha})) - (f d_{v} h_{\alpha})(a(X_{\alpha}))  \nonumber  \\
 = &~  \int_{X_{\alpha}} f(x) h_{\alpha} (x) \,\cH^{1}(dx)  - \int_{X_{\alpha}} f(x) d_{v}(x) h'_{\alpha} (x)  \,\cH^{1}(dx) \nonumber \\
&\qquad
+ (f d_{v} h_{\alpha}) (b(X_{\alpha}))
-  (f d_{v} h_{\alpha})(a(X_{\alpha})) \nonumber \\
 = &~  
\int_{X_{\alpha}} f(x)\big(1 - d_{v}(x) (\log h_{\alpha})'(x)\big) \,h_{\alpha}(x)\cH^{1}(dx)  \nonumber \\
&\qquad
+ (f d_{v} h_{\alpha}) (b(X_{\alpha}))
-  (f d_{v} h_{\alpha})(a(X_{\alpha})). \label{eq:intdf'h}
\end{align}
Then considering along each ray $X_{\alpha}$ 
the two regions $\{ v\geq 0 \}$ and $\{ v<0 \}$, 
we notice that \eqref{E:logder} gives
\begin{align*}
-d_{v}(x) (\log h_{\alpha})'(x) 
\leq  &~ 
\sfd(\{ v =0 \},x) \,(N-1)\,
\frac{s_{K/(N-1)}' (\sfd_{b(X_{\alpha})}(x))}{s_{K/(N-1)}(\sfd_{b(X_{\alpha})}(x))} \, \chi_{\{v\geq 0\}} (x) \\
& + 
 \sfd(\{ v= 0 \},x) \, (N-1)\,\frac{s_{K/(N-1)}' (\sfd_{a(X_{\alpha})}(x))}{s_{K/(N-1)}(\sfd_{a(X_{\alpha})}(x))}\,\chi_{\{v< 0\}}(x) = : V_{\alpha}(x) .
\end{align*}
Hence we can collect the estimates, using Remark \ref{R:distancesign}, and obtain 
$$
\int_{X_{\alpha}} d_{v}(x)f'(x) h_{\alpha}(x) \,\cH^{1}(dx)  
\leq \int_{X_{\alpha}} \left(1+ V_{\alpha}(x) \right) f(x) h_{\alpha}(x)\,\cH^{1}(dx),
$$
provided $f$ is non-negative. 
Thanks to Lemma \ref{lem:nucases},
$$
\nu
 = 2\int_{Q} \left(1+ V_{\alpha}(x) \right) \mm_{\alpha} \qq(d\alpha) \\
 = 2  \left(1+ V \right) \, \mm, 
$$
is a  well defined Radon (possibly signed) measure.
\\ Hence, continuing from  \eqref{eq:intdf'h}, 
the expression
\begin{equation}\label{eq:defDeltasfdp2pf}
\Delta d_{v}^{2} : =  2\int_{Q} (h_{\alpha}  - d_{v} h'_{\alpha})\cH^{1}\llcorner_{X_{\alpha}} \,\qq(d\alpha) + 2\int_{Q} (h_{\alpha} d_{v})  [\delta_{b(X_{\alpha})} - \delta_{a(X_{\alpha})}] \,\qq(d\alpha), 
\end{equation}
once restricted to bounded subsets, defines a Borel measure  with values in  $\R\cup \{-\infty\}$ which satisfies $\Delta d_{v}^{2}\leq \nu$.  
Now, combining Theorem \ref{T:sigmadisint} with \eqref{eq:intdf'h} and \eqref{eq:defDeltasfdp2pf}, we get
$$
 \int_{X} f \,\Delta d_{v}^{2}(dx)= 2 \int_{Q} \int_{X_{\alpha}} d_{v}(x)f'(x) h_{\alpha}(x) \,\cH^{1}(dx) \qq(d\alpha)= 2\int_{\T_{d_{v}}}  d_{v}(x)f'(x)  \mm(dx),
$$
for any compactly supported Lipschitz function $f : X \to \R$.  
Therefore, Corollary \ref{C:squared} yields
$$
 \int_{\T_{d_{v}}}D^{-}f (-\nabla d_{v}^{2}) \, \mm \leq \int_{X} f \,\Delta d_{v}^{2}(dx) \leq 
 \int_{\T_{d_{v}}}D^{+}f (-\nabla d_{v}^{2}) \, \mm,
$$
for any compactly supported Lipschitz function $f : X \to \R$. Since $X \setminus \T_{d_{v}} \subset \{ v = 0 \}=\{d_{v}=0\}$, from the locality properties of differentials (see \cite[equation (3.7)]{Gigli12})
we can turn the previous inequalities in the next ones
\begin{equation}\label{eq:ULBDeltasfd2pf}
 \int_{X}D^{-}f (-\nabla d_{v}^{2}) \, \mm \leq \int_{X} f \,\Delta d_{v}^{2}(dx) \leq 
 \int_{X}D^{+}f (-\nabla d_{v}^{2}) \, \mm,
 \end{equation}
valid
for any compactly supported Lipschitz function $f : X \to \R$. In order to show that  $d^{2}_{v} \in D(\bold{\Delta})$
with $\Delta d_{v}^{2} \in \bold{\Delta}(d_{v}^{2})$, we are thus left to prove that $\Delta d_{v}^{2}$ is a signed Radon measure.

We now claim that $\Delta d_{v}^{2}$ is a sum of two Radon measures over $X$. Since $\Delta d_{v}^{2}\leq \nu$ with $\nu$ signed Radon measure, thanks to the  Riesz-Markov-Kakutani Representation Theorem it is enough to show that $\Delta d_{v}^{2}$ defines a Radon functional. 
\\To this aim, fix a compact subset $W\subset X$ and fix   a compactly supported Lipschitz cutoff function $\chi_{W}:X\to [0,1]$ satisfying $\chi_{W}\equiv 1$ on $W$.
At first observe that, using \eqref{eq:ULBDeltasfd2pf}, for any Lipschitz function $f:X\to \R$ with $\supp(f)\subset W$ we have
\begin{align*}
\left| \int_{X} \chi_{W} \,\Delta d_{v}^{2}(dx) \right| & \leq 2 \left(\max_{x\in \supp(\chi_{W})} d_{v}(x)\right) \; \Lip(\chi_{W})\; \mm(\supp(\chi_{W})) \in  (0,\infty)  \nonumber\\
\left| \int_{X} (f\chi_{W}) \,\Delta d_{v}^{2}(dx) \right| & \leq 2 \left(\max_{x\in \supp(\chi_{W})} d_{v}(x)\right) \; \Lip(f\chi_{W})\; \mm(\supp(\chi_{W})) \in  (0,\infty).
\end{align*}
Thus for any Lipschitz function $f:X\to \R$ with $\supp(f)\subset W$, using that $\Delta d_{v}^{2} \leq \nu \leq \nu^{+} $, on one hand we have 
\begin{align}
\int_{X} f  \,\Delta d_{v}^{2}&= - \int_{X} (\max f-f)\, \chi_{W} \,\Delta d_{v}^{2} +  \int_{X} (\max f)\, \chi_{W} \,\Delta d_{v}^{2} \nonumber\\
&\geq  - \int_{X} (\max f-f)\, \chi_{W} \, \nu^{+} - C_{W} (\max f), \label{eq:LBintfDeltad2}
\end{align}
where $C_{W}:=2 (\Lip \chi_{W})\, \max_{x\in \supp(\chi_{W}) } \sfd_{p}(x) \, \mm(\supp(\chi_{W})) \in  (0,\infty)$ depends only on $\chi_{W}$.
\\On the other hand, 
\begin{align}
\int_{X} f  \,\Delta d_{v}^{2}&=  \int_{X} f^{+}  \,\Delta d_{v}^{2} - \int_{X} f^{-}  \,\Delta d_{v}^{2} \nonumber\\
&\leq \int_{X} f^{+}  \, \nu^{+}  + \int_{X} (\max f^{-}-f^{-})\, \chi_{W} \, \nu^{+} +C_{W} (\max f^{-})\nonumber\\
&\leq \max |f| \; (\nu^{+}(W)+\nu^{+}(\supp(\chi_{W}))+C_{W}). \label{eq:UBintfDeltad2}
\end{align}
The combination of \eqref{eq:LBintfDeltad2} and \eqref{eq:UBintfDeltad2} gives that, for every compact subset $W\subset X$ there exists a constant $C'_{W}=2 (\nu^{+}(\supp(\chi_{W}))+ \Lip \chi_{W}\, \max_{x\in \supp(\chi_{W}) } d_{v}(x) \, \mm(\supp(\chi_{W}))\in (0,\infty)$ such that
$$\left| \int_{X} f  \,\Delta d_{v}^{2} \right| \leq C'_{W} \, \max |f|$$
for every Lipschitz function $f:X\to \R$ with $\supp(f)\subset W$, showing that  $\Delta d_{v}^{2}$ is a Radon functional and thus  $d^{2}_{v} \in D(\bold{\Delta})$
with $\Delta d_{v}^{2} \in \bold{\Delta}(d_{v}^{2})$.
\\In order to complete the proof we are left with showing \eqref{eq:Deltadv2geq}: again from  \eqref{E:logder}
\begin{align*}
-d_{v}(x) (\log h_{\alpha})'(x) 
\geq  &~ 
-(N-1) \sfd(\{ v =0 \},x) \,
\frac{s_{K/(N-1)}' (\sfd_{a(X_{\alpha})}(x))}{s_{K/(N-1)}(\sfd_{a(X_{\alpha})}(x))} \, \chi_{\{v\geq 0\}} (x) \\
&~  - (N-1)
 \sfd(\{ v= 0 \},x) \, \frac{s_{K/(N-1)}' (\sfd_{b(X_{\alpha})}(x))}{s_{K/(N-1)}(\sfd_{b(X_{\alpha})}(x))}\,\chi_{\{v< 0\}}(x),
\end{align*}
and the claim is proved.
\end{proof}
\begin{remark}
\begin{itemize}
\item In case $X$ is bounded, then in the  proof of Theorem \ref{T:d3} one can pick $W=X$ and $\chi_{W}\equiv 1$, giving that the total variation of $\Delta d_{v}^{2}$ is bounded by $\|\Delta d_{v}^{2}\|\leq 2 \nu^{+}(X)$.
\item   Theorem  \ref{T:Deltadv2Smooth} can be proved using  Corollary \ref{C:sigmadisintSmooth}  in the proof of Theorem \ref{T:d3} and following verbatim the arguments. Uniqueness of the representation of the Laplacian, follows then from infinitesimal Hilbertianity of smooth manifolds.
\end{itemize}

\end{remark}

The representation formula for the Laplacian of the signed distance function on $X\setminus\set{v=0}$ follows from Theorem \ref{T:d3}   by chain rule \cite[Proposition 4.11]{Gigli12}.

\begin{corollary}\label{cor:Deltadv}
Let $(X,\sfd,\mm)$ be an e.n.b. metric measure space verifying $\MCP(K,N)$, for some $K\in \R, N\in (1,\infty)$. 

Consider the signed distance function $d_{v}$ 
for some continuous function $v : X\to\R$ and the associated disintegration $\mm\llcorner_{X\setminus \set{v=0}} =\int_{Q} h_{\alpha} \cH^{1}\llcorner_{X_{\alpha}} \, \qq(d\alpha)$. Then 
\begin{enumerate}
\item   $|d_{v}| \in D(\bold{\Delta}, X\setminus \set{v=0})$ and  one element of $\bold{\Delta}(|d_{v}|)\llcorner_{ X\setminus \set{v=0}}$, that we denote with $\Delta |d_{v}|\llcorner_{ X\setminus \set{v=0}}$ is the Radon functional on $X\setminus \set{v=0}$  with the   following representation formula: 
\begin{align}
\Delta |d_{v}|\llcorner_{ X\setminus \set{v=0}} = - \sgn(v)\, (\log h_{\alpha})' \mm\llcorner_{ X\setminus \set{v=0}} - \int_{Q}( h_{\alpha} [\delta_{a(X_{\alpha})\cap \set{v>0}} + \delta_{b(X_{\alpha})\cap  \set{v<0}}] \,\qq(d\alpha). \label{eq:repdelta|vp|}
\end{align}
Moreover the  next comparison results hold true:
\begin{align}
\Delta |d_{v}|\llcorner_{ X\setminus \set{v=0}} &\leq  (N-1) \frac{s_{K/(N-1)}' (\sfd_{b(X_{\alpha})}(x))}{s_{K/(N-1)}(\sfd_{b(X_{\alpha})}(x))}  \, \mm\llcorner_{\set{v>0}} \nonumber \\
&\quad+ (N-1) \frac{s_{K/(N-1)}' (\sfd_{a(X_{\alpha})}(x))}{s_{K/(N-1)}(\sfd_{a(X_{\alpha})}(x))} \, \mm\llcorner_{\set{v<0}},
\label{eq:Delta|dv|leq} \\
\left[\Delta |d_{v}|\llcorner_{ X\setminus \set{v=0}}\right]^{reg}& :=   - \sgn (v)   (\log h_{\alpha})'  \mm\llcorner_{ X\setminus \set{v=0}}  \nonumber \\
& \geq - (N-1) \frac{s_{K/(N-1)}' (\sfd_{a(X_{\alpha})}(x))}{s_{K/(N-1)}(\sfd_{a(X_{\alpha})}(x))} \, \mm\llcorner_{ \set{v>0}} \nonumber \\
&\qquad - (N-1) \frac{s_{K/(N-1)}' (\sfd_{b(X_{\alpha})}(x))}{s_{K/(N-1)}(\sfd_{b(X_{\alpha})}(x))}  \, \mm\llcorner_{ \set{v<0}},\label{eq:Delta|dv|geq}
\end{align}
where $\left[\Delta |d_{v}|\llcorner_{ X\setminus \set{v=0}}\right]^{reg}$ is the regular part of $\Delta |d_v|\llcorner_{ X\setminus \set{v=0}}$  (i.e. absolutely continuous with respect to $\mm$).

\item   $d_{v} \in D(\bold{\Delta}, X\setminus \set{v=0})$ and  one element of $\bold{\Delta}(d_{v})\llcorner_{ X\setminus \set{v=0}}$, that we denote with $\Delta d_{v}\llcorner_{ X\setminus \set{v=0}}$,  is the Radon functional on $X\setminus \set{v=0}$ with  the  following representation formula: 
\begin{equation}\label{eq:repdeltavp}
\Delta d_{v}\llcorner_{ X\setminus \set{v=0}} = - (\log h_{\alpha})' \mm\llcorner_{ X\setminus \set{v=0}} - \int_{Q}( h_{\alpha} [\delta_{a(X_{\alpha})\cap  \set{v>0}} - \delta_{b(X_{\alpha})\cap  \set{v<0}}] \,\qq(d\alpha).
\end{equation}
Moreover the  next comparison results hold true:
\begin{align}
\Delta d_{v}\llcorner_{ X\setminus \set{v=0}} &\leq (N-1)  \frac{s_{K/(N-1)}' (\sfd_{b(X_{\alpha})}(x))}{s_{K/(N-1)}(\sfd_{b(X_{\alpha})}(x))}  \, \mm\llcorner_{ X\setminus \set{v=0}} + \int_{Q} h_{\alpha} \delta_{b(X_{\alpha})\cap \set{v<0} } \,\qq(d\alpha),
\label{eq:Deltadvleq}
\\
\Delta d_{v} \llcorner_{ X\setminus \set{v=0}} & \geq  - (N-1) \frac{s_{K/(N-1)}' (\sfd_{a(X_{\alpha})}(x))}{s_{K/(N-1)}(\sfd_{a(X_{\alpha})}(x))} \, \mm\llcorner_{ X\setminus \set{v=0}} - \int_{Q}( h_{\alpha} [\delta_{a(X_{\alpha})\cap  \set{v>0}}] \,\qq(d\alpha).  \label{eq:Deltadvgeq}
\end{align}
\end{enumerate}
\end{corollary}

\begin{proof}
Writing $\sgn(v) d_{v}= \sqrt {d^{2}_{v}}$, a direct application of chain rule \cite[Proposition 4.11]{Gigli12} combined with Theorem \ref{T:d3} gives that  $|d_{v}| \in D(\bold{\Delta}, X\setminus \set{v=0})$ and that $\Delta |d_{v}|$ defined in \eqref{eq:repdelta|vp|} is an element of $\bold{\Delta} |d_{v}| \llcorner_{ X\setminus \set{v=0}}$.  The comparison results \eqref{eq:Delta|dv|leq}, \eqref{eq:Delta|dv|geq} follow from the definition  \eqref{eq:repdelta|vp|} together with   \eqref{E:logder}.
\\Since $d_{v}=\sgn(v) \, |dv|$, it is clear that $d_{v} \in D(\bold{\Delta}, X\setminus \set{v=0})$ with $\bold{\Delta} (d_{v})\llcorner_{X\setminus \set{v=0}}= \sgn(v) \, \bold{\Delta} (|d_{v}|)\llcorner_{X\setminus \set{v=0}}$; thus $\Delta d_{v}\llcorner_{ X\setminus \set{v=0}}$ defined in \eqref{eq:repdeltavp} is an element of $\bold{\Delta} (d_{v})\llcorner_{X\setminus \set{v=0}}$ and the comparison results \eqref{eq:Deltadvleq}, \eqref{eq:Deltadvgeq} follow again  from  \eqref{E:logder}.
\end{proof}

We now specialise the above results  to the distance function from a point $p\in X$, i.e. we pick  $v=\sfd_{p}$ so that $\{v = 0 \} = p$ and $v \geq 0$ everywhere. Note that, in this case, $b(X_{\alpha})=p$ for $\qq$-a.e. $\alpha\in Q$.

\begin{corollary}\label{C:d2}
Let $(X,\sfd,\mm)$ be an e.n.b. metric measure space verifying $\MCP(K,N)$, for some $K\in \R, N\in (1,\infty)$. Fix $p \in X$, consider $\sfd_{p} : = \sfd(p,\cdot)$ and the associated disintegration $\mm =\int_{Q} h_{\alpha} \cH^{1}\llcorner X_{\alpha} \, \qq(d\alpha)$.
\\ Then $\sfd_{p}^{2} \in D(\bold{\Delta})$ and  one element of $\bold{\Delta}(\sfd_{p}^{2})$, that we denote with $\Delta \sfd_{p}^{2}$, is  a sum of two signed Radon measures and satisfies the following representation formula: 
\begin{equation}\label{eq:repdeltadp}
\Delta \sfd_{p}^{2} = 2 (1  - \sfd_{p} (\log h_{\alpha})' )\mm - 2\int_{Q} h_{\alpha}\sfd_{p} \,\delta_{a(X_{\alpha})} \,\qq(d\alpha).
\end{equation}
Moreover, the  next comparison results hold true:
\begin{align}
\Delta \sfd_{p}^{2} &\leq \nu:= 2  \left(1+ (N-1)\, \sfd_{p}(x)  \frac{s_{K/(N-1)}' (\sfd_{p}(x))}{s_{K/(N-1)}(\sfd_{p}(x))} \right) \, \mm ,  \label{eq:Deltadp2leq} \\
\left[\Delta \sfd_{p}^{2}\right]^{reg}& := 2\left(1  -  \sfd_{p} (\log h_{\alpha})' \right)\mm \geq 2  \left(1 -(N-1)\, \sfd_{p}  \frac{s_{K/(N-1)}' (\sfd_{a(X_{\alpha})}(x))}{s_{K/(N-1)}(\sfd_{a(X_{\alpha})}(x))}  \right) \mm, \label{eq:Deltadp2geq2}
\end{align}
where $\left[\Delta \sfd_{p}^{2}\right]^{reg}$ is the regular part of $\Delta \sfd_{p}^{2}$  (i.e. absolutely continuous with respect to $\mm$).
\end{corollary}

\begin{remark}[On the lower bound \eqref{eq:Deltadp2geq2}]
Denote with ${\mathcal C}_{p}:=\{a(X_{\alpha})\}_{\alpha\in Q}$ the cut locus of $p$. Then for every $\ve>0$ there exists $C_{\ve}>0$ so that for every bounded subset $W\subset X$ it holds:
\begin{align*}
\left[\Delta \sfd_{p}^{2}\right]^{reg}_{\llcorner {W}}&\geq 2  \left(1 - (N-1)\, \sfd_{p}  \frac{s_{K/(N-1)}' (\sfd_{a(X_{\alpha})}(x))}{s_{K/(N-1)}(\sfd_{a(X_{\alpha})}(x))}  \right) \mm\llcorner_{W} \nonumber \\
&\geq -C_{\ve, W} \mm\llcorner_{W} \quad \text{on } W\cap \{x=g_{t}(a_{\alpha})\,:\, t\geq \ve\}\supset W\cap\{x \in X\,:\, \sfd(x, {\mathcal C}_{p} )\geq \ve\}.
\end{align*}
\end{remark}

The representation formula for the Laplacian of the distance function follows from Corollary \ref{C:d2}  by chain rule \cite[Proposition 4.11]{Gigli12}, writing $\sgn(v) d_{v}= \sqrt {d^{2}_{v}}$.

\begin{corollary}\label{cor:Deltad}
Let $(X,\sfd,\mm)$ be an e.n.b. metric measure space verifying $\MCP(K,N)$, for some $K\in \R, N\in (1,\infty)$. Fix $p \in X$, consider $\sfd_{p} : = \sfd(p,\cdot)$ and the associated disintegration $\mm=\int_{Q} h_{\alpha} \cH^{1}\llcorner_{X_{\alpha}} \, \qq(d\alpha)$.
\\ Then $\sfd_{p} \in D(\bold{\Delta}, X\setminus \{p\})$ and  one element of $\bold{\Delta}\sfd_{p}\llcorner_{X\setminus \{p\}}$, that we denote with $\Delta \sfd_{p}\llcorner_{X\setminus \{p\}}$,  is a Radon functional with the following representation formula: 
\begin{equation}\label{eq:repdeltadp}
\Delta \sfd_{p}\llcorner_{X\setminus \{p\}} = - (\log h_{\alpha})' \, \mm -   \int_{Q} h_{\alpha} \delta_{a(X_{\alpha})} \,\qq(d\alpha).
\end{equation}
Moreover, the  next comparison results hold true:
\begin{align}
\Delta \sfd_{p}\llcorner_{X\setminus \{p\}}&\leq (N-1)\,  \frac{s_{K/(N-1)}' (\sfd_{p}(x))}{s_{K/(N-1)}(\sfd_{p}(x))}  \, \mm ,  \label{eq:Deltadpleq} \\
\left[\Delta \sfd_{p}\llcorner_{X\setminus \{p\}}\right]^{reg}& :=   - (\log h_{\alpha})'  \mm \geq  - (N-1)  \frac{s_{K/(N-1)}' (\sfd_{a(X_{\alpha})}(x))}{s_{K/(N-1)}(\sfd_{a(X_{\alpha})}(x))}   \mm, \label{eq:Deltadpgeq}
\end{align}
where $\left[\Delta \sfd_{p}\llcorner_{X\setminus \{p\}}\right]^{reg}$ is the regular part of $\Delta \sfd_{p}\llcorner_{X\setminus \{p\}}$  (i.e. absolutely continuous with respect to $\mm$).
\end{corollary}

\begin{remark}
Corollary \ref{cor:Deltad} should be compared with \cite[Corollary 5.15, Remark 5.16]{Gigli12}, where it was proved that $\sfd_{p} \in D(\bold{\Delta}, X\setminus \{p\})$ together with the upper bound \eqref{eq:Deltadpleq} under the assumption that $(X,\sfd,\mm)$ is an infinitesimally strictly convex $\MCP(K,N)$-space. 
\\ Let us stress that, by the very definition, the Laplacian in the  infinitesimally strictly convex setting is single valued, simplifying the treatment. 
\\ One  novelty of Corollary \ref{cor:Deltad} is that the  infinitesimal strict convexity is replaced by the essentially non branching property which, a priori, does not exclude a multi-valued Laplacian.  In addition to that, the geometrically new content of  Corollary \ref{cor:Deltad}  when compared with \cite{Gigli12} is that it contains an \emph{exact representation formula}  \eqref{eq:repdeltadp} which also gives the new lower bound  \eqref{eq:Deltadpgeq}.
\end{remark}

\part{Applications}\label{part2}

In Part II of the paper we collect all the main applications of the results obtained in 
Part \ref{part1}.

\section{The singular part of the Laplacian}\label{Ss:singularpart}

In order to state the next corollary recall that from essentially non-branching and $\MCP(K,N)$ it follows that for every fixed $p\in X$ and $\mm$-a.e. 
$x \in X$ (precisely on $\T^{nb}_{\sfd_{p}}$) there exists a unique geodesic $\gamma^{x}$ starting from $x$ and arriving at $p$, i.e. $\gamma^{x}_{0}= x$ and $\gamma^{x}_{1} = p$. For each $t \in [0,1]$, define the map
\begin{equation}\label{eq:defTt}
T_{t} : \T^{nb}_{\sfd_{p}} \to \T^{nb}_{\sfd_{p}}, \qquad T_{t}(x) : = \gamma^{x}_{t}.
\end{equation}
It is worth noting that $T_{t}$ is also the $W_{2}$-optimal transport map from the (renormalized) ambient measure $\mm$ to $\delta_{p}$, provided $\mm(X) < \infty$.

The goal of the next proposition is to get  some refined information on the cut locus  $\mathcal{C}_{p}$ of $p$; more precisely, we infer an upper bound on an optimal transport type Minkowski content of $\mathcal{C}_{p}$.

\begin{proposition}\label{prop:Cp}
Let $(X,\sfd,\mm)$ be an e.n.b. metric measure space verifying $\MCP(K,N)$, for some $K\in \R, N\in (1,\infty)$. 
Fix any point $p \in X$ and consider for each $t\in [0,1]$ the map $T_{t}$ defined by \eqref{eq:defTt}.

Then, for every bounded open subset $W\subset X$ it holds
\begin{equation}\label{E:cutlocu2}
\limsup_{\ve \downarrow 0}\frac{\mm((X \setminus T_{\ve}(X))\cap W)}{\ve} \leq \| [\Delta \sfd^{2}_{p}]_{sing} \| (W) <\infty.
\end{equation}
\end{proposition}

\begin{remark}[Geometric meaning of Proposition \ref{prop:Cp}]
Fix $p \in X$, consider $\sfd_{p} : = \sfd(p,\cdot)$ and the associated disintegration $\mm=\int_{Q} h_{\alpha} \cH^{1}\llcorner_{X_{\alpha}} \, \qq(d\alpha)$. 
Then the cut locus $\mathcal{C}_{p}$ of $p$ coincides with the set of initial points $\{a(X_{\alpha})\}_{\alpha\in Q}$ of the transport rays. The set  $X \setminus T_{\ve}(X)$ thus
can be seen as an ``optimal transport neighborhood'' of the cut locus $\mathcal{C}_{p}$  and therefore
\eqref{E:cutlocu2} gives an optimal transport type estimate on a weak version of the codimension one Minkowski content of $\mathcal{C}_{p}$.

Since the cut locus of a point in an e.n.b. $\MCP(K,N)$ space can be dense (this can be the case already for the boundary of a convex body in $\R^{3}$),  one cannot expect an upper  bound on  the classical codimension one Minkowski content of $\mathcal{C}_{p}$.
The bound \eqref{E:cutlocu2} looks interesting already in the classical setting of a smooth Riemannian manifold. Indeed it is well known that $\mathcal{C}_{p}$ is rectifiable with locally finite codimension one Hausdorff measure (see for instance \cite{MantMen}), but in the literature it seems not to be present any (local) bound on its codimension one Minkowski content.
\end{remark}

\begin{proof}
If $X$ is bounded, one can choose $W=X$ and the proof is easier (there is no need to introduce an intermediate set $U$ in the arguments below); we thus discuss directly the case when $X$ is not bounded.
\\Let $U\supset W$ be a bounded open subset such that $W$ is compactly contained in $U$, in particular $\sfd(W,X\setminus U)>0$.
\\With a slight abuse of notation, for ease of writing,  in the next computations we identify the ray  $(X_{\alpha}, \sfd, \mm_{\alpha})$ with the real interval $\Big((a_{\alpha}, b_{\alpha}),  |\cdot|, h_{\alpha} \L^{1} \Big)$ isomorphic to it as a m.m.s..
\\Recalling from Remark \ref{R:continuityboundary} that $h_{\alpha}:X_{\alpha}\simeq (a_{\alpha}, b_{\alpha})\to \R^{+}$ is continuous up to the initial point $a_{\alpha}$, it is clear that
$$
h_{\alpha}(a(X_{\alpha})) \sfd_{p}(a(X_{\alpha}))
= \lim_{\ve \downarrow 0} \frac{1}{\ve}
\int_{[a_{\alpha}, a_{\alpha} + \ve |X_{\alpha}|]} h_{\alpha}(s) \, ds,
$$
where $|X_{\alpha}|$ denotes the length of the transport ray $X_{\alpha}$, i.e. $|X_{\alpha}|=\sfd(a(X_{\alpha}),b(X_{\alpha})) = \sfd(a(X_{\alpha}),p)$.
Hence, for any bounded open subset $U\subset X$ it holds 
\begin{align*}
\| [\Delta \sfd^{2}_{p}]^{sing} \|(U) 
&=\int_{\{\alpha\in Q:a(X_{\alpha})\in U\}}   (h_{\alpha} \sfd_{p})(a(X_{\alpha})) \, \qq(d\alpha)  \\
&= 
\int_{\{\alpha\in Q:a(X_{\alpha})\in U\}} 
\lim_{\ve \to 0} \frac{1}{\ve}
\int_{[a_{\alpha}, a_{\alpha} + \ve |X_{\alpha}|]} h_{\alpha}(s) \, ds \, \qq(d\alpha),
\end{align*}
where $\| [\Delta \sfd^{2}_{p}]^{sing} \|(U) $ denotes the total variation measure of $U$.
Since by Corollary \ref{C:d2} we know that $\| [\Delta \sfd^{2}_{p}]^{sing} \|(U)<\infty$, by Fatou's Lemma we infer
\begin{equation}\label{eq:Deltasinglimsup}
\limsup_{\ve \downarrow 0}
\frac{1}{\ve}
\int_{\{\alpha\in Q:a(X_{\alpha})\in U\}} 
\int_{[a_{\alpha}, a_{\alpha} + \ve |X_{\alpha}|]} h_{\alpha}(s) \, ds \,\qq(d\alpha) \leq \| [\Delta \sfd^{2}_{p}]^{sing} \| (U)<\infty .
\end{equation}
We then look for a more convenient expression of the left-hand side of the previous inequality. 
First,  note that for $\ve$ sufficiently small such that 
$\ve/(1-\ve)< \frac{\sfd(W, X\setminus U)}{\sfd_{p}(W)}$ it holds
\begin{equation}\label{eq:BoundUW}
\int_{Q}\int_{[a_{\alpha}, a_{\alpha} + \ve|X_{\alpha}|]\cap W} h_{\alpha}(s) \, ds \,\qq(d\alpha)\leq  \int_{\{\alpha\in Q:a(X_{\alpha})\in U\}} 
\int_{[a_{\alpha}, a_{\alpha} + \ve |X_{\alpha}|]} h_{\alpha}(s) \, ds \,\qq(d\alpha). 
\end{equation}
Recalling the definition of the map $T_{t}$ given in \eqref{eq:defTt}, we now claim that 
\begin{equation}\label{E:cutlocus}
\mm((X \setminus T_{\ve}(X)) \cap W) = 
\int_{Q} 
\int_{[a_{\alpha}, a_{\alpha} + \ve|X_{\alpha}|]\cap W} h_{\alpha}(s) \, ds \,\qq(d\alpha).
\end{equation}
Indeed, on the one hand, by the  Disintegration Theorem \ref{T:sigmadisint} we know that
$$
\mm((X \setminus T_{\ve}(X)) \cap W) = 
\int_{Q} 
\int_{X_{\alpha} \cap (X \setminus T_{\ve}(X))\cap W} h_{\alpha}(s) \, ds \,\qq(d\alpha).
$$
On the other hand, since  trivially 
$$
X_{\alpha} \cap (X \setminus T_{t}(X)) \cap W = 
X_{\alpha} \setminus T_{t}(X) \cap W, 
$$
and since, as $T_{t}$ is translating along  $\T^{nb}_{\sfd_{p}}$, one has
$X_{\alpha} \setminus T_{t}(X) = 
X_{\alpha} \setminus T_{t}(X_{\alpha})$, we obtain
$$
X_{\alpha} \cap (X \setminus T_{t}(X)) \cap W= 
X_{\alpha}  \setminus T_{t}(X_{\alpha})\cap W.
$$
The claim \eqref{E:cutlocus} follows.
The combination of \eqref{eq:Deltasinglimsup}, \eqref{eq:BoundUW} and \eqref{E:cutlocus} gives
that
\begin{equation*}
\limsup_{\ve \downarrow 0}\frac{\mm((X \setminus T_{\ve}(X))\cap W)}{\ve} \leq \| [\Delta \sfd^{2}_{p}]^{sing} \| (U) <\infty,
\end{equation*}
for every $U$ bounded open subset compactly containing the open set $W$. Since from Theorem \ref{T:d3} we know that $\Delta \sfd^{2}_{p}$ is a Radon measure, the thesis \eqref{E:cutlocu2} follows.
\end{proof}

We  next give some suffcient condition implying  that the densities $h_{\alpha}$, given by the Disintegration Theorem \ref{T:sigmadisint}, are null at the final points.

\begin{lemma}
Let $(X,\sfd,\mm)$ be an e.n.b. $\MCP(K,N)$ space, for some $K\in \R, N\in (1,\infty)$.
\\Let  $u = \sfd_{p} = \sfd(p,\cdot)$ for some $p \in X$ and consider the disintegration associated to $\sfd_{p}$: $\mm = \int_{Q} h_{\alpha} \cH^{1}\llcorner_{X_{\alpha}} \, \qq(d\alpha)$.
\\Assume there exists $s>1$ such that 
\begin{equation}\label{eq:mBrp}
\liminf_{r\downarrow 0} \frac{\mm(B_{r}(p))}{r^{s}} < \infty,
\end{equation}
then $h_{\alpha}(p) = 0$, for $\qq$-a.e. $\alpha \in Q$.
\\More generally, for any 1-Lipschitz function $u$, denoting $\mm\llcorner_{\T^{nb}_{u}} = \int_{Q} h_{\alpha} \cH^{1}\llcorner_{X_{\alpha}} \, \qq(d\alpha)$ the associated disintegration, it holds that
\begin{equation}\label{eq:TVha}
\left\| \int_{Q } h_{\alpha}(a(X_{\alpha})) \delta_{a(X_{\alpha})} \qq(d\alpha) \right\| \leq  \liminf_{r\downarrow 0} \frac{\mm(\cup_{\alpha} [a(X_{\alpha}), a(X_{\alpha})+r])}{r}=\beta\in [0,+\infty],
\end{equation}
where the leftmost term is the total variation of the corresponding measure.
\end{lemma}

\begin{proof}
Suppose by contradiction the claim was false, i.e. there exists $\bar Q \subset Q$ where $h_{\alpha}(p) \geq c > 0$, with $\qq(\bar Q )> 0$. 
Observe that a.e. transport ray $X_{\alpha}$ ends in $p$, i.e.  $b(X_{\alpha}) = p$ for $\qq$-a.e. $\alpha\in Q$. As usual, we identify the transport ray $X_{\alpha}$ with the real interval $[a_{\alpha}, b_{\alpha}]$.
Then by Fatou's Lemma it holds:
\begin{align*}
\liminf_{r\downarrow 0} \frac{\mm(B_{r}(p))}{r^{s}}  &\geq \liminf_{r\downarrow 0} \int_{\bar Q} \frac{1}{r^{s}}\int_{[b_{\alpha}-r,b_{\alpha}]}h_{\alpha}(t) dt \, \qq(d\alpha) \\
&\geq \int_{\bar Q} \liminf_{r\downarrow 0} \frac{1}{r} \int_{[b_{\alpha}-r,b_{\alpha}]}\frac{h_{\alpha}(t)}{r^{s-1}} dt \, \qq(d\alpha)  = \infty,
\end{align*}
giving a contradiction and proving the claim.
\\The second part of the lemma follows along analogous arguments.
\end{proof}

\begin{remark}
If $(X,\sfd,\mm)$ is an $\RCD(K,N)$ space not isometric to a circle or to a  (possibly unbounded) real interval then \eqref{eq:mBrp} is satisfied for $\mm$-a.e. $p\in X$. 
\\Indeed if $(X,\sfd,\mm)$ is an $\RCD(K,N)$ space, using the rectifiability result  \cite[Theorem 1.1]{MondinoNaber} (see also \cite{GMR2013} and compare with \cite{CC97,CC00a,CC00b}) together with the absolute continuity of the reference measure $\mm$ with the respect to the Hausdorff measure of the bi-Lipschitz charts obtained independently in \cite[Theorem 1.2]{KellMondino} and  \cite[Theorem 3.5]{GigliPasqualetto}, it follows that for $\mm$-a.e. $p\in X$ there exists $n=n(p)\in \N\cap[1,\infty)$ such that  $
\liminf_{r\downarrow 0} \frac{\mm(B_{r}(p))}{r^{n}} < \infty.$ If moreover we assume $(X,\sfd)$ not to be isometric to a circle or to a (possibly unbounded) real interval, then by \cite{KitabeppuLakzian} it follows that $n(p)>1$ for $\mm$-a.e. $p\in X$.

If $(X,\sfd,\mm)$ is an $\MCP(K,N)$ space then the validity \eqref{eq:mBrp} 
is not known. 
\end{remark}

\section{$\CD(K,N)$ is equivalent to a $(K,N)$-Bochner-type inequality}\label{Sec:CDBE}

The Bochner inequality is one of the most fundamental estimates in geometric analysis. For a smooth $N$-dimensional Riemannian manifold $(M, g)$ with Ricci$_{g}\geq Kg$, for some $K\in \R$, it states that for any smooth function $u\in C^{3}(M)$ it holds
\begin{equation}\nonumber
\frac{1}{2} \Delta |\nabla u|^{2} - \langle \nabla u, \nabla \Delta u \rangle  \geq K  |\nabla u|^{2}+ |\nabla^{2} u|^{2} \geq K  |\nabla u|^{2}+ \frac{1}{N} (\Delta u)^{2},
\end{equation}
where  $|\nabla^{2} u|^{2}$ is the Hilbert-Schmidt norm of the Hessian matrix $\nabla^{2} u$ and the rightmost inequality follows directly by Cauchy-Schwartz inequality. Note in particular that if $u$ is a distance function, then on a open dense set of full measure $|\nabla u|^{2}=1$  and the Hessian is a block matrix with vanishing slot in the direction of the ``gradient of the distance'' ; in particular, for a distance function the inequality can be improved to
\begin{equation}\label{eq:Bochner'} 
- \langle \nabla u, \nabla \Delta u \rangle \geq  K + \frac{1}{N-1} (\Delta u)^{2}, \quad \text{a.e. }.
\end{equation} 
Finally, note that the term $- \langle \nabla u, \nabla \Delta u \rangle$ corresponds to ``the derivative of  $\Delta u$ in the direction of $-\nabla u$''; thus, if we consider the transport set associated to $u$, such a term would correspond to what we denoted $(\Delta u)'$. Since in a general m.m.s.  it is not clear there is enough regularity to write $(\Delta u)'$, it is natural to consider the following version of \eqref{eq:Bochner'} ``integrated along transport rays'':
\begin{equation}\label{eq:Bochner''} 
\Delta u( g_{t}(x)) - \Delta u (x) \geq  K t + \frac{1}{N-1} \int_{(0,t)}(\Delta u)^{2}(g_{s}(x))\,ds. , \quad \text{a.e. } x,t.
\end{equation}
This is the $(K,N)$-Bochner inequality that will be proved to be equivalent to the $\CD(K,N)$ condition.

In order to state the results,  it is useful to recall that given a $1$-Lipschitz function $u$ on an e.n.b.  $\CD(K,N)$ space  there is a natural disintegration of $\mm$ restricted to the transport set $\T^{nb}_{u}$ (see Theorem \ref{T:sigmadisint}):
\begin{equation}\label{eq:halphaCDKN}
\mm\llcorner_{\T^{nb}_{u}} = \int_{Q} h_{\alpha} \, \cH^{1}\llcorner_{X_{\alpha}} \qq(d\alpha). 
\end{equation}
We will denote  $\inte(\T^{nb}_{u}) : = \cup_{\alpha\in Q} \mathring{X_{\alpha}},$ where 
$\mathring{X_{\alpha}}$ stands for the relative interior of $X_{\alpha}$; it can also be identified by isometry with the open interval 
$(a_{\alpha}, b_{\alpha})$. 
\\The function $h_{\alpha}$ in  \eqref{eq:halphaCDKN} is a $\CD(K,N)$ density on $(a_{\alpha},b_{\alpha})$, so  in particular it is semi-concave; thus if $D_{\alpha}$ is the set of differentiability points of $h_{\alpha}$, then $(a_{\alpha},b_{\alpha}) \setminus D_{\alpha}$ is countable.

Our next result roughly states that the $(K,N)$-Bochner type inequality \eqref{eq:Bochner''}  holds for  those $1$-Lipschitz functions for which we have 
found an explicit representation formula for the Laplacian, namely  those $1$-Lipschitz functions verifying  the hypothesis of Theorem \ref{T:main1} and 
for any distance function with sign  $d_{v}$. 
Recall that, for any $u$ belonging to these classes of functions, $\Delta u$ outside of the initial and final points of transport rays forming $\T^{nb}_{u}$ is absolutely continuous with respect to $\mm$.

\begin{theorem}[$\CD_{loc}(K,N)$+e.n.b. $\Rightarrow$ $(K,N)$-Bochner type inequality]\label{thm:CDToBE}
Let $(X,\sfd,\mm)$ be an e.n.b. metric measure space verifying $\CD_{loc}(K,N)$. Then the following holds:
\begin{enumerate}
\item Let $u : X \to \R$ be any $1$-Lipschitz function such that 
$\int_{Q} |X_{\alpha}|^{-1}\,\qq(d\alpha)<\infty$.
Then for $\qq$-a.e. $\alpha \in Q$, for each $x \in X_{\alpha}$  it holds
\begin{equation}\label{E:}
\Delta u( g_{t}(x)) - \Delta u (x) \geq  K t + \frac{1}{N-1} \int_{(0,t)}(\Delta u)^{2}(g_{s}(x))\,ds, 
\end{equation}
for all $t \in \R$ such that $g_{t}(x) \in \T_{u}$, up to a countable set depending only on $\alpha$.

\item Let $u=d_{v}$ be a signed distance function. Then for $\qq$-a.e. $\alpha \in Q$, for each 
$x \in \mathring{X}_{\alpha}\setminus\{ v =0 \}$   the $(K,N)$-Bochner type inequality \eqref{E:} holds 
for all $t \in \R$ such that 
$g_{t}(x) \in \mathring{X}_{\alpha}\setminus\{ v =0 \}$ and 
$\sgn(d_{v}(x)) = \sgn (d_{v}(g_{t}(x)))$, provided the densities 
$\Delta d_{v} (x)$ and $\Delta d_{v} (g_{t}(x))$ exist.
\end{enumerate}
\end{theorem}

\begin{proof}
We prove just \emph{1.}, the proof of \emph{2.} being completely analogous (using Corollary \ref{cor:Deltadv} in place of Theorem \ref{T:main1}).
\\Fix $\alpha \in Q$ and $x \in \inte(R_{u}^{nb}(\alpha)) = (a_{\alpha}, b_{\alpha})$ for which the representation of $\Delta u$ given by Theorem \ref{T:main1} is valid:
$$
\Delta u (x) = - (\log h_{\alpha})'(x).
$$
In particular $h_{\alpha}$ is differentiable at $x$. As observed above, for each $\alpha$, $\Delta u (x)$ is defined on $D_{\alpha} \subset(a_{\alpha},b_{\alpha})$, with $(a_{\alpha},b_{\alpha}) \setminus D_{\alpha}$ countable. Therefore the claim reduces to show for $\qq$-a.e. $\alpha \in Q$ that
\begin{equation}\label{E:Deltahalpha}
(\log h_{\alpha})'(x) - (\log h_{\alpha})'(g_{t}(x)) \geq 
K t + \frac{1}{N-1} \int_{(0,t)} ( (\log h_{\alpha})'(g_{s}(x)))^{2} \,ds,
\end{equation}
whenever $x,g_{t}(x) \in D_{\alpha}$. To prove \eqref{E:Deltahalpha},  consider 
a non-negative $C^2$ function $\psi$ supported on $[-1,1]$ with $\int \psi = 1$. Let   $\psi_\eps(x):=\psi(x/\eps)$; of course  $\psi_{\eps}$ is supported on $[-\eps,\eps]$ with $\int \psi_\eps = 1$.
Define the function $h^\eps_{\alpha}$ on $(a_{\alpha}+\eps,b_{\alpha}-\eps)$ by:
\begin{equation}\label{eq:defLogConv}
\log h^\eps_{\alpha} := \log h_{\alpha} \ast \psi_\eps .
\end{equation}
Since by Theorem \ref{T:sigmadisint} we know that  $h_{\alpha}$ is a $\CD(K,N)$ density, also $h^\eps_{\alpha}$ is a $C^2$-smooth $\CD(K,N)$ density on $(a_{\alpha}+\eps ,b_{\alpha} -\eps)$ by Proposition \ref{prop:log-convolve}; in particular 
\eqref{E:Deltahalpha} is satisfied by $h^{\ve}_{\alpha}$. Taking the limit as $\ve \to 0$ we obtain that $(\log h^\eps_{\alpha})' \to (\log h_{\alpha})'$ pointwise on $D_{\alpha}$ and in $L^{1}((a_{\alpha}, b_{\alpha}))$. Thus we can pass into the limit as $\ve \to 0$
in  \eqref{E:Deltahalpha} and get that it is also satisfied by $h_{\alpha}$.
\end{proof}

Also the converse implication holds, giving a complete  equivalence between the $(K,N)$-Bochner type inequality \eqref{E:} on signed distance functions and the $\CD(K,N)$ condition. 

\begin{theorem}[$\MCP(K',N')$+e.n.b.+ $(K,N)$-Bochner type inequality $\Rightarrow$ $\CD(K,N)$]\label{T:dvLaplace}
Let $(X,\sfd,\mm)$ be an e.n.b. metric measure space verifying $\MCP(K',N')$ for some $K'\in \R, N'\in (1,\infty)$, with $\mm(X) < \infty$.
Assume that, for every  signed distance function $d_{v} : X \to \R$, for $\qq$-a.e. $\alpha \in Q$, for each 
$x \in \mathring{X}_{\alpha}\setminus\{ v =0 \}$ it holds
\begin{equation}\label{E:viceversa}
\Delta d_{v}( g_{t}(x)) - \Delta d_{v} (x) \geq  K t + \frac{1}{N-1} \int_{(0,t)}(\Delta d_{v})^{2}(g_{s}(x))\,ds,  \; \text{}
\end{equation}
for all $t \in \R$ such that 
$g_{t}(x) \in \mathring{X}_{\alpha}\setminus\{ v =0 \}$ and 
$\sgn(d_{v}(x)) = \sgn (d_{v}(g_{t}(x)))$, provided the densities 
$\Delta d_{v} (x)$ and $\Delta d_{v} (g_{t}(x))$ exist.

\smallskip
Then $(X,\sfd,\mm)$ satisfies $\CD(K,N)$.
\end{theorem}

\begin{remark}
We briefly comment on the statement of Theorem \ref{T:dvLaplace}. Using the assumption of e.n.b. and of $\MCP(K',N')$, we deduce from Corollary 
\ref{cor:Deltadv} that any 
$d_{v} \in D(\bold{\Delta}, X \setminus \{v =0\})$.
Therefore, in the assumption  \eqref{E:viceversa}, we  consider  
$\Delta d_{v}(g_{t}(x))$ only for those $g_{t}(x)$ belonging 
to $\{ v >0 \}$ or to $\{ v  < 0 \}$, provided $x \in \{ v >0 \}$ or $x \in \{ v <0 \}$  respectively.

Let us also comment on the assumptions $\CD_{loc}(K,N)$ Vs $\CD(K,N)$ and $\mm(X)<\infty$ in the last two results. It was proved in \cite{CMi16} that, under the assumption $\mm(X)<\infty$, an e.n.b.  $\CD_{loc}(K,N)$ space satisfies $\CD(K,N)$ globally; on the other hand the implication is open without the assumption $\mm(X)<\infty$. We thus assumed  $\CD_{loc}(K,N)$ in Theorem \ref{thm:CDToBE} as, a priori, it is more general and still gives that all the conditional densities $h_{\alpha}$ are $\CD(K,N)$ densities (see Theorem \ref{T:sigmadisint}).
\end{remark}

\begin{proof}  
We  show that, given any $1$-Lipschitz function $\f : X \to \R$, the conditional probabilities associated to the transport set $\T^{nb}_{\f}$ of $\f$  satisfy $\CD(K,N)$. 
From \cite{CMi16} it will then follow that $(X,\sfd,\mm)$ satisfy $\CD(K,N)$. 
\smallskip

{\bf Step 1.} Let us fix $\f : X \to \R$ a $1$-Lipschitz function and the associated non-branched transport set $\T^{nb}_{\f}$. 
Fix also $c \in \R$, let $\f_{c}:=\f - c$ and consider the  associated signed distance function $d_{\f_{c}}$ from the level set $\{\f = c \}$.

Note that the function $d_{\f_{c}}$ coincides with $\f_{c}$ 
along $(R_{\f}^{nb})^{-1}(\{\f = c\})$, i.e. 
along each transport ray of $\f$ having non empty intersection with  $\{\f = c\}$. 
\\Indeed, fix any $x \in \T^{nb}_{\f_{c}}$ with $\f(x)\geq c$ (the argument for $\f(x)\leq c$ is analogous) such  that 
there exists $y \in R^{nb}_{\f}(x)$ with $\f(y) = c$ 
(i.e. $x\in (R_{\f}^{nb})^{-1}(\{\f = c\})$), then 
for any other $z \in \{\f = c\}$ it holds
$$
\sfd(x,y) = \f(x) - \f(y) = \f(x) - \f(z) \leq \sfd(x,z),
$$
showing that $\sfd(x,y) = d_{\f_{c}}(x)$ and  that 
$d_{\f_{c}}(x) = \f(x) - \f(y) = \f(x)-c= \f_{c}(x)$.
Hence if $x \in (R_{\f}^{nb})^{-1}(\{ \f = c\})$, then 
$$
|d_{\f_{c}}(x) - d_{c}(y)| = \sfd(x,y), 
$$
for some $(x,y) \in (R_{\f}^{nb})$
implying that $(x,y) \in R_{d_{\f_{c}}}$. 
Since a branching structure for $d_{\f_{c}}$ inside 
$(R_{\f}^{nb})^{-1}(\{ \f = c\})$ will imply a branching structure for $\f_{c}$, 
this implies that on $(R_{\f}^{nb})^{-1}(\{ \f = c\})$ the equivalence 
relation $R_{\f}^{nb}$ implies $R_{d_{\f_{c}}}^{nb}$.
In particular it follows that  $\T^{nb}_{\f} \cap (R_{\f}^{nb})^{-1}(\{\f = c\}) 
\subset \T^{nb}_{d_{\f_{c}}}$. Since we have shown that along 

\smallskip

{\bf Step 2.} Consider the disintegrations associated to $\T^{nb}_{\f}$ and to $\T^{nb}_{d_{\f_{c}}}$  via Theorem \ref{T:sigmadisint}:
$$
\mm\llcorner_{\T^{nb}_{\f}} = \int_{Q_{\f}} \mm_{\alpha,\f} \, \qq_{\f}(d\alpha), \qquad 
\mm\llcorner_{\T^{nb}_{d_{\f_{c}}}} = \int_{Q_{d_{\f_{c}}}} \mm_{\alpha,d_{\f_{c}}} \, \qq_{d_{\f_{c}}}(d\alpha),
$$
with $\mm_{\alpha,\f} = h_{\alpha,\f} \cH^{1}\llcorner_{X_{\alpha,\f}}$ and 
$\mm_{\alpha,d_{\f_{c}}} = h_{\alpha,d_{\f_{c}}} \cH^{1}\llcorner_{X_{\alpha,d_{\f_{c}}}}$. 
\\From  Step 1 and the uniqueness of the disintegration, it follows that up to a constant factor
$$
h_{\alpha,\f} = h_{\alpha,d_{\f_{c}}} \qquad \textrm{on}
\quad X_{\alpha,\f},
$$
for all those $\alpha$ such that $X_{\alpha,\f} \cap \{\f = c \}\neq\emptyset$. Moreover from Corollary \ref{cor:Deltadv}
we deduce that 
$$
\Delta d_{\f_{c}}\llcorner_{\mathring {X}_{\alpha,d_{\f_{c}}}\cap \{\f\neq c\}} 
= - (\log h_{\alpha,d_{\f_{c}}})' .
$$
The last two identities together with the assumption \eqref{E:viceversa} applied to $d_{\f_{c}}$  imply that for all those $\alpha$ such that $X_{\alpha,\f} \cap \{\f = c \}\neq\emptyset$, for each $ x \in \mathring {X}_{\alpha,\f}\cap \{\f\neq c\}$ it holds
\begin{equation}\label{E:CD-discrete}
- [(\log h_{\alpha,\f})'(g_{t}(x)) -
(\log h_{\alpha,\f})'(x) ] \geq K t + \frac{1}{N-1} \int_{(0,t)} [(\log h_{\alpha,\f})']^{2}(g_{s}(x)) \,ds,
\end{equation}
for all those $t$  such that $\f(g_{t}(x)) > c$ provided $\f (x ) >  c$ (and appropriate modifications if
$\f (x )< c$ ).
Identifying $\mathring X_{\alpha}$ with  the isometric real interval $(a_{\alpha},b_{\alpha})$ and denoting with $c_{\alpha}$ the unique point corresponding to $\mathring X_{\alpha} \cap \{ \f = c \}$,  \eqref{E:CD-discrete} 
becomes
\begin{equation}\label{E:CD-discrete2}
- [(\log h_{\alpha,\f})'( x +t) -
(\log h_{\alpha,\f})'(x) ] \geq K t + \frac{1}{N-1} \int_{(0,t)} [(\log h_{\alpha,\f})']^{2}(x+s) \,ds,
\end{equation}
for each $x \in (a_{\alpha},c_{\alpha})$ and $t$ such that $x + t \leq c_{\alpha}$.
We again regularise by logarithmic convolution, i.e. as in \eqref{eq:defLogConv}. In order to simplify the notation, we will omit  the subscript $\f$.
We have:
$$
(\log h^{\ve}_{\alpha})'(x) = \int (\log h_{\alpha})'(y)\psi_{\ve}(x-y) dx
$$
$$
(\log h^{\ve}_{\alpha})'(y) - (\log h^{\ve}_{\alpha})'(y+t)   = \int [(\log h_{\alpha})'(x) - (\log h_{\alpha})'(x+t)]  \psi_{\ve}(x-y)dx.
$$
Moreover
\begin{align*}
\int &~ \int_{(0,t)} ( (\log h_{\alpha})'(x+s))^{2}
\psi_{\ve}(x-y)\,dsdx  \\
&~ = \int_{(0,t)} \int ( (\log h_{\alpha})'(x+s))^{2}
\psi_{\ve}(x-y)\,dxds \\ 
&~ \geq \int_{(0,t)} \left( \int  (\log h_{\alpha})'(x+s))
\psi_{\ve}(x-y)\,dx \right)^{2}ds \\ 
&~ = \int_{(0,t)} \log (h_{\alpha}^{\ve})'(y+s)^{2}ds.
\end{align*}
Hence \eqref{E:CD-discrete2} is valid for $\log h^{\ve}_{\alpha, \f}$ for each $\ve>0$ implying (just differentiate in $t$) that $h^{\ve}_{\alpha,\f}$ is a $\CD(K,N)$ density on $(a_{\alpha},c_{\alpha})$. Letting $\ve \downarrow 0$ we obtain that $h_{\alpha, \f}$ is a  $\CD(K,N)$ density on $(a_{\alpha},c_{\alpha})$. From the arbitrariness of $c$, we conclude that  $h_{\alpha,\f}$ is a  $\CD(K,N)$ density.
Hence $(X,\sfd,\mm)$ verifies $\CD^{1}_{Lip}(K,N)$ (see \cite{CMi16} for the definition of $\CD^{1}_{Lip}(K,N)$). Then we can conclude using \cite{CMi16} that $(X,\sfd,\mm)$ satisfies $\CD(K,N)$. 
\end{proof}


\bigskip
\bigskip

\section{Splitting Theorem under $\MCP(0,N)$}\label{Sec:Split}

Before stating the main result of the section, let us introduce some notation. 
\\Given a metric space $(X,\sfd)$, a curve $\bar{\gamma}:\R\to X$ is called \emph{line} if it is an isometric immersion  i.e. 
$$
\bar \gamma : \R \to X, \qquad \sfd(\bar \gamma_{t}, \bar \gamma_{s}) = |t-s|, \text{ for all $s,t \in \R$}.
$$
To a line $\bar{\gamma}:\R\to X$ we associate  the Busemann functions
$$
\sfb^{+}(x) := \lim_{t\to + \infty} \sfd(x,\bar \gamma_{t}) - t, 
\qquad \sfb^{-}(x) := \lim_{t\to + \infty} \sfd(x,\bar \gamma_{-t}) - t.
$$
Straightforwardly from the triangle inequality, one can check that the Busemann functions are well-defined maps  $\sfb^{\pm} : X \to \R$ and
$$
|\sfb^{\pm}(x) - \sfb^{\pm}(y) | \leq \sfd(x,y).
$$
Since $\sfb^{\pm}$ are $1$-Lipschitz functions, we can consider the associated non-branching transport set $\T^{nb}_{\sfb^{\pm}}$ defined in \eqref{eq:defTub}. 

\begin{theorem}[Splitting Theorem]\label{thm:split}
Let $(X,\sfd,\mm)$ be an e.n.b.  infinitesimally Hilbertian $\MCP(0,N)$ space containing a line. Then $(X,\mm)$ is isomorphic as a measure space to a splitting $Q\times \R$.
\\More precisely the following holds. Denoting $\T^{nb}_{\sfb^{+}}=\cup_{\alpha\in Q} X_{\alpha}$ the non-branching transport set induced by $\sfb^{+}$ with the associated (disjoint) decomposition in transport rays, it holds that $\mm(X \setminus \T^{nb}_{\sfb^{+}})=0$ and  the map
\begin{equation}\label{def:Phi}
\Phi:\T^{nb}_{\sfb^{+}}\to Q\times \R, \quad x\mapsto \Phi(x):=(\alpha(x), \sfb^{+}(x))
\end{equation}
is an isomorphism of measures spaces, i.e.
\begin{itemize}
 \item $\Phi$ is a bijection,  
 \item $\Phi$ induces an isomorphism between the  $\sigma$-algebra of  $\mm$-measurable subsets of $\T^{nb}_{\sfb^{+}}$ and the $\sigma$-algebra of  $\qq\otimes \L^{1}$-measurable subsets of  $Q\times \R$,  where $\qq$ is quotient measure in the disintegration 
 $\mm \llcorner_{\T^{nb}_{\sfb^{+}}}=\int_{Q} \mm_{\alpha} \qq(d\alpha)$ given by Theorem \ref{T:sigmadisint}.
 \item  $\Phi_{\sharp} \mm\llcorner_{\T^{nb}_{\sfb^{+}}}= \qq'\otimes \L^{1}$. Here $\qq'$ is a non-negative measure over $Q$ equivalent to $\qq$, i.e. $\qq'\ll \qq$ and $\qq\ll \qq'$.
 \end{itemize}
Moreover, for every $\alpha \in Q$, the map $\sfb^{+}:X_{\alpha}\to \R$ is an isometry.

If in addition  $(X,\sfd)$ is non-branching, then $X$ is homeomorphic to a splitting $Q\times \R$. More precisely,   $X=\T_{\sfb^{+}}=\T^{nb}_{\sfb^{+}}$ and the map $\Phi:X\to Q\times \R$ defined in \eqref{def:Phi} is an homeomorphism. Here the set of rays $Q$ is induced with the compact-open topology as a subset of $C(\R, X)$, where each ray is parametrised by $(\sfb^{+})^{-1}$;   i.e.
\begin{align}
&\text{Given $\beta\in Q, \{\alpha_{n}\}_{n\in \N}\subset Q$, it holds }  \beta=\lim_{n\to \infty} \alpha_{n} \quad \text{ if and only if } \nonumber \\
& 0= \lim_{n\to \infty}  \sup_{t\in I} \;  \sfd\left(X_{\alpha_{n}} ((\sfb^{+})^{-1}(t)), X_{\beta} ((\sfb^{+})^{-1}(t)) \right), \; \text{for every compact interval }I\subset \R. \label{eq:ConvQ}
\end{align}
\end{theorem}

\begin{remark}
For smooth Riemannian manifolds \cite{ChGr}, as well as for Ricci-limits \cite{CC96} and  $\RCD(0,N)$ spaces \cite{GigliSplitting}, the Splitting Theorem has a stronger statement giving an \emph{isometric splitting}. However under the assumptions of Theorem \ref{thm:split} it is not conceivable to expect also a splitting of the metric. Indeed the Heisenberg group ${\mathbb H}^{n}$ is an example of  non-branching  infinitesimally Hilbertian $\MCP(0,N)$ space \cite{Juillet} containing a line, which is homeomorphic and isomorphic as measure space to a splitting (indeed it is homeomorphic to $\R^{n}$ and the measure is exactly the $n$-dimensional Lebesgue measure) but it is not isometric to a splitting.
\end{remark}

We start by establishing some preliminary lemmas on the properties of Busemann functions.

\begin{lemma}\label{L:moveforward}\label{lem:Tb=X}
For any proper geodesic space , $\T_{\sfb^{\pm}} = X$. 
\end{lemma}

\begin{proof}
Fix any $x \in X$ and $s > 0$. 
For each $t\in \R$, consider a unit speed geodesic 
$$
\gamma^{t} : [0, \sfd(x,\bar \gamma_{t}) ] \to X, \qquad 
\textrm{such that} \ 
\gamma_{0}^{t} = x \ \textrm{and} \ \gamma^{t}_{\sfd(x,\bar \gamma_{t})} = \bar \gamma_{t}. 
$$ 
From triangular inequality $\lim_{t\to \pm \infty} \sfd(x,\bar\gamma_{t}) = \infty$. Hence  any fixed $s > 0$, for $|t|$ sufficiently large, 
 belongs to the domain of $\gamma^{t}$.
Consider then the following trivial identities
$$
 \sfd(x,\bar \gamma_{t}) -  t - \sfd(\gamma_{s}^{t},\bar \gamma_{t}) + t =  \sfd(x,\gamma^{t}_{s}) = s>0.
$$
Taking the limit as $t \to + \infty$ and using uniform convergence, gives 
$$
\sfb^{+}(x) - \sfb^{+}(y) = \sfd(x,y)= s > 0
$$
where $y$ is any accumulation point of $\{\gamma_{s}^{t}\}_{t\geq 0}$. In particular this shows that  that each point $x \in X$ can be moved forwardly with respect to $\sfb^{+}$ (into $y$) proving in particular that $x \in \T_{\sfb^{+}}$.
The  proof for $\sfb^{-}$ can be achieved along the same lines.
\end{proof}

The proof of Lemma \ref{L:moveforward} also proves the following Corollary.

\begin{corollary}\label{C:finalpoints}
Let $(X,\sfd)$ be proper and geodesic. 
Then $\mathfrak{b}_{\sfb^{\pm}} = \emptyset$, i.e. the set of final points associated to $\sfb^{+}$ and to $\sfb^{-}$ are both empty.
\end{corollary}

Applying results from Part \ref{part1} we easily obtain the following 
result.

\begin{proposition}\label{prop:Deltableq0}
Let $(X,\sfd,\mm)$ be an e.n.b. metric measure space verifying $\MCP(0,N)$ containing a line. Then $\sfb^{\pm}\in D({\bf \Delta},X)$ and there exists a Radon measure $\Delta \sfb^{\pm}\in {\bf \Delta} \sfb^{\pm}$ satisfying
\begin{equation}\label{eq:Deltableq0}
\Delta \sfb^{\pm} \leq 0.
\end{equation}
\end{proposition}

\begin{proof}
We only prove the claim for $\sfb^{+}$, the proof for $\sfb^{-}$ being analogous.  
First of all from Theorem \ref{T:sigmadisint}, we have the disintegration
$$
\mm = \int_{Q} h_{\alpha} \cH^{1}\llcorner_{X_{\alpha}}\, \qq(d\alpha). 
$$
Thanks to Corollary \ref{C:finalpoints} we deduce that each ray $X_{\alpha}$ is isomorphic to a right half line (or to a full line), in particular it has infinite length.
\\The combination of  Theorem \ref{T:main1} with Lemma \ref{L:moveforward} thus gives that  $\sfb^{+}\in D({\bf \Delta},X)$ and that 
$$
\Delta \sfb^{+} := - \int_{Q} h_{\alpha}' \H^{1}\llcorner_{X_{\alpha}}\, \qq(d\alpha) 
- \int_{Q} h_{\alpha} \delta_{a(X_{\alpha}) \cap U }\,\qq(d\alpha)
$$
defines  a Radon measure $\Delta \sfb^{+}\in {\bf \Delta} \sfb^{+}$. We are left to show that  $\Delta \sfb^{+}\leq 0$.
\\As above, we identify $X_{\alpha}$ with the right half line $[a_{\alpha},\infty)$ endowed with the  $\MCP(0,N)$ density $h_{\alpha}$.
Using \eqref{E:MCPdef2}, we deduce that  for $a_{\alpha}<x_{0}\leq x_{1}< b<\infty$  it holds
$$
\left( \frac{b - x_{1}}{b -x_{0}} \right)^{N-1} 
\leq \frac{h_{\alpha}(x_{1} ) }{h_{\alpha} (x_{0})}.
$$
Letting $b \to \infty$, it follows that $h_{\alpha}(x_{0}) \leq h_{\alpha}(x_{1})$, showing that $h_{\alpha}' \geq 0$ whenever $h_{\alpha}'$ exists.
\\Thus $\Delta \sfb^{+}\leq 0$ and the proposition follows.
\end{proof}

Observe also that,  by triangle inequality, one has
$$
\sfd(x,\bar \gamma_{t}) - t + \sfd(x,\bar \gamma_{-s}) - s \geq 0.
$$
Setting $\sfb:=\sfb^{+}+\sfb^{-}$ and letting  $t,s\to \infty$, it gives
\begin{equation}\label{eq:bgeq0}
\sfb\geq 0 \text{ on $X$}, \quad \text{and} \quad \sfb\equiv 0  \text{ on $\bar{\gamma}$}.
\end{equation}

From now on we assume $(X,\sfd,\mm)$ to be infinitesimally Hilbertian, which is equivalent to assume that the Laplacian $\Delta$ is single valued (on its domain) and linear.
Proposition \ref{prop:Deltableq0}  then implies  
\begin{equation}\label{eq:bSH}
\sfb:=\sfb^{+}+\sfb^{-}\in  D({\bf \Delta}, X), \quad \Delta \sfb\leq 0.
\end{equation}

It is worth noting that \eqref{eq:bSH} will be the only implication of the paper where infinitesimal Hilbertianity plays a role.
We now want to combine \eqref{eq:bSH} and \eqref{eq:bgeq0} with strong maximum principle in order to infer that $\sfb\equiv 0$. 
The next statement was proved in \cite[Theorem 9.13]{BB}  (actually we report a slightly weaker statement which will suffice to our scopes). 

\begin{theorem}[Strong Maximum Principle]\label{thm:SMP}
Let $(X,\sfd,\mm)$ be a metric measure space supporting a local weak $(1,2)$-Poincar\'e inequality with $\mm$ locally doubling. Let $u\in \LIP(X)$ and  $\Omega\subset X$ be a connected bounded open subset.
\\If  $u$ attains its maximum in an interior point of $\Omega$ and 
\begin{equation}\label{eq:uSHM}
\int_{\Omega} |\nabla u|^{2} \, \mm \leq  \int_{\Omega} |\nabla (u+f)|^{2} \, \mm, \quad \forall f\in \LIP(X), \supp (f)\subset \Omega, \; f\leq 0,
\end{equation}
then $u$ is constant on $\Omega$.
 \end{theorem}

Let us discuss the validity of the strong maximum principle in our setting.  Clearly, from Bishop-Gromov inequality it follows that a $\MCP(0,N)$ space is doubling. 
Moreover, essentially non-branching  $\MCP(0,N)$ spaces satisfy  a local weak $(1,1)$-Poincar\'e inequality \cite{VRN} (\cite{VRN} assumes negligibility of cut-locus from $\mm$-a.e. point that is satisfied whenever the space is 
essentially non-branching, see Remark  \ref{rem:MCPDoubPoinc1}),  which in turns implies that the space  supports a local weak $(1,2)$-Poincar\'e inequality.
In conclusion if $(X,\sfd,\mm)$ is an essentially non-branching  $\MCP(0,N)$ space, then the strong maximum principle holds.
The simple link between \eqref{eq:uSHM} and the measure-valued Laplacian  was established in \cite[Theorem 4.3]{GiMo}; for completeness, below we report the argument together with the desired conclusion  $\sfb\equiv 0$.

\begin{lemma}\label{lem:b=0}
Let $(X,\sfd,\mm)$ be an infinitesimally Hilbertian, essentially non-branching,   metric measure space satisfying $\MCP(0,N)$. Assume $(X,\sfd)$ contains a line and let $\sfb:=\sfb^{+}+\sfb^{-}$.
\\Then $\sfb \equiv 0$ on $X$.
\end{lemma}

\begin{proof}
It is enough to prove that \eqref{eq:bSH} implies \eqref{eq:uSHM} for $u:=-\sfb$, then the claim will follow by the combination of \eqref{eq:bgeq0} with Theorem \ref{thm:SMP}.
\\ Let $\Omega\subset X$ be a connected bounded open subset and $f\in \LIP(X)$ be non-positive with $\supp (f)\subset \Omega$. Since the map $\ve\mapsto \int_{\Omega} |\nabla(-\sfb+\ve f)|^{2} \mm$ is convex and  $\Delta \sfb\leq 0$, we have
\begin{align*}
 \int_{\Omega} |\nabla(-\sfb+ f)|^{2} \mm -  \int_{\Omega} |\nabla(-\sfb)|^{2} \mm&\geq \lim_{\ve\downarrow 0} \int_{\Omega} \frac{|\nabla(-\sfb+ \ve f)|^{2}- |\nabla(-\sfb)|^{2}}{\ve} \mm \nonumber \\
 &=- 2 \int_{\Omega} \langle \nabla \sfb, \nabla f \rangle \mm = 2  \int_{\Omega} f\, \Delta \sfb  \geq 0,
\end{align*}
proving \eqref{eq:uSHM} for $u:=-\sfb$.
\end{proof}

\begin{lemma}\label{lem:Xa=R}
Let $(X,\sfd,\mm)$ be an infinitesimally Hilbertian, essentially non-branching,   metric measure space satisfying $\MCP(0,N)$. Assume $(X,\sfd)$ contains a line.
Let $\T^{nb}_{\sfb^{+}}=\cup_{\alpha \in Q} X_{\alpha}$ be the ray decomposition of the non-branching transport set $\T^{nb}_{\sfb^{+}}$ associated to $\sfb^{+}$.

Then for each $\alpha \in Q$, the ray $X_{\alpha}$ is isometric to $\R$; in other words $a(X_{\alpha})=\emptyset=b(X_{\alpha})$.
\end{lemma}

\begin{proof}
From Lemma \ref{lem:b=0} we know that $\sfb^{+}=-\sfb^{-}$ on all $X$. It follows that 
$$
\{(x,y)\in R^{nb}_{\sfb^{+}}\}=\{(y,x)\in R^{nb}_{\sfb^{-}}\}.
$$
Thus $\T^{nb}_{\sfb^{+}}=\T^{nb}_{\sfb^{-}}$ with the same ray decomposition (from the support sense); clearly, on each ray,  the orientation induced by $\sfb^{+}$ is the opposite from the one induced by $\sfb^{-}$. In particular, the set of initial points for $\sfb^{+}$ coincides with the set of final points for $\sfb^{-}$:
$$
\mathfrak{a}_{\sfb^{+}}:=
\{x\in \T^{nb}_{\sfb^{+}} \, :\, (y,x)\in R^{nb}_{\sfb^{+}} \Leftrightarrow y=x\}=\{x\in \T^{nb}_{\sfb^{-}} \, :\, (x,y)\in R^{nb}_{\sfb^{-}} \Leftrightarrow x=y\}=:\mathfrak{b}_{\sfb^{-}}.
$$
Since from Corollary \ref{C:finalpoints} the set of final points for $\sfb^{-}$ is empty, i.e. $\mathfrak{b}_{\sfb^{-}}=\emptyset$, it follows that the both the sets of initial and final points for $\sfb^{+}$ are empty; in other words, each ray $X_{\alpha}$ is isometric to $\R$.
\end{proof}
\bigskip

\textbf{Proof of the Splitting Theorem \ref{thm:split}.}
\\By combining the lemmas above we can quickly get the first part of Theorem  \ref{thm:split}.
 Indeed, from Lemma \ref{lem:Tb=X} we already know that $X=\T_{\sfb^{+}}$ and, from Lemma \ref{L:nullsetMCP} we know that  $\mm(\T_{\sfb^{+}}\setminus \T^{nb}_{\sfb^{+}})=0$; thus the claim  $\mm(X\setminus \T^{nb}_{\sfb^{+}})=0$ is proved.
\\Moreover, Theorem \ref{T:sigmadisint} ensures that  there exists a disintegration of $\mm$  verifying 
$$
\mm\llcorner_{\T_{\sfb^{+}}^{nb}} = \int_{Q} \mm_{\alpha} \, \qq(d\alpha), \qquad \qq(Q) = 1,
$$
where, for $\qq$-a.e. $\alpha$, $\mm_{\alpha}$ is a Radon measure $\mm_{\alpha}\ll \cH^{1}\llcorner_{X_{\alpha}}$ 
and $(X_{\alpha},\sfd,\mm_{\alpha})$ verifies $\MCP(0,N)$.
\\From Lemma \ref{lem:Xa=R} we know that $(X_{\alpha}, \sfd)$ is isometric to the real line (note that the isometry is simply $\sfb^{+}:X_{\alpha}\to \R$), and thus Lemma \ref{L:rigiditMCP} implies that $\mm_{\alpha}=c_{\alpha} \cH^{1}\llcorner_{X_{\alpha}}$ for some constant $c_{\alpha}>0$, for $\qq$-a.e. $\alpha\in Q$.
\\Define the measure $\qq'$ on $Q$ as
$$
\qq'(B)=\int_{B} c_{\alpha} \qq(d\alpha), \quad \text{for any $\qq$-measurable subset  $B\subset Q$.} 
$$
It is clear that $\qq'\ll \qq$ and that $\qq\ll \qq'$, i.e. they are equivalent measures, and that
$$
\mm\llcorner_{\T_{\sfb^{+}}^{nb}} = \int_{Q} \cH^{1}\llcorner_{X_{\alpha}}  \, \qq'(d\alpha).
$$
The last disintegration formula is equivalent to claiming that the map 
\begin{equation*}
\Phi:\T^{nb}_{\sfb^{+}}\to Q\times \R, \quad x\mapsto \Phi(x):=(\alpha(x), \sfb^{+}(x))
\end{equation*}
is an isomorphism of measures spaces, i.e.  $\Phi$ induces an isomorphism between the  $\sigma$-algebra of  $\mm$-measurable subsets of $\T^{nb}_{\sfb^{+}}$ and the $\sigma$-algebra of  $\qq\otimes \L^{1}$-measurable subsets of  $Q\times \R$,  and 
$\Phi_{\sharp} \mm\llcorner_{\T^{nb}_{\sfb^{+}}}= \qq'\otimes \L^{1}$. It is also clear that  $\Phi:\T^{nb}_{\sfb^{+}}\to Q\times \R$ is bijective,  as $\T^{nb}_{\sfb^{+}}=\cup_{\alpha\in Q} X_{\alpha}$ is a partition, and  $\sfb^{+}:X_{\alpha}\to \R$ is an isometry for every $\alpha\in Q$.
\bigskip

\textbf{Proof of the second part of Theorem \ref{thm:split}.}
\\From the very definition \eqref{eq:defTub} of the non-branched transport set $\T^{nb}_{\sfb^{+}}$, if $(X,\sfd)$ is non-branching then $\T^{nb}_{\sfb^{+}}=\T_{\sfb^{+}}$. Thus, Lemma \ref{lem:Tb=X} gives $X=\T_{\sfb^{+}}=\T^{nb}_{\sfb^{+}}$. 
\\From the first part, we already know that $\Phi:X\to Q\times \R$ is bijective. Since convergence in $Q$ (see \eqref{eq:ConvQ}) is equivalent to the local uniform convergence of the rays, it is  clear that $\Phi^{-1}$ is continuous.
\\It is then enough to show that $\Phi$ is continuous. We argue by contradiction. Assume that there exist a sequence $\{x_{n}\}_{n\in \N}\subset X$ with $x_{n}\to x$ in $X$ such that $\{(\alpha(x_{n}), \sfb^{+}(x_{n}))\}_{n\in \N}$ does not converge to $(\alpha(x), \sfb^{+}(x))$. 
Since $\sfb^{+}:X\to \R$ is continuous (actually it is even 1-Lipschitz), it is clear that $\sfb^{+}(x_{n})\to \sfb^{+}(x)$ and thus it must be that $\{\alpha(x_{n})\}_{n\in \N}$ does not converge to $\alpha(x)$. 
By the definition \eqref{eq:ConvQ} of convergence in $Q$, it follows that, up to subsequences, it holds
\begin{equation}\label{eq:contrConvE}
 0< \ve= \lim_{n\to \infty}  \sup_{t\in I}  \sfd\left(X_{\alpha(x_{n})} ((\sfb^{+})^{-1}(t)), X_{\alpha(x)} ((\sfb^{+})^{-1}(t)) \right), \; \text{for some compact interval }I\subset \R. 
\end{equation}
As already observed, $\sfb^{+}:X_{\beta}\to \R$ is an isometry for every $\beta\in Q$ and thus it can be used to parametrise each ray; in the formula above as well as in the  following we fix such  a parametrisation. 
\\Since by assumption $x_{n}\to x$, for every closed interval $I\subset \R$ containing $\sfb^{+}(x)$, it is clear that the union of the images of the rays $X_{\alpha(x_{n})}$ restricted to $I$ are all contained in a compact subset of $X$. Thus, by Arzel\'a-Ascoli Theorem, such restrictions converge uniformly to a geodesic $\gamma$ of $X$ passing through $x$. By a standard diagonal argument, $\gamma$ can be extended to a geodesic defined on the whole $\R$ and
\begin{equation}\label{eq:Xalphagamma}
X_{\alpha(x_{n})}\to \gamma \text{ uniformly on compact intervals}.
\end{equation}
  Recalling that the relation $R_{\sfb^{+}}$ is closed (see \eqref{E:Gamma} and  \eqref{E:R}) we get that $\gamma$ is a ray passing through $x$, i.e. $\gamma=X_{\beta}$ for some $\beta\in Q$. Since the rays are pairwise disjoint, it follows that $\beta=\alpha(x)$.
\\Therefore \eqref{eq:Xalphagamma} contradicts \eqref{eq:contrConvE}.
\hfill$\Box$

\end{document}